\documentclass[11pt,twoside]{article}
\usepackage[textwidth=6in,textheight=8.2in,centering]{geometry}
\usepackage{fancyhdr}
\usepackage{tikz}
\usepackage{pgfplots}
\usepackage{amsmath}
\usepackage{amssymb}
\usepackage{amsthm}
\usepackage{mathtools}
\usepackage{setspace}
\usepackage{mathrsfs}
\usepackage{graphicx}
\usepackage{subcaption}
\usepackage{accents}

\usepackage{hyperref}
\usepackage[english]{babel}
\usepackage{enumitem}
\usepackage{float}
\usepackage{comment}
\usepackage{hyperref}
\hypersetup{
    colorlinks=true,
    linkcolor=blue}
\usetikzlibrary{arrows,automata}

\DeclareMathOperator*{\esssup}{ess\,sup}

\DeclareMathOperator*{\essinf}{ess\,inf}
\DeclareMathOperator*{\leb}{\mathrm{Leb}}
\newcommand{\dleb}{d\mathrm{Leb}}

\newtheorem{theorem}{Theorem}[section] 

\theoremstyle{definition}
\newtheorem{remark}[theorem]{Remark}
\newtheorem{definition}[theorem]{Definition}

\newtheorem{example}[theorem]{Example}

\theoremstyle{plain}
\newtheorem{lemma}[theorem]{Lemma} 

\newtheorem{corollary}[theorem]{Corollary}
\newtheorem{proposition}[theorem]{Proposition}
\newtheorem{thmx}{Theorem}
\newtheorem{step}{Step}
\newtheorem{stp}{Step}

\newcommand{\textoverline}[1]{$\overline{\mbox{#1}}$}
\usepackage{amsmath, amssymb}

\newcommand {\Sec}[1] {Section~\ref{#1}}

\newcommand {\Step}[1] {Step~\ref{#1}} 
\newcommand {\ex}[1] {Example~\ref{#1}}
   
\newcommand {\fig}[1] {Figure~\ref{#1}}

\newcommand {\eqn}[1] {(\ref{#1})} 
\newcommand {\thrm}[1] {Theorem~\ref{#1}} 
\newcommand {\cor}[1] {Corollary~\ref{#1}} 
\newcommand {\prop}[1] {Proposition~\ref{#1}} 
\newcommand {\dfn}[1] {Definition~\ref{#1}} 
\newcommand {\rem}[1] {Remark~\ref{#1}} 
\newcommand {\lem}[1] {Lemma~\ref{#1}}

\newcommand{\jp}[1]{\textcolor{black}{#1}}
\newcommand{\jps}[1]{\textcolor{black}{#1}}

\usepackage{dsfont}
\DeclareMathOperator*{\BV}{BV}
\newcommand{\norm}[1]{\left\lVert #1 \right\rVert}
\DeclareMathOperator*{\Var}{var}
\DeclareMathOperator*{\var}{var}
\newcommand\restr[2]{{
  \left.\kern-\nulldelimiterspace 
  #1 
  \vphantom{\big|} 
  \right|_{#2} 
  }}

\pgfplotsset{compat=1.18}
\begin{document}

\title{\textbf{AVERAGING FOR RANDOM METASTABLE SYSTEMS}}

\author{Cecilia Gonz\'alez-Tokman\thanks{School of Mathematics and Physics, University of Queensland, St Lucia QLD 4072, Australia. \\
\texttt{cecilia.gt@uq.edu.au}}, Joshua Peters \thanks{School of Mathematics and Physics, University of Queensland, St Lucia QLD 4072, Australia. \\
\texttt{joshua.peters@uq.net.au} \jps{(corresponding author)}. }}

\maketitle

\begin{abstract}\noindent
Random metastability occurs when an externally forced or noisy system possesses more than one state of apparent equilibrium. This work investigates a class of random dynamical systems, arising from perturbing a one-dimensional piecewise smooth expanding map of the interval with two invariant subintervals, each supporting a unique ergodic absolutely continuous invariant measure. Upon perturbation, this invariance is destroyed, allowing trajectories to randomly switch between subintervals. We show that the invariant density of the randomly perturbed system may be approximated by an explicit convex combination of the two initially invariant densities, obtained by averaging. Further, we also identify the limit of the second Oseledets space, or coherent structure, as the perturbation shrinks to zero. Our results are applied to random paired tent maps over ergodic, measure-preserving, and invertible driving systems. Finally, we provide generalisations to systems admitting more than two initially invariant sets.       
\end{abstract}

\newpage
\begingroup
\hypersetup{linkcolor=[rgb]{0.0,0.0,0.0}}
\tableofcontents
\endgroup
\newpage
\noindent

\section{Introduction}	
\allowdisplaybreaks
Metastability characterises systems possessing more than one state of apparent equilibrium. Such systems appear in numerous examples of natural phenomena. In molecular dynamics, transitions between conformations are rare events allowing one to describe the resulting macroscopic dynamical behaviour through a flipping process of metastable states \cite{conformations,conformations2}. Concerning oceanic flows, metastable systems have been used to study slow mixing regions of the ocean (called gyres), which have contributed to phenomena such as the \textit{Great Pacific Garbage Patch} \cite{FSN,FSS,FSM,FPE,DFH}. Random metastable systems arise when such transition patterns or ocean currents are subjected to external forces (e.g. changes in chemical potentials or wind patterns, respectively).  
\\
\\
The first description of metastability can be traced back to the work of van't Hoff \cite{H_Edynamique} in the study of chemical reaction-rate theory. Concerning particle systems, Kramers developed a model for chemical reactions in a double-well potential, based on Brownian motion \cite{K_Brownian}. In this direction, Lebowitz and Penrose proposed a rigorous theory of metastability in \cite{PL_Waals}, where the authors investigated the escape rate and lifetime of metastable states. Building on this work, Sewell provided the first axiomatisation of metastabilty in the context of quantum mechanics \cite{S_quantum}. For dynamical systems, Freidlin and Wentzell introduced the concept of large deviations on the path space to analyse the long-term behaviour of dynamical systems influenced by random perturbations \cite{FW_RPDS}. This approach aims to identify the most likely path between metastable states, and was later adapted by Cassandro, Galves, Olivieri and Vares to study interacting particle systems \cite{CGOV_metastable}. This resulted in a variety of results related to Markovian lattice models \cite{OV_Largedev}. Davies developed spectral techniques in \cite{D_Stab,D_icing,D_MarkovI,D_MarkovII,D_spec} to study metastability of Markov processes. In this case, the spectrum of generators for reversible Markov processes were investigated. These results were further developed by Gaveau and Schulman \cite{BS_nonequ}; and Gaveau and Moreau \cite{GM_meta}. Closely related to spectral methods is the potential-theoretic approach to metastability, initiated by Bovier, Eckhoff, Gayrard and Klein \cite{BEGK_Meta}. This approach focuses on sequences of visits to metastable sets, aiming to analyse hitting times and return probabilities. Finally is the martingale approach to metastability introduced by Beltr\'{a}n and Landim \cite{markov1,markov2}. For further details on the history of metastability, we refer the reader to \cite{BH_Metastability,Landim_mon,OV_Largedev} and the references therein.
\\
\\
In the context of deterministic systems, Keller and Liverani pioneered in \cite{KL_escape} the study of metastability through a dynamical systems approach. From this perspective, they provide a precise definition of metastability using the spectrum of the Perron-Frobenius operator. If the dynamical system admits a unique invariant density, then metastability is characterised entirely through the second eigenvalue of the Perron-Frobenius operator. Namely, its second eigenvalue is close to but strictly less than $1$, giving rise to a dynamical system possessing slow mixing properties equipped with almost invariant (metastable) states. Such systems may emerge by adding a small perturbation to a system with two ergodic invariant measures supported on two invariant subintervals. The perturbation is made in such a way that a \textit{hole} appears, allowing the initially invariant sets to communicate, making the system ergodic on the entire space.  
\\
\\
In the setting of piecewise expanding metastable systems, \cite{GTHW_metastable} provides a rigorous approximation for the leading and second eigenfunctions of the Perron-Frobenius operator for small hole sizes. A similar asymptotic result for the leading eigenfunction is obtained in the setting of intermittent metastable systems in \cite{wael_intermittent}. Continuing in the direction of piecewise expanding metastable systems, \cite{SD} obtains an approximation for the diffusion coefficient (or variance) of the Central Limit Theorem, admitted by \cite{Liv_CLTDet}. In particular, the authors reveal that for small hole sizes, the invariant measure and diffusion coefficient for the infinite-dimensional system may be approximated by the map's induced finite state Markov chain. Many other properties of deterministic metastable systems have been investigated including escape rates \cite{gibbs}, extreme value laws \cite{hitting}, and their relationship with open systems \cite{open}. 
\\
\\
For random dynamical systems, metastability is characterised in a slightly different manner. As opposed to referring to the spectral picture of the Perron-Frobenius operator, we look to its Oseledets decomposition, discovered by Oseledets in \cite{Oseledets} and later developed by Froyland, Lloyd, and Quas in \cite{FLQ_coherent,FLQ_semi} for Perron-Frobenius operator cocycles. Here, instead of studying eigenvalues with corresponding eigenfunctions, we look to Lyapunov exponents with corresponding Oseledets spaces. In this setting, the top Lyapunov exponent is zero with an associated Oseledets space spanned by the random invariant density of the system. Metastability is then characterised by the second Lyapunov exponent being negative but close to zero. The corresponding second Oseledets space provides a so-called coherent structure which decays asymptotically at a slow rate according to the second Lyapunov exponent.
\\
\\
In the direction of random metastable systems, \cite{BS_rand} provides a generalisation of the celebrated Keller-Liverani escape rate formulae \cite{KL_escape}. They illustrate that their techniques can be used to approximate annealed invariant densities of metastable systems subject to i.i.d. random perturbations. Such invariant densities are fixed points of a averaged Perron-Frobenius operator. In \cite{GS_entropy}, Froyland and Stancevic study the connection between Perron-Frobenius operator cocycles of metastable systems and escape rates of random maps. Since random metastable systems are characterised by an asymptotic quantity, namely the second Lyapunov exponent, constructing examples of such systems proves to be a difficult task. In fact, \cite{GTQ_stab} demonstrates that small perturbations of some random systems can result in drastic changes in the Lyapunov exponents of the Perron-Frobenius operator cocycle. More recently, focus has been placed on the approximation of the second Lyapunov exponent for piecewise expanding maps. In \cite{Crimmins}, Crimmins provides general conditions for quasi-compact Perron-Frobenius operator cocycles to admit an Oseledets decomposition that depends continuously on perturbations. This result generalises the famous Keller-Liverani perturbation theory \cite{stab_spec}, to the random setting. In the class of random metastable systems, in \cite{horan_lin} and \cite{Horan}, Horan provides an upper estimate on the second Lyapunov exponent for so-called random paired tent maps. This estimate is refined in \cite{GTQ_quarantine} where it is shown that the top two Lyapunov exponents are both simple, and the only exceptional exponents outside a readily computed range.
\subsection{Statement of main results}
In this paper, we consider a Perron-Frobenius operator cocycle of random paired metastable systems driven by an ergodic, invertible and measure-preserving transformation $\sigma$ of the probability space $(\Omega,\mathcal{F},\mathbb{P})$. The initial system $T^0$ is deterministic and admits two unique ergodic absolutely continuous invariant measures ($\mu_L$ and $\mu_R$) with associated invariant densities ($\phi_L$ and $\phi_R$) supported on two invariant subintervals ($I_L$ and $I_R$). We consider small perturbations of $T^0$, denoted $T_\omega^\varepsilon$, which introduce random holes into our system (defined as $H_{\star,\omega}^\varepsilon:=(T_\omega^\varepsilon)^{-1}(I_\star^c)\cap I_\star$ for $\star\in\{L,R\}$ and $\omega\in\Omega$), allowing trajectories to randomly switch between the initially invariant subintervals (see \fig{fig:ptm}), giving rise to a unique ergodic absolutely continuous invariant measure on $I_L\cup I_R$. The perturbations are made so that $\mu_\star(H_{\star,\omega}^\varepsilon) = \varepsilon\beta_{\star,\omega}+o_{\varepsilon\to 0}(\varepsilon)$ where $\beta_{\star}\in L^{\infty}(\mathbb{P})$. Juxtaposed to \cite{BS_rand}, instead of obtaining annealed (or averaged) results, we strive for quenched (or fibre-wise) descriptions of the functions spanning the top and second Oseledets spaces. This approach is advantageous as fluctuations in the functions spanning these spaces are not featured in the annealed setting. 
\\
\\
Our first main result shows that the limiting random invariant density is a non-random convex combination of $\phi_L$ and $\phi_R$, with weights depending only on perturbations through averaged quantities. The statement of \thrm{thrm:phi_lims} follows. The class of random dynamical systems for which our results apply are described in \Sec{sec:map+pert}. 
\begin{thmx} Let $\{(\Omega,\mathcal{F},\mathbb{P},\sigma,\BV([-1,1]),\mathcal{L}^\varepsilon)\}_{\varepsilon\geq 0}$ be a \jp{family} of random dynamical systems of paired metastable maps $T_\omega^\varepsilon:[-1,1]\to [-1,1]$ satisfying \hyperref[list:I1]{\textbf{(I1)}}-\hyperref[list:I6]{\textbf{(I6)}} and \hyperref[list:P1]{\textbf{(P1)}}-\hyperref[list:P7]{\textbf{(P7)}} (see \Sec{sec:map+pert}). Denote by $(\phi_\omega^\varepsilon)_{\omega\in\Omega}$ the random invariant density of $(T_\omega^\varepsilon)_{\omega\in\Omega}$. If $\int_{\Omega}\beta_{L,\omega} +\beta_{R,\omega}\, d\mathbb{P}(\omega)\neq 0$, then as $\varepsilon\to 0$,
    $$\phi_\omega^\varepsilon \stackrel{L^1}{\to} \phi_{}^0:= \frac{\int_{\Omega} \beta_{R,\omega}\, d\mathbb{P}(\omega)}{\int_{\Omega}\beta_{L,\omega}+\beta_{R,\omega} \, d\mathbb{P}(\omega)} \phi_L +\frac{\int_{\Omega} \beta_{L,\omega}\, d\mathbb{P}(\omega)}{\int_{\Omega}\beta_{L,\omega}+\beta_{R,\omega} \, d\mathbb{P}(\omega)}\phi_R$$
    uniformly over $\omega\in\Omega$ away from a $\mathbb{P}$-null set. 
    \label{thrm:A}
\end{thmx}
\noindent \thrm{thrm:A} also provides us with an approximation of the limiting random invariant measures for the Markov chains in random environments driven by an ergodic, invertible and measure-preserving transformation $\sigma:\Omega\to \Omega$ with transition matrices $(P_\omega^\varepsilon)_{\omega\in\Omega}$, where
$$P_\omega^\varepsilon := \begin{pmatrix} 1- \varepsilon \beta_{L,\omega} & \varepsilon\beta_{R,\omega} \\ \varepsilon \beta_{L,\omega} & 1- \varepsilon \beta_{R,\omega} \end{pmatrix}.$$
For fixed $\omega\in\Omega$, entries of $P_\omega^\varepsilon$ describe the transition probabilities between two states $S_L$ and $S_R$ as indicated in \fig{fig:mc}.
\begin{figure}[H]
    \centering
    \begin{tikzpicture}[scale = 6, node distance=4cm,->,>=latex,auto,
  every edge/.append style={thick}]
  \node[state] (1) {$S_L$};
  \node[state] (2) [right of=1] {$S_R$};  
  \path (1) edge[loop left]  node{$1-\varepsilon\beta_{L,\omega}$} (1)
            edge[bend left]  node{$\varepsilon\beta_{L,\omega}$}   (2)
        (2) edge[loop right] node{$1-\varepsilon\beta_{R,\omega}$}  (2)
            edge[bend left] node{$\varepsilon\beta_{R,\omega}$}     (1);
\end{tikzpicture}
    \caption{Transition probabilities between $S_L$ and $S_R$ for a fixed $\omega\in\Omega$.}
    \label{fig:mc}
\end{figure}
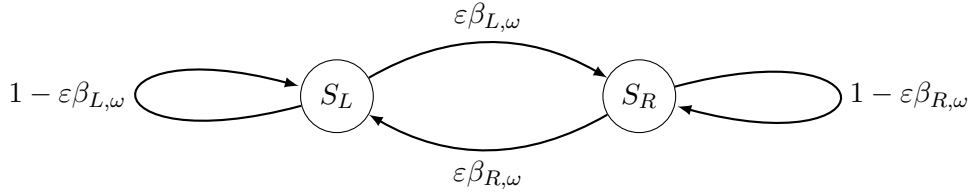\noindent
When $\varepsilon=0$, $P_\omega^0:= I$ and the states $S_L$ and $S_R$ are invariant, giving rise to two unique ergodic invariant probability measures $\mu_L = \delta_L$ and $\mu_R=\delta_R$.\footnote{We note that $\delta_L = \begin{pmatrix} 1 \\ 0\end{pmatrix}$ and $\delta_R = \begin{pmatrix} 0 \\ 1\end{pmatrix}$.} We are interested in the behaviour of these Markov chains as $\varepsilon\to 0$. Indeed, if for $\star\in\{L,R\}$, $\beta_{\star}\in L^\infty(\mathbb{P})$, and $\int_{\Omega}\beta_{L,\omega} +\beta_{R,\omega}\, d\mathbb{P}(\omega)\neq 0$, then a consequence of \thrm{thrm:A} is that as $\varepsilon\to 0$, the random invariant measure of the matrix cocycle $(P_\omega^\varepsilon)_{\omega\in\Omega}$ converges to
$$\mu^0:=\begin{pmatrix}
    \frac{\int_{\Omega} \beta_{R,\omega}\, d\mathbb{P}(\omega)}{\int_{\Omega}\beta_{L,\omega}+\beta_{R,\omega} \, d\mathbb{P}(\omega)}\\
    \frac{\int_{\Omega} \beta_{L,\omega}\, d\mathbb{P}(\omega)}{\int_{\Omega}\beta_{L,\omega}+\beta_{R,\omega} \, d\mathbb{P}(\omega)}
\end{pmatrix}$$
uniformly over $\omega\in\Omega$ away from a $\mathbb{P}$-null set. We refer the reader to \rem{rem:markov-connection} where we further address the connection between the infinite dimensional system and its induced Markov chain.\\
\\
In addition to \thrm{thrm:phi_lims}, we provide a quenched description of the function spanning the second Oseledets space as $\varepsilon\to 0$. Here, \jp{through \lem{lem:cont}, we obtain that the second Oseledets space is one-dimensional}, and show that a similar approximation proposed in \cite{GTHW_metastable} can be made in the random setting. The statement of \thrm{thrm:cont2nd} follows. The proof of this result may be found in \Sec{sec:2nd_space}. 
\begin{thmx} Let $\{(\Omega,\mathcal{F},\mathbb{P},\sigma,\BV([-1,1]),\mathcal{L}^\varepsilon)\}_{\varepsilon\geq 0}$ be a \jp{family} of random dynamical systems of paired metastable maps $T_\omega^\varepsilon:[-1,1]\to [-1,1]$ satisfying \hyperref[list:I1]{\textbf{(I1)}}-\hyperref[list:I6]{\textbf{(I6)}} and \hyperref[list:P1]{\textbf{(P1)}}-\hyperref[list:P7]{\textbf{(P7)}} (see \Sec{sec:map+pert}). Denote by $(\psi_\omega^\varepsilon)_{\omega\in\Omega}$ the family of functions spanning the second Oseledets spaces of $(\mathcal{L}_\omega^\varepsilon)_{\omega\in\Omega}$. Choose the sign of $\psi_\omega^\varepsilon$ such that $\int_{I_L}\psi_\omega^\varepsilon\, \dleb(x)> 0$ for $\mathbb{P}$-a.e. $\omega \in\Omega$, then as $\varepsilon\to 0$,
    $$\psi_\omega^\varepsilon \stackrel{L^1}{\to} \psi_{}^0:= \frac{1}{2} \phi_L -\frac{1}{2}\phi_R$$   
for $\mathbb{P}$-a.e. $\omega\in\Omega$.
     \label{thrm:B}
\end{thmx}
\noindent
\jp{Finally, we extend \thrm{thrm:A} to random metastable systems with $m\geq 2$ initially invariant sets. We highlight that this result builds on \cite{GTHW_metastable}, providing an explicit system of linear equations whose solution describes the limiting invariant density for random metastable systems with multiple initially invariant sets. As discussed in \Sec{sec:mstate}, one can naturally extend assumptions \hyperref[list:I1]{\textbf{(I1)}}-\hyperref[list:I6]{\textbf{(I6)}} and \hyperref[list:P1]{\textbf{(P1)}}-\hyperref[list:P7]{\textbf{(P7)}} to conditions {\textoverline{\textbf{(I1)}}}-{\textoverline{\textbf{(I6)}}} and {\textoverline{\textbf{(P1)}}}-{\textoverline{\textbf{(P7)}}}, yielding the following result  (\thrm{thrm:m-invdens}), whose proof may be found in \Sec{sec:mstate}.
\begin{thmx}
     Let $\{(\Omega,\mathcal{F},\mathbb{P},\sigma,\BV(I),\mathcal{L}^\varepsilon)\}_{\varepsilon\geq 0}$ be a \jp{family} of random dynamical systems of  metastable maps $T_\omega^\varepsilon:I\to I$ satisfying {\textoverline{\textbf{(I1)}}}-{\textoverline{\textbf{(I6)}}} and {\textoverline{\textbf{(P1)}}}-{\textoverline{\textbf{(P7)}}} (see \Sec{sec:mstate}). Suppose that $\Delta,N : \Omega \to M_{m\times m}(\mathbb{R})$ are as in \eqn{eqn:Delta} and \eqn{eqn:Mdecomp}. Denote by $(\phi_\omega^\varepsilon)_{\omega\in\Omega}$ the random invariant density of $(T_\omega^\varepsilon)_{\omega\in\Omega}$. If $\int_\Omega \Delta_\omega \, d\mathbb{P}(\omega)$ is invertible, then there exists a unique vector $v^0 = \begin{pmatrix}
      c_1 & c_2 & \cdots & c_m   
     \end{pmatrix}^T\in\mathbb{R}^m$ with $c_i\geq 0$ for each $i\in\{1,\cdots ,m\}$ and $\sum_{i=1}^m c_i=1$, such that as $\varepsilon\to 0$,
     $$\phi_\omega^{{\varepsilon}} \stackrel{L^1}{\to} \phi^0= \sum_{i=1}^m c_{i} \phi_i$$
     {uniformly} over $\omega\in\Omega$ away from a $\mathbb{P}$-null set. Further, $v^0$ is the solution to
     \begin{equation*}
         \left(I- \left(\int_{\Omega} \Delta_\omega\, d\mathbb{P}(\omega)\right)^{-1}\int_{\Omega}N_\omega\, d\mathbb{P}(\omega)\right)v^0 = 0
     \end{equation*}
     that satisfies $\sum_{i=1}^m (v^0)_i=1$.
    \label{thrm:C}
 \end{thmx}}
\noindent
We structure the paper as follows. In \Sec{sec:prelim} we introduce relevant definitions and results related to random dynamical systems and Perron-Frobenius operators. The initial system and the perturbations made to it are introduced in \Sec{sec:map+pert}. Here we discuss how this gives rise to random paired metastable systems. \Sec{sec:cont} is dedicated to showing that in the context of \cite{Crimmins}, the Oseledets decomposition for the Perron-Frobenius operator cocycle depends continuously on perturbations. This allows us to assume in \Sec{sec:inv-density} that any accumulation point of the random invariant density for the perturbed system is given by a convex combination of the initially invariant densities, allowing us to prove \thrm{thrm:A}. In \Sec{sec:2nd_space} we prove \thrm{thrm:B}, providing a characterisation of the second Oseledets space and thus a description of the corresponding coherent structure.   \jp{In \Sec{sec:mstate} we present \thrm{thrm:C}, an extension of \thrm{thrm:A} to the case where the initial system admits $m\geq2$ initially invariant sets.} Finally, in \Sec{sec:example}, we apply our results to Horan's random paired tent maps. 
\section{Preliminaries}
In this section, we collate definitions and results relevant to this paper. Primarily, we introduce transfer operator techniques for random dynamical systems.   
\begin{definition}
\sloppy A \textit{semi-invertible random dynamical system} is a tuple $(\Omega,\mathcal{F},\mathbb{P},\sigma,X,\mathcal{L})$, where the base $\sigma:\Omega\to \Omega$ is an \textit{invertible}, measure-preserving transformation of the probability space $(\Omega,\mathcal{F},\mathbb{P})$, $(X,\| \cdot \|)$ is a Banach space, and $\mathcal{L}:\Omega\to L(X)$ is a (strongly measurable) family of bounded linear operators of $X$, called the generator.\footnote{Here $L(X)$ denotes the space of bounded linear operators on the Banach space $X$.} 
\label{def:seminvertrds}
\end{definition}
\begin{remark}
    For convenience, whenever we refer to a random dynamical system, we assume it is semi-invertible. That is, the base map $\sigma:\Omega\to \Omega$ is invertible, but its generators need not be. 
\end{remark}
\noindent
In general, $X$ can be any given Banach space. In this paper, we will be interested in $X$ being the Banach space of functions with \textit{bounded variation}.
\begin{definition}
Let $(S,\mathcal{D},\mu)$ be a measure space with $S=[a,b]$. The space $\BV_\mu(S)$ is called the space of \textit{bounded variation} on $S$ where
$$||f||_{\BV_\mu(S)} = \inf_{\tilde{f} = f\, \mu-\mathrm{a.e.}}\var(\tilde{f})+||\tilde{f}||_{L^1(\mu)}$$
and $\var(f)$ is the total variation of $f$ over $S$;
    $$\var(f) = \sup\left\{ \sum_{i=1}^n |f(x_i)-f(x_{i-1})| \ \ \bigg| \ \ n\geq 1, \ a\leq x_0 <x_1<\cdots < x_n\leq b \right\}.$$
    \label{dfn:BV}
\end{definition}\noindent
Elements of $\BV_\mu(S)$, denoted $[f]_\mu$, are equivalence classes of functions with bounded variation on $S$. In this paper, we consider functions $f\in \BV(S)$ with norm $\|f\|_{\BV(S)}=\var(f)+\|f\|_{L^1(\mu)}$ and emphasise that $f$ is identified through a representative of minimal variation from the equivalence class $[f]_\mu$. 
\\
\\
\noindent
We associate with a \textit{non-singular transformation} $T$ a unique \textit{Perron-Frobenius operator}. Since it describes the evolution of ensembles of points, or densities, such operators serve as a powerful tool when studying the statistical behaviour of trajectories of $T$. 
\begin{definition}
    A measurable function $T:I\to I$ on a measure space $(I,\mathcal{D},\mu)$ is a \textit{non-singular transformation} if $\mu(T^{-1}(D))=0$ for all $D\in \mathcal{D}$ such that $\mu(D)=0$.
    \label{dfn:non-sing}
\end{definition}

\begin{definition}
\label{def:pfoperator}
    Let $(I,\mathcal{D},\mu)$ be a measure space and $T:I\to I$ be a non-singular transformation. \jp{The unique operator $\mathcal{L}_T:L^1(\mu)\to L^1(\mu)$ satisfying
    $$\int_D\mathcal{L}_T(f)(x)\, d\mu(x) = \int_{T^{-1}(D)}f(x)\, d\mu(x)$$
    for all $f\in L^1(\mu)$ and $D\in\mathcal{D}$ is called the \textit{Perron-Frobenius operator} (associated with $T$).}
\end{definition}\noindent
In certain cases, $\mathcal{L}_T$ may be restricted (or extended) to a bounded linear operator on another Banach space (e.g. $X=\BV(I)$) in which case the operators are still referred to as Perron-Frobenius operators. 
\begin{remark}
\sloppy For notational purposes we denote by $(\mathcal{L}_\omega)_{\omega\in\Omega}$ the family of Perron-Frobenius operators associated with the family of non-singular transformations $(T_\omega)_{\omega\in\Omega}$.
\end{remark}\noindent
Combining \dfn{def:seminvertrds} and \dfn{def:pfoperator}, one can construct a random dynamical system from a family of Perron-Frobenius operators $(\mathcal{L}_{\omega})_{\omega\in\Omega}$ associated with the set of non-singular transformations $(T_\omega)_{\omega\in\Omega}.$ This forms a Perron-Frobenius operator cocycle.

\begin{example}[Perron-Frobenius operator cocycle]
Consider a random dynamical system $(\Omega,\mathcal{F},\mathbb{P},\sigma,X,\mathcal{L})$, where $\sigma:\Omega\to \Omega$ is an invertible, ergodic and measure-preserving transformation, and its generator $\mathcal{L}:\Omega \to L(X)$ takes values on the Perron-Frobenius operators associated with the non-singular transformations $T_\omega$ of the measure space $(I,\mathcal{D},\mu)$, given by $\omega \mapsto \mathcal{L}_{\omega}$. This gives rise to a Perron-Frobenius operator cocycle,
$$(n,\omega)\mapsto \mathcal{L}_{\omega}^{(n)}\jp{:=}\mathcal{L}_{\sigma^{n-1}\omega}\circ \cdots \circ \mathcal{L}_{\omega}.$$
Here the evolution of a density $f$ is governed by a cocycle of Perron-Frobenius operators driven by the base dynamics $\sigma:\Omega\to \Omega$. That is, if $f$ represents the initial mass distribution in the system, then $\mathcal{L}_\omega^{(n)}$ describes the mass distribution after the application of \jp{$T_\omega^{(n)}:=T_{\sigma^{n-1}\omega}\circ \cdots \circ T_\omega$}. 
\end{example}
\noindent
The \textit{random fixed points} of the Perron-Frobenius operator are of interest in random dynamical systems. 
\begin{definition}
Let $(\Omega,\mathcal{F},\mathbb{P},\sigma,X,\mathcal{L})$ be a random dynamical system, with associated non-singular transformations $T_\omega:I\to I$. A family $(\mu_\omega)_{\omega\in\Omega}$ is called a \textit{random invariant measure} for $(T_\omega)_{\omega\in \Omega}$ if $\mu_\omega$ is a probability measure on $I$, for any Borel measurable subset $D$ of $I$ the map $\omega \mapsto \mu_\omega(D)$ is measurable, and 
$$\mu_\omega(T^{-1}_\omega(D)) = \mu_{\sigma\omega}(D)$$
for $\mathbb{P}$-a.e. $\omega\in \Omega$. A family $(h_\omega)_{\omega\in\Omega}$ is called a \textit{random invariant density} for $(T_\omega)_{\omega\in \Omega}$ if $h_\omega\geq 0$, $h_\omega\in L^1(\mu)$, $||h_\omega||_{L^1(\mu)}=1$, the map $\omega\mapsto h_\omega$ is measurable, and 
$$\mathcal{L}_\omega h_\omega = h_{\sigma\omega}$$
for $\mathbb{P}$-a.e. $\omega\in \Omega$. We also say $(h_\omega)_{\omega\in\Omega}$ is a \textit{random fixed point} of $(\mathcal{L}_\omega)_{\omega\in\Omega}$.
\label{def:RIM-RID}
\end{definition}
\begin{remark}
    If the random invariant measure $(\mu_\omega)_{\omega\in\Omega}$ is absolutely continuous with respect to Lebesgue, we refer to it as a random absolutely continuous invariant measure (RACIM). In this case, its density $(h_\omega)_{\omega\in\Omega} = (\frac{d\mu_\omega}{d\leb})_{\omega\in\Omega}$ is a random invariant density.
    \label{rem:RACIM}
\end{remark}

\begin{definition}
The cocycle $(\mathcal{L}_\omega)_{\omega\in\Omega}$ on a Banach space $X$ is \textit{strongly measurable} if for any fixed $f\in X$, $\omega\mapsto\mathcal{L}_\omega f$ is $(\mathcal{F}_\Omega,\mathcal{F}_X)$-measurable.
\label{dfn:strong-measurable}
\end{definition}\noindent
In this definition, $\mathcal{F}_\Omega$ denotes the $\sigma$-algebra over $\Omega$, and $\mathcal{F}_X$ denotes the Borel $\sigma$-algebra on $X$.
\begin{remark}
 \sloppy     In the case that $(\mathcal{L}_\omega)_{\omega\in\Omega}$ is {strongly measurable} we call $(\Omega,\mathcal{F},\mathbb{P},\sigma,X,\mathcal{L})$ a strongly measurable random dynamical system. 
\end{remark}
\noindent
 In some cases one cannot obtain strong measurability of $(\mathcal{L}_\omega)_{\omega\in\Omega}$ in which case one can hope that the mapping $\omega\mapsto \mathcal{L}_\omega$ is $\mathbb{P}$\textit{-continuous}, a concept introduced by Thieullen in \cite{IOC_TP}. This gives rise to a \textit{$\mathbb{P}$-continuous random dynamical system}.

 \begin{definition}
      Let $(\Omega,\mathcal{F},\mathbb{P})$ be a Borel probability space and $(Y, \tau)$ a topological space. A mapping $\mathcal{L}:\Omega \to Y$ is said to be \textit{$\mathbb{P}$-continuous} if $\Omega$ can be expressed as a countable union of Borel sets such that the restriction of $\mathcal{L}$ to each of them is continuous. 
      \label{def:P-cont}
 \end{definition}
\begin{remark}
    Let $(\Omega,\mathcal{F},\mathbb{P},\sigma,X,\mathcal{L})$ be a random dynamical system. If its generators $\mathcal{L}:\Omega \to L(X)$, given by $\omega\mapsto\mathcal{L}_\omega$, are $\mathbb{P}$-continuous with respect to the norm topology on $L(X)$, then the cocycle $(\mathcal{L}_\omega)_{\omega\in\Omega}$ is called $\mathbb{P}$-continuous, and the tuple $(\Omega, \mathcal{F}, \mathbb{P}, \sigma, X , \mathcal{L})$ is a \textit{$\mathbb{P}$-continuous random dynamical system}. 
\end{remark}
\noindent
The asymptotic behaviour of the spectral picture for the Perron-Frobenius operator cocycle is of great interest when studying the statistical properties of random systems. 
\begin{definition}
The \textit{index of compactness} of an operator $\mathcal{L}_\omega$ denoted $\alpha(\mathcal{L}_\omega)$, is the infimum of those real numbers $t$ such that the image of the unit ball in $X$ under $\mathcal{L}_\omega$ may be covered by finitely many balls of radius $t$.
\label{def:IndexOfCompactness}
\end{definition}\noindent
The index of compactness provides a notion of `how far' an operator is from being compact. This definition was extended by Thieullen to random compositions of operators in \cite{IOC_TP}.
\begin{definition}
The \textit{asymptotic index of compactness} for the cocycle $(\mathcal{L}_\omega)_{\omega\in\Omega}$ on $X$ is
$$\kappa(\omega):=\lim_{n\to \infty} \frac{1}{n}\log \alpha(\mathcal{L}_\omega^{(n)}).$$
\label{def:AssymptoticIndexOfCompactness}
\end{definition}\noindent
\jp{We call the cocycle \textit{quasi-compact} if its \textit{maximal Lyapunov exponent}, $$\lambda_1(\omega) :=\lim_{n\to \infty} \frac{1}{n}\log||\mathcal{L}_\omega^{(n)}||,$$ is strictly larger than its asymptotic index of compactness. If $\int_\Omega \log||\mathcal{L}_\omega||\, d\mathbb{P}(\omega)<\infty$, by the Kingman sub-additive ergodic theorem, these limits exist and are constant for $\mathbb{P}$-a.e. $\omega\in\Omega$ if the driving system $\sigma:\Omega\to \Omega$ is ergodic {\cite{Kingman}}.} Under some additional assumptions on the random dynamical system, one can obtain a spectrum of Lyapunov exponents through {multiplicative ergodic theorems}. One example is \jp{the} \textit{Oseledets decomposition} \jp{which} splits $X$ into $\omega$-dependent subspaces which decay/expand according to its associated Lyapunov exponent\jps{s} $\lambda_i(\omega)$. Again, these are constant $\mathbb{P}$-a.e. when $\sigma$ is ergodic. 
\begin{definition}
    Consider a random dynamical system $\mathcal{R}=(\Omega,\mathcal{F},\mathbb{P},\sigma,X,\mathcal{L})$. An \textit{Oseledets splitting} for $\mathcal{R}$ consists of isolated (exceptional) Lyapunov exponents 
    $$\infty>\lambda_1>\lambda_2>\cdots>\lambda_\ell>\kappa\geq -\infty,$$
    where the index $\ell\geq 1$ is allowed to be finite or countably infinite, and Oseledets subspaces $V_1(\omega),\dots, V_\ell(\omega), W(\omega)$ such that for all $i=1,\dots, \ell$ and $\mathbb{P}$-a.e. $\omega\in\Omega$ we have
   \begin{itemize}
    \item[(a)] $\dim(V_i(\omega))=m_i<\infty$;
    \item[(b)] $\mathcal{L}_\omega V_i(\omega)=V_i(\sigma\omega)$ and $\mathcal{L}_\omega W(\omega)\subseteq W(\sigma\omega)$;
    \item[(c)] $V_1(\omega)\oplus\cdots \oplus V_\ell(\omega)\oplus W(\omega)=X$;
    \item[(d)] for $f\in V_i(\omega)\setminus\{0\}$, $\lim_{n\to\infty}\frac{1}{n}\log||\mathcal{L}_\omega^{(n)}f||= \lambda_i$;
    \item[(e)] for $f\in W(\omega)\setminus\{0\}$, $\lim_{n\to\infty}\frac{1}{n}\log||\mathcal{L}_\omega^{(n)}f||\leq \kappa$.
\end{itemize}
\label{def:Osc-splitting}
\end{definition}

\begin{theorem}[{\cite[Theorem 17]{FLQ_semi}}]
Let $\Omega$ be a Borel subset of a separable complete metric space, $\mathcal{F}$ the Borel $\sigma$-algebra and $\mathbb{P}$ a Borel probability. Let $X$ be a Banach space and consider a random dynamical system $\mathcal{R}=(\Omega,\mathcal{F},\mathbb{P},\sigma,X,\mathcal{L})$ with base transformation $\sigma:\Omega\to \Omega$ an ergodic homeomorphism, and suppose that the generator $\mathcal{L} : \Omega \to L(X)$ is $\mathbb{P}$-continuous and satisfies
$$\int_\Omega \log^+ ||\mathcal{L}_\omega||\, d\mathbb{P}(\omega)<\infty.$$
If $\kappa<\lambda_1$, $\mathcal{R}$ admits a unique $\mathbb{P}$-continuous Oseledets splitting.
\label{the:OsceledetsDecomp}
\end{theorem}
\noindent
We refer to the set of all $\lambda_i$ as the Lyapunov spectrum of the Perron-Frobenius operator cocycle $(\mathcal{L}_\omega)_{\omega\in\Omega}$. We recall that in \thrm{the:OsceledetsDecomp} the Lyapunov spectrum and asymptotic index of compactness are constant for $\mathbb{P}$-a.e. $\omega\in\Omega$ since $\sigma$ is ergodic. When the Oseledets splitting of $(\mathcal{L}_\omega)_{\omega\in\Omega}$ \jp{satisfies uniform hyperbolicity conditions}, by adopting the same terminology of \cite{Crimmins}, we call this a \textit{hyperbolic Oseledets splitting}. \jp{Let $\mathbb{X}=\Omega\times X$, and define the map $\pi:\mathbb{X}\to\Omega$ as the projection onto $\Omega$. Denote by $\mathrm{End}(\mathbb{X},\sigma)$ the set of all bounded linear endomorphisms of $\mathbb{X}$ covering $\sigma$, i.e.
$$\mathrm{End}(\mathbb{X},\sigma)=\{\mathcal{L}:\mathbb{X}\to \mathbb{X} \ | \ \pi\circ \mathcal{L} = \sigma \circ \pi \ \mathrm{and} \ f\mapsto \tau_{\sigma\omega}(\mathcal{L}(\omega,f))\in L(X{})\}$$
where $\tau_\omega:\pi^{-1}(\omega)\to X$ is given by $\tau_\omega(\omega,f)=f$. Let $\mathcal{G}(X)$ denote the Grassmanian of $X$ (the set of all closed, \jps{complemented} subspaces of $X$). For $E,F\in\mathcal{G}(X)$, let $\Pi_{E||F}$ be the projection onto $E$ parallel to $F$ so that the image of $\Pi_{E||F}$ is $E$ and $\mathrm{ker}(\Pi_{E||F})=F$. For a fixed $d\in\mathbb{Z}^+$ we let $\mathcal{G}_d(X)$ and $\mathcal{G}^d(X)$ denote the set of closed $d$-dimensional and $d$-codimensional subspaces of $X$ respectively.
\begin{definition} 
    Suppose that $\mathcal{L}\in \mathrm{End}(\mathbb{X},\sigma)$, $d\in\mathbb{Z}^+$, $0\leq \mu<\lambda$, $(E_\omega)_{\omega\in\Omega}\in \mathcal{G}_d(X)^\Omega$ and $(F_\omega)_{\omega\in\Omega}\in \mathcal{G}^d(X)^\Omega$. We say that the family of subspaces $(E_\omega)_{\omega\in\Omega}$ and $(F_\omega)_{\omega\in\Omega}$ form a \textit{$(\lambda,\mu,d)$-hyperbolic splitting} for $\mathcal{L}$, and that $\mathcal{L}$ has a \textit{hyperbolic splitting} of index $d$, if there exists constants $C_{\lambda},C_\mu,\Theta>0$ such that:
    \begin{enumerate}
    \item[(a)] For every $\omega\in\Omega$ we have $E_\omega \oplus F_\omega = X$ and
    $$\max\{\|\Pi_{F_\omega||E_\omega} \|,\|\Pi_{E_\omega||F_\omega} \|\}\leq \Theta.$$
    \item[(b)] For each $\omega\in\Omega$ we have $\mathcal{L}_\omega E_\omega = E_{\sigma\omega}$. Moreover, for every $n\in\mathbb{N}$ and $f\in E_\omega$ we have 
    $$||\mathcal{L}_\omega^{(n)}f||\geq C_\lambda \lambda^n ||f||.$$
    \item[(c)] For each $\omega\in\Omega$ we have $\mathcal{L}_\omega F_\omega \subseteq F_{\sigma\omega}$ and for every $n\in\mathbb{N}$ we have 
    $$||\mathcal{L}_\omega^{(n)}|_{F_\omega}||\leq C_\mu \mu^n.$$
    \end{enumerate}
    \label{def:hyperbolic-splitting}
\end{definition}}
\jp{\begin{remark}
    In the deterministic setting ($\mathcal L_\omega =\mathcal L_0$ for $\mathbb{P}$-a.e. $\omega \in \Omega$), if $\mathcal L_0$ is quasi-compact, the random dynamical system $(\Omega, \mathcal{F}, \mathbb{P}, \sigma, X, \mathcal{L})$ has a hyperbolic Oseledets splitting (see \cite[Example 5.6]{Crimmins}). Furthermore, if the leading eigenvalue of $\mathcal L_0$ is 1, with multiplicity $d\geq 1$, and $\mathcal L_0$ has no other eigenvalues of modulus one, then the random dynamical system has a hyperbolic Oseledets splitting of index $d$.  \label{rmk:free-splitting}
\end{remark}} 
\noindent As opposed to studying the asymptotic index of compactness \jp{directly}, one can often prove that the Perron-Frobenius operator cocycle is quasi-compact by showing the collection $(\mathcal{L}_\omega)_{\omega\in\Omega}$ satisfies a \textit{uniform Lasota-Yorke inequality}.
\begin{definition}
    We say that $\mathcal{L}_\omega:(X,\| \cdot\|)\to (X,\| \cdot\|)$ satisfies a \textit{uniform Lasota-Yorke inequality} with constants $C_1,C_2,r,R>0$ and $0<r<R\leq 1$, if for every $\omega\in \Omega$, $f\in X$ and $n\in \mathbb{N}$ we have 
    $$||\mathcal{L}_\omega^{(n)} f||\leq C_1r^n||f||+C_2 R^n|f|$$
    where $|\cdot|$ is a weak norm on $(X,\| \cdot\|)$.\footnote{Recall that $|\cdot|$ is a weak norm on $(X,\|\cdot\|)$ if there exists a constant $C>0$ such that $|f|\leq C \|f\|$ for all $f\in X$.} 
    \label{def:ULY}
\end{definition}
\label{sec:prelim} \noindent
In the proceeding sections, we will consider perturbations of Perron-Frobenius operator cocycles. One way to quantify the size of perturbations is through the \textit{operator triple norm}.
\begin{definition}
    Let $\mathcal{L}_\omega :(X,\| \cdot\|)\to (X,\| \cdot\|)$ where $X$ is a Banach space equipped strong and weak norm $\|\cdot\|$ and $|\cdot|$, respectively. The \textit{operator triple norm} of $\mathcal{L}_\omega$ is
    $$|||\mathcal{L}_\omega|||:=\sup_{||f||=1}|\mathcal{L}_\omega f|.$$
    \label{def:triple-norm}
\end{definition} \noindent
Finally, we make use of Landau notation throughout the paper. In the following definitions we consider functions $f,g:\mathbb{R}\to \mathbb{R}$.
\begin{definition}
     We write $f(x)=O_{x\to a}(g(x))$ if there exists $M,\delta>0$ such that for all $x$ satisfying $|x-a|<\delta$,
     $$|f(x)|\leq M|g(x)|.$$
     \label{def:O}
\end{definition}
\begin{definition}
    We write $f(x)=o_{x\to a}(g(x))$ if for all $C>0$ there exists $\delta>0$ such that for all $x$ satisfying $|x-a|<\delta$,
    $$|f(x)|\leq C|g(x)|.$$
    \label{def:o}
\end{definition}
\begin{remark}
   In many situations the constants $C,M$ involved in the above asymptotic approximations may depend on a second variable, say $\omega$. In this case we write $f(x)=O_{\omega,x\to a}(g(x))$ and $f(x)=o_{\omega,x\to a}(g(x))$, respectively.
\end{remark}

\section{Paired metastable systems and their perturbations}
\label{sec:map+pert}
In this section, we introduce a class of random dynamical systems with two metastable states. We define these maps as perturbations of an autonomous system possessing two initially invariant intervals $I_L$ and $I_R$, which both support a unique absolutely continuous invariant measure (ACIM). Upon perturbation, so-called \textit{random holes} emerge, allowing trajectories to switch between $I_L$ and $I_R$. This gives rise to systems with a unique random absolutely continuous invariant measure, describing the long-term statistical behaviour of paired metastable systems.    

\subsection{The initial system}
Let $I=[-1,1]$ be equipped with the Borel $\sigma$-algebra $\mathcal{B}$, and Lebesgue measure $\leb$. Suppose that $T^0: I\to I$ is a piecewise $C^2$ uniformly expanding map with two invariant subintervals. This means $T^0$ satisfies the following conditions.
\begin{itemize}
    \item[\textbf{(I1)}] Piecewise $C^2$.\\
    There exists a critical set $\mathcal{C}^0 = \{-1=c_{0}<c_{1}<\cdots < c_{d}=1\}$ such that for each $i=0,\dots,d-1$, the map $T^0|_{(c_i,c_{i+1})}$ extends to a $C^2$ function $\hat{T}_i^0$ on a neighbourhood of $[c_i,c_{i+1}]$.  \label{list:I1}
    \end{itemize}
    \begin{itemize}    
    \item[\textbf{(I2)}] Uniform expansion.\\
    $$\inf_{x\in I\setminus \mathcal{C}^0}|(T^0)^\prime(x)|>1.$$ \label{list:I2}
    \end{itemize}
    \begin{itemize}
    \item[\textbf{(I3)}] Existence of boundary point.\\
    There is a boundary point $b\in(-1,1)$ such that the sets $I_L := [-1,b]$ and $I_R:=[b,1]$ are invariant under $T^0$.\footnote{That is for $\star\in\{ L,R \}$, $(T^0)^{-1}(I_\star)\subseteq I_\star$.} \label{list:I3}
    \end{itemize}
    Denote by $\mathcal{L}^0$ the Perron-Frobenius operator associated with $T^0$ acting on $(\BV(I),||\cdot||_{\BV(I)})$ with weak norm $||\cdot||_{L^1(\leb)}$.
    \begin{remark}
     Thanks to \cite{EG_LY}, conditions \hyperref[list:I1]{\textbf{(I1)}} and \hyperref[list:I2]{\textbf{(I2)}} ensure that the Perron-Frobenius operator associated with $T^0$ acting on $(\BV(I),||\cdot||_{\BV(I)})$ satisfies a Lasota-Yorke inequality. This fact will be used when we wish to emphasise that $\mathcal{L}^0$ is quasi-compact. \label{rem:LY0}
    \end{remark}\noindent
Note that the dynamics of the initial system is autonomous, thus for $\star\in\{L,R\}$, the existence of an ACIM of bounded variation for $T^0|_{I_\star}$ is guaranteed by the classical work of Lasota and Yorke \cite{LY_acim}. We assume in addition the following. 
\begin{itemize}
    \item[\textbf{(I4)}] Unique ACIMs on initially invariant sets.\\
    For $\star \in \{L,R\}$, $T^0|_{I_\star}$ has only one ACIM $\mu_\star$, whose density is denoted by $\phi_\star:= d\mu_\star/\dleb$.
    \label{list:I4}
\end{itemize}
\hyperref[list:I4]{\textbf{(I4)}} implies that all ACIMs of $T^0$ may be expressed as convex combinations of those ergodic measures supported on $I_L$ and $I_R$. Conditions guaranteeing that \hyperref[list:I4]{\textbf{(I4)}} is satisfied are outlined in \cite[Theorem 1]{LY_acim}.
\\
\\
\jp{For $k\in\mathbb{N}$, let $T^{0\, (k)}:= \overbrace{T^0\circ \cdots \circ \ T^0}^{k\text{-times}}$}. Define the set of all points belonging to $H^0:=(T^0)^{-1}(\{b\})\setminus \{b\}$ as infinitesimal holes.
\begin{itemize}
    \item[\textbf{(I5)}] No return of the critical set to infinitesimal holes.\\
    For every $k>0$, $T^{0 \,(k)}(\mathcal{C}^0)\cap H^0 = \emptyset$.
    \label{list:I5}
\end{itemize}
As discussed in \cite[Section 2.1]{GTHW_metastable}, condition \hyperref[list:I5]{\textbf{(I5)}} is essential to ensure that for $\star\in \{L,R\}$, each unique invariant density $\phi_\star$, guaranteed by \hyperref[list:I4]{\textbf{(I4)}}, is continuous at all points in $I_\star \cap H^0$. Finally, we require that for $\star\in \{L,R\}$, the invariant densities $\phi_\star$ are positive at each point in $I_\star \cap H^0$. 
\begin{itemize}
    \item[\textbf{(I6)}] Positive ACIMs at infinitesimal holes.\\
   $\phi_L$ is positive at each of the points in $H^0\cap I_L$, and $\phi_R$ is positive at each of the points in $H^0\cap I_R$.
   \label{list:I6}
\end{itemize}
Condition \hyperref[list:I6]{\textbf{(I6)}} is satisfied if, for example, the maps $T^0|_{I_\star}$ are weakly covering for $\star\in\{L,R\}$ \cite{Liverani_DOC}.\footnote{A piecewise expanding map $T:I\to I$ with critical set $\mathcal{S} = \{-1=s_{0}<s_{1}<\cdots < s_{d}=1\}$ is weakly covering if there is some $N\in\mathbb{N}$ such that for every $i$, $\cup_{k=0}^N T^{(k)}([s_i,s_{i+1}])=I$.}\\

\subsection{The perturbations}
\label{sec:perts}
In what follows, let $(\Omega,\mathcal{F},\mathbb{P})$ be a probability space. Fix $\varepsilon>0$ and let $\omega\in\Omega$. We consider $C^2$-small perturbations of $T^0:I\to I$, denoted $T^\varepsilon:\Omega\times I \to I$, driven by an ergodic transformation $\sigma:\Omega\to \Omega$.\footnote{For notational convenience, and in a slight abuse of notation, $T^0$ will also denote the \textit{random} map $T^0:\Omega\times I \to I$ which satisfies $T^0_\omega :=T^0$ for every $\omega\in\Omega$.} This means $\sigma:\Omega\to \Omega$ and $T^\varepsilon:\Omega\times I \to I$ satisfy the following.
\begin{itemize}
    \item[\textbf{(P1)}] Ergodic driving and finite range.\\
    $\sigma:\Omega\to \Omega$ is an ergodic, $\mathbb{P}$-preserving, \jp{invertible transformation} of the probability space $(\Omega,\mathcal{F},\mathbb{P})$; for all $\varepsilon\geq0$, the mapping $\omega \mapsto T_{\omega}^\varepsilon$ has finite range; and the skew-product
    $$(\omega,x)\mapsto (\sigma\omega,T_\omega^\varepsilon(x))$$
    is measurable with respect to the product $\sigma$-algebra $\mathcal{F}\otimes \mathcal{B}$ on $\Omega\times I$.
    \label{list:P1}
\end{itemize}
\begin{itemize}
    \item[\textbf{(P2)}] $C^2$-small perturbations.\\
    There exists a critical set $\mathcal{C}^\varepsilon_\omega = \{-1 = c_{0,\omega}^\varepsilon<\cdots < c_{d,\omega}^\varepsilon = 1\}$ such that for each $i=0,\dots,d$, $\omega \mapsto c_{i,\omega}^\varepsilon$ is measurable and $\varepsilon \mapsto c_{i,\omega}^\varepsilon$ is $C^2$. Furthermore, there exists $\delta>0$ such that: 
    \label{list:P2}
\begin{itemize}
    \item[(a)] for $i=1,\dots,d-2$, $[c_{i}+\delta,c_{i+1}-\delta]\subset [c_{i,\omega}^\varepsilon,c_{i+1,\omega}^\varepsilon] \subset [c_{i}-\delta,c_{i+1}+\delta]$,\footnote{In this way, we have a one-to-one correspondence between the critical sets of $T_\omega^\varepsilon$ and $T^0$ given by $\mathcal{C}_\omega^\varepsilon$ and $\mathcal{C}^0$, respectively.}
    \item[(b)] for $i=0,\cdots,d$, $c_{i,\omega}^\varepsilon$ converges uniformly over $\omega\in\Omega$ away from a  $\mathbb{P}$-null set to its corresponding point $c_{i}\in \mathcal{C}^0$ as $\varepsilon\to 0$, and
    \item[(c)] for $i=0,\dots,d-1$, there is a $C^2$ extension $\hat{T}_{i,\omega}^\varepsilon:[c_{i}-\delta,c_{i+1}+\delta]\to \mathbb{R}$ of $T_\omega^\varepsilon|_{(c_{i,\omega}^\varepsilon,c_{i+1,\omega}^\varepsilon)}$ that converges in $C^2$, and uniformly over $\omega\in\Omega$ away from a  $\mathbb{P}$-null set to the $C^2$ extension $\hat{T}_{i}^0:[c_{i}-\delta,c_{i+1}+\delta]\to \mathbb{R}$ of $T^0|_{(c_{i},c_{i+1})}$ as $\varepsilon\to 0$.
\end{itemize}
\end{itemize}
Denote by $\mathcal{L}_\omega^\varepsilon$ the Perron-Frobenius operator associated with $T_\omega^\varepsilon$ acting on $(\BV(I),||\cdot||_{\BV(I)})$ with weak norm $||\cdot||_{L^1(\leb)}$. 
\begin{remark}
  We impose \hyperref[list:P1]{\textbf{(P1)}} to ensure the mapping $\omega\mapsto \mathcal{L}_\omega^\varepsilon$ is $\mathbb{P}$-continuous (recall \dfn{def:P-cont}). This is satisfied through the relaxed condition that $\omega \mapsto T_{\omega}^\varepsilon$ has countable range. In our setting, to ensure uniform convergence of certain quantities, we instead require the stricter condition that $\omega \mapsto T_{\omega}^\varepsilon$ has finite range (as discussed in \cite[Remark 3.10]{BS_rand}).
\end{remark}

\begin{remark}
    Thanks to \cite[Proposition 3.12]{FGTQ_LYmapStab} (which follows from \cite[Lemma 13]{K_StochStab}) and \cite[Example 5.2]{EG_LY}, \hyperref[list:P2]{\textbf{(P2)}} asserts that 
     $$\lim_{\varepsilon\to 0} \esssup_{\omega\in\Omega} ||| \mathcal{L}_\omega^\varepsilon - \mathcal{L}^0 |||=0$$
     where $|||\cdot|||$ denotes the $\BV-L^1$ triple norm (see \dfn{def:triple-norm}).  
     \label{rem:p1-imp-trip}
\end{remark} 
\noindent
\noindent
For $\star\in \{L,R\}$ we define the holes $H_{\star,\omega}^\varepsilon$ as the set of all points mapping from $I_\star$ to $I_\star^c$ under one iteration of $T_\omega^\varepsilon$. Namely, $H_{\star,\omega}^\varepsilon:= I_\star \cap (T_\omega^\varepsilon)^{-1}(I_\star^c)$. 
\begin{itemize}
    \item[\textbf{(P3)}] Convergence of holes.\\
    For $\star\in\{L,R\}$, $H_{\star,\omega}^\varepsilon$ is a union of finitely many intervals, and as $\varepsilon\to 0$, $H_{\star,\omega}^\varepsilon$ converges to $H^0\cap I_{\star}$ (in the Hausdorff metric) uniformly over $\omega\in\Omega$ away from a $\mathbb{P}$-null set.\footnote{Recall that the set $H^0=(T^0)^{-1}(\{b\})\setminus \{b\}$ consists of infinitesimal holes.}
    \label{list:P3}
\end{itemize}
Since the holes are themselves functions of $\omega\in\Omega$, we place some constraints on how the measures of them behave across each fibre.   
\begin{itemize}
    \item[\textbf{(P4)}] Measure of holes.\\
    For $\star\in\{L,R\}$, $\mu_\star(H_{\star,\omega}^\varepsilon)=\varepsilon \beta_{\star,\omega}+o_{\varepsilon\to 0}(\varepsilon)$ where $\beta_\star \in  L^\infty(\mathbb{P})$.\footnote{We emphasise that the error in the measure of the hole is independent of $\omega\in\Omega$.}
    \label{list:P4}
\end{itemize}
\noindent
To ensure the system depends continuously on perturbations, we require the perturbed system to admit a hyperbolic Oseledets splitting (see \dfn{def:hyperbolic-splitting}). For this, we assume the following.
\begin{itemize}
    \item[\textbf{(P5)}] Uniform Lasota-Yorke inequality.\\
    The Perron-Frobenius operator associated with $T_\omega^\varepsilon$, denoted $\mathcal{L}_\omega^\varepsilon$, acting on \linebreak $(\BV(I),||\cdot||_{\BV(I)})$ with weak norm $||\cdot||_{L^1(\leb)}$ satisfies a uniform Lasota-Yorke inequality across both $\omega\in\Omega$ and $\varepsilon>0$ (see \dfn{def:ULY}).
    \label{list:P5}
\end{itemize}
Further, we require $T_\omega^\varepsilon$ to admit a unique random absolutely continuous invariant measure (see \rem{rem:RACIM}).
\begin{itemize}
    \item[\textbf{(P6)}] Unique RACIM.\\
    For $\varepsilon>0$, $(T_\omega^\varepsilon)_{\omega\in\Omega}$ has only one RACIM $(\mu_\omega^\varepsilon)_{\omega\in\Omega}$, with density $(\phi_\omega^\varepsilon)_{\omega\in\Omega}:=\left(d\mu_\omega^\varepsilon/\dleb\right)_{\omega\in\Omega}$.
    \label{list:P6}
\end{itemize}
Cases in which \hyperref[list:P6]{\textbf{(P6)}} is satisfied are outlined in \cite{Unique-RACIM}. Finally, as discussed in \cite[Section 2.4]{GTHW_metastable}, we enforce a condition that ensures no holes emerge near the boundary. 
\begin{itemize}
    \item[\textbf{(P7)}] Boundary condition.
    \begin{itemize}
        \item[(a)] If $b\notin \mathcal{C}_0$, then $T^0(b)=b$ and for all $\varepsilon>0$ and $\mathbb{P}$-a.e. $\omega\in\Omega$, $T_\omega^\varepsilon(b)=b$. 
        \item[(b)] If $b\in \mathcal{C}_0$, then $T^0(b^-)<b<T^0(b^+)$ and for all $\varepsilon>0$ and $\mathbb{P}$-a.e. $\omega\in\Omega$, $b\in \mathcal{C}_\omega^\varepsilon$.\footnote{We denote by  $T^0(b^\mp)$ the left and right limits of $T^0(x)$ as $x\to b$, respectively.}
    \end{itemize}
    \label{list:P7}
\end{itemize}

\noindent
Throughout the remainder of the paper we assume that the conditions of \Sec{sec:map+pert} are satisfied. 
\section{Continuity of Oseledets Decomposition}
\label{sec:cont}
In this section, we address whether the Oseledets projections and Lyapunov exponents of Perron-Frobenius operator cocycles of paired metastable maps are continuous with respect to the perturbations described in \Sec{sec:map+pert}. This result relies on Crimmins' random perturbation theory \cite{Crimmins}, a random analogue of the renowned Keller-Liverani perturbation theory \cite{stab_spec}. The main result of this section allows us to express the limiting functions spanning the top two Oseledets spaces for paired metastable systems as linear combinations of the initially invariant densities.   
\\
\\
\noindent
The results of \cite{Crimmins} require the Banach space on which the Perron-Frobenius operator is acting to be separable. In our setting, for all $\varepsilon\geq 0$, we consider the Perron-Frobenius operator acting on $\BV$ which is not separable. As described in \cite[Appendix 2.B]{thermoformalism}, this separability assumption is used throughout \cite{Crimmins} to obtain measurability of certain objects. It is argued in \cite{thermoformalism} that Crimmins' stability result may be applied to the non-separable Banach space $\BV$ under the alternative condition \hyperref[list:P1]{\textbf{(P1)}}. We are interested in applying \cite[Theorem A]{Crimmins} to paired metastable systems. As opposed to considering a \jp{family} of separable strongly measurable linear random dynamical systems, as in the statement of \cite[Theorem A]{Crimmins}, \hyperref[list:P1]{\textbf{(P1)}} allows us to apply \cite[Theorem A]{Crimmins} to a sequence of $\mathbb{P}$-continuous random dynamical systems whilst ensuring the conclusions of the result still hold. Therefore, with support of \cite[Appendix 2.B]{thermoformalism}, under \hyperref[list:P1]{\textbf{(P1)}}, \cite[Theorem A]{Crimmins} asserts that if:
\begin{itemize}
    \item[1.] $(\Omega,\mathcal{F},\mathbb{P},\sigma, \BV(I),\mathcal{L}^0)$ has a hyperbolic Oseledets splitting and is $\|\cdot\|_{L^1(\leb)}$-bounded (see \dfn{def:hyperbolic-splitting});
    \item[2.] the set $\{(\Omega,\mathcal{F},\mathbb{P},\sigma, \BV(I),\mathcal{L}^\varepsilon)\}_{\varepsilon\geq 0}$ satisfies a uniform Lasota-Yorke inequality (see \dfn{def:ULY}); and 
    \item[3.] $\lim_{\varepsilon\to 0} \esssup_{\omega\in\Omega} ||| \mathcal{L}_\omega^\varepsilon - \mathcal{L}^0 |||=0$ where $|||\cdot|||$ denotes the $\BV-L^1$ triple norm (see \dfn{def:triple-norm}),
\end{itemize}
then $\mathcal{L}_\omega^\varepsilon$ has an Oseledets splitting for sufficiently small $\varepsilon$, and the Lyapunov exponents and Oseledets projections of $\mathcal{L}_\omega^\varepsilon$ converge to those of $\mathcal{L}^0$ as $\varepsilon\to 0$. 
\\
\\
\noindent
In this paper, we use the above to guarantee that the Oseledets spaces and Lyapunov exponents of $\mathcal{L}_\omega^\varepsilon$ depend continuously on perturbations. We refer the reader to \cite{Crimmins} for the precise statements of relevant results.
\begin{remark}
Under \hyperref[list:P1]{\textbf{(P1)}}, for all $\varepsilon\geq 0$, if $\omega \mapsto T_\omega^\varepsilon$ has finite range, then the map $\omega\mapsto \mathcal{L}_\omega^\varepsilon$ becomes $\mathbb{P}$-continuous. If $\omega \mapsto T_\omega^\varepsilon$ had an uncountable range for $\mathbb{P}$-a.e. $\omega\in\Omega$, then there is no way to ensure $\omega\mapsto \mathcal{L}_\omega^\varepsilon$ is $\mathbb{P}$-continuous into $L(\BV)$ since the Perron-Frobenius operators are discrete in the $\BV$ operator norm topology, allowing us to force arbitrary sets to be measurable. Consequently, \cite{Horan} points out that any set in $\Omega$ could be made measurable if $\mathcal{L}_\omega^\varepsilon$ is measurable, which is generally untrue.      
\end{remark}
\noindent
\noindent
We aim to prove that \jp{in $L^1$}, the accumulation points of the functions spanning the top and second Oseledets spaces are always the same so that they admit a limit as $\varepsilon\to 0$. We begin by showing that the Lyapunov exponents and Oseledets projections of $\mathcal{L}_\omega^\varepsilon$ depend continuously on perturbations. 
\begin{lemma}
 Let $\{(\Omega,\mathcal{F},\mathbb{P},\sigma,\BV(I),\mathcal{L}^\varepsilon)\}_{\varepsilon\geq 0}$ be a \jp{family} of random dynamical systems of paired metastable maps $T_\omega^\varepsilon:I\to I$ satisfying \hyperref[list:I1]{\textbf{(I1)}}-\hyperref[list:I6]{\textbf{(I6)}} and \hyperref[list:P1]{\textbf{(P1)}}-\hyperref[list:P7]{\textbf{(P7)}}. Then the Lyapunov exponents and Oseledets projections of $\mathcal{L}_\omega^\varepsilon$ converge to those of $\mathcal{L}^0$ as $\varepsilon\to 0$. In particular, for all $\varepsilon>0$ sufficiently small:
 \begin{itemize}
     \item[(a)] the top Oseledets space of $\mathcal{L}_\omega^\varepsilon$ is one-dimensional, and spanned by $\phi_\omega^\varepsilon\in \BV(I)$ with associated Lyapunov exponent $\lambda_1^\varepsilon=0$ of multiplicity one; and
     \item[(b)] the second Oseledets space of $\mathcal{L}_\omega^\varepsilon$ is one-dimensional, and spanned by $\psi_\omega^\varepsilon\in \BV(I)$ with associated Lyapunov exponent $\lambda_2^\varepsilon<0$ of multiplicity one.
 \end{itemize} 
     Further, as $\varepsilon\to 0$, $\lambda_2^\varepsilon\to0$, and the spaces spanned by $\phi_\omega^{{\varepsilon}}$ and $\psi_\omega^{{\varepsilon}}$ converge {uniformly} over $\omega\in\Omega$ to the plane spanned by $\phi_L$ and $\phi_R$ away from a $\mathbb{P}$-null set. Any accumulation point of the sequences $\phi_\omega^{{\varepsilon}}$ and $\psi_\omega^{{\varepsilon}}$, denoted $\phi_\omega^0$ and $\psi_\omega^0$, respectively, lie in $\BV(I)$ and satisfy $\phi_\omega^0,\psi_\omega^0 \in \mathrm{span} \{\phi_L, \phi_R\}$ where this convergence holds in $L^1$, and for $\mathbb{P}$-a.e. $\omega\in\Omega$.   
 \label{lem:cont}
\end{lemma}
\begin{remark}
Due to \hyperref[list:P5]{\textbf{(P5)}}, the operator $\mathcal{L}_\omega^\varepsilon$ satisfies a uniform Lasota-Yorke inequality across both $\omega\in\Omega$ and $\varepsilon>0$. Further, since $\BV$ is compactly embedded in $L^1$,  for $\mathbb{P}$-a.e. $\omega\in\Omega$ one can choose a sequence of values $\tilde{\varepsilon}$ converging to $0$ such that the sequences $\phi_\omega^{\tilde{\varepsilon}}$ and $\psi_\omega^{\tilde{\varepsilon}}$ converge in $L^1$ to some function $\phi_\omega^0$ and $\psi_\omega^0$.
\label{rem:subseq}
\end{remark}
\begin{proof}[Proof of \lem{lem:cont}]
Equip the Banach space $(\BV(I),||\cdot||_{\BV(I)})$ with weak norm $||\cdot||_{L^1(\leb)}$. Due to \hyperref[list:P1]{\textbf{(P1)}}, $\{(\Omega,\mathcal{F},\mathbb{P},\sigma, \BV(I),\mathcal{L}^\varepsilon)\}_{\varepsilon\geq 0}$ is a \jp{family} of $\mathbb{P}$-continuous random dynamical systems of paired metastable maps satisfying the conditions mentioned in \Sec{sec:map+pert}. Therefore, thanks to \rem{rmk:free-splitting}, \rem{rem:LY0}, \rem{rem:p1-imp-trip}, and the above discussion, conditions \hyperref[list:I1]{\textbf{(I1)}}-\hyperref[list:I6]{\textbf{(I6)}} and \hyperref[list:P1]{\textbf{(P1)}}-\hyperref[list:P7]{\textbf{(P7)}} ensure that \cite[Theorem A]{Crimmins} applies and asserts that $\mathcal{L}_\omega^\varepsilon$ has an Oseledets splitting, and further, the Lyapunov exponents and Oseledets projections of $\mathcal{L}_\omega^\varepsilon$ converge to those of $\mathcal{L}^0$ as $\varepsilon\to 0$. \jp{We note that in our setting, due to \rem{rmk:free-splitting}, $(\Omega,\mathcal{F},\mathbb{P},\sigma,\BV(I),\mathcal{L}^0)$ has a hyperbolic Oseledets splitting of index $d=2$.} Observe that due to \hyperref[list:I4]{\textbf{(I4)}}, any accumulation point of the sequences $\phi_\omega^{{\varepsilon}}$ and $\psi_\omega^{{\varepsilon}}$, denoted $\phi_\omega^0$ and $\psi_\omega^0$, are invariant under $\mathcal{L}^0$ and therefore may be expressed as a linear combination of $\phi_L$ and $\phi_R$. Thus, by \cite[Theorem A]{Crimmins}, as $\varepsilon\to 0$, $\lambda_2^\varepsilon\to0$, and the spaces spanned by $\phi_\omega^{{\varepsilon}}$ and $\psi_\omega^{{\varepsilon}}$ converge {uniformly} over $\omega\in\Omega$ to the plane spanned by $\phi_L$ and $\phi_R$ away from a $\mathbb{P}$-null set. Further, any accumulation point of the sequences $\phi_\omega^\varepsilon$ and $\psi_\omega^\varepsilon$ lie in $\BV(I)$, and satisfy $\phi_\omega^0,\psi_\omega^0 \in \mathrm{span} \{\phi_L, \phi_R\}$, where this convergence holds in $L^1$ and for $\mathbb{P}$-a.e. $\omega\in\Omega$. \\
\\ \noindent 
Fix $\varepsilon>0$ sufficiently small. We now address (a) and (b). By \cite[Theorem A]{Crimmins}, $\phi_\omega^\varepsilon,\psi_\omega^\varepsilon\in \BV(I)$. For (a), the fact that the top Oseledets space spanned by $\phi_\omega^\varepsilon$ is one-dimensional with $\lambda_1^\varepsilon=0$ follows from \hyperref[list:P6]{\textbf{(P6)}}. For (b), when $\varepsilon=0$, the top Oseledets space is two-dimensional and is spanned by the initially invariant densities, $\phi_L$ and $\phi_R$. From (a), since the top Oseledets space of $\mathcal{L}_\omega^\varepsilon$ is one-dimensional for $\varepsilon>0$ sufficiently small, and both the top and second Oseledets spaces depend continuously on perturbations, it follows that the Oseledets space spanned by $\psi_\omega^\varepsilon$ is one-dimensional. Finally, since $\lambda_2^\varepsilon$ is continuous in $\varepsilon$, converging to $0$ as $\varepsilon\to 0$, and \hyperref[list:P6]{\textbf{(P6)}} ensures that $\lambda_1^\varepsilon$ has multiplicity one, we may conclude that $\lambda_2^\varepsilon<0$.    
\end{proof}
    \begin{remark}
    If both the top and second Oseledets spaces of $\mathcal{L}_\omega^\varepsilon$ and $\mathcal{L}^0$ are one-dimensional, then \cite[Theorem A]{Crimmins} asserts that the sequences $\phi_\omega^{{\varepsilon}}$ and $\psi_\omega^{{\varepsilon}}$ converge in $L^1$ to $\phi_\omega^0$ and $\psi_\omega^0$, uniformly over $\omega\in\Omega$ away from a $\mathbb{P}$-null set. In our setting, since the limiting top Oseledets space is two-dimensional, we may deduce that the spaces spanned by $\phi_\omega^{{\varepsilon}}$ and $\psi_\omega^{{\varepsilon}}$ converge in $L^1$ to the plane spanned by $\phi_L$ and $\phi_R$ uniformly over $\omega\in\Omega$ away from a $\mathbb{P}$-null set. Without any further information, we cannot conclude that any convergent subsequence of $\phi_\omega^{{\varepsilon}}$ and $\psi_\omega^{{\varepsilon}}$ converge uniformly to the precise linear combination lying in $\mathrm{span} \{\phi_L, \phi_R\}$. The existence of such convergent subsequences is discussed in \rem{rem:subseq}.
\end{remark}

\section{Characterisation of the limiting invariant density}
\label{sec:inv-density}
In this section we identify the limiting invariant density for paired metastable systems. Recall that $\phi_\omega^\varepsilon$ denotes the function spanning the top Oseledets space of $\mathcal{L}_\omega^\varepsilon$. Due to \lem{lem:cont}, for $\mathbb{P}$-a.e. $\omega\in\Omega$, any accumulation of $\phi_\omega^{{\varepsilon}}$ may be expressed as a linear combination of $\phi_L$ and $\phi_R$ as $\varepsilon\to 0$. This section is dedicated to further characterising any accumulation point of $\phi_\omega^{{\varepsilon}}$, say $\phi_\omega^0=\lim_{\tilde{\varepsilon}\to 0}\phi_\omega^{\tilde{\varepsilon}}$, by explicitly computing the weights associated with the initially invariant densities, $\phi_L$ and $\phi_R$.\footnote{By \textit{weights} we mean the coefficients multiplied by $\phi_L$ and $\phi_R$ appearing in the statements of \thrm{thrm:A} and \thrm{thrm:B}.} 
\subsection{The weights}
We begin by computing an auxiliary limit. This is crucial when approximating the limiting invariant density for paired metastable systems. In particular, we will find in the proceeding subsection that this limit describes the weights associated with the linear combination of $\phi_L$ and $\phi_R$ appearing in the limiting invariant density.  \\
\\
\noindent
For $\star\in\{L,R\}$, let $\beta_\star\in L^\infty(\mathbb{P})$ (as in \hyperref[list:P4]{\textbf{(P4)}}). Consider the sequences of functions $(\pi_\omega^\varepsilon)_{\varepsilon>0}$ where
\begin{equation}
    \pi_{\omega}^\varepsilon:=\sum_{n=0}^\infty(\varepsilon\beta_{R,\sigma^{-n-1}\omega}+o_{\varepsilon\to 0}(\varepsilon))\prod_{k=0}^{n-1}(1-\varepsilon(\beta_{L,\sigma^{-k-1}\omega}+\beta_{R,\sigma^{-k-1}\omega})).
    \label{eqn:pweps}
\end{equation}
For a fixed $\varepsilon>0$, if $\int_{\Omega }\beta_{L,\omega}+\beta_{R,\omega}\, d\mathbb{P}(\omega)\neq 0$ then the tail end of \eqn{eqn:pweps} decays exponentially with $n\in\mathbb{N}$ and thus converges up to order $o_{\varepsilon\to 0}(\varepsilon)$. We are interested in determining whether $\lim_{\varepsilon\to 0}\pi_{\omega}^\varepsilon$ exists. To provide a positive answer to this question, we require the following result on moving averages. 
\begin{proposition}[{\cite[Theorem 1]{moving_average}}]
Let $(\Omega,\mathcal{F},\mathbb{P})$ be a probability space. Suppose that $\sigma:\Omega\to \Omega$ is as in \hyperref[list:P1]{\textbf{(P1)}}, and $f:\Omega\to\mathbb{R}$ is an $\mathcal{F}$-measurable function. If $a_n=O_{n\to\infty}(n)$, then for $\mathbb{P}$-a.e. $\omega\in\Omega$ 
$$\lim_{n\to \infty} \frac{1}{n}\sum_{k=a_n}^{a_n+n-1}f_{\sigma^k\omega}=\int_{\Omega} f_\omega \, d\mathbb{P}(\omega).$$
\label{prop:movingavg}
\end{proposition}
\noindent
\jp{\begin{remark}
    Prior to stating the main result of this section, we motivate why one may expect \eqn{eqn:pweps} to play a role in computing the weights for the limiting invariant density. Indeed, for $\star\in\{L,R\}$, let $\beta_{\star}\in  L^\infty(\mathbb{P})$ (as in \hyperref[list:P4]{\textbf{(P4)}}). Consider the 2-state Markov chains in random environments driven by $\sigma:\Omega\to \Omega$ (as in \hyperref[list:P1]{\textbf{(P1)}}), with transition matrices $(P_\omega^\varepsilon)_{\omega\in\Omega}$ where
$$P_\omega^\varepsilon:=\begin{pmatrix} 1-\varepsilon\beta_{L,\omega} & \varepsilon\beta_{R,\omega} \\ \varepsilon\beta_{L,\omega} &  1-\varepsilon\beta_{R,\omega} \end{pmatrix}.$$
\jp{For $n\in\mathbb{N}$ and $\omega\in\Omega$ let $P_\omega^{\varepsilon\,(n)}:= P_{\sigma^{n-1}\omega}^\varepsilon\cdots P_{\sigma\omega}^\varepsilon P_\omega^\varepsilon$}. An inductive argument shows that
\fontsize{10}{10}
\begin{equation}
        P_{\sigma^{-n}\omega}^{\varepsilon\, (n)}= \begin{pmatrix}1- \varepsilon\sum_{k=0}^{n-1}\beta_{L,\sigma^{-k-1}\omega}\prod_{i=0}^{k-1}(1-\varepsilon\gamma_{\sigma^{-i-1}\omega}) & \varepsilon\sum_{k=0}^{n-1}\beta_{R,\sigma^{-k-1}\omega}\prod_{i=0}^{k-1}(1-\varepsilon\gamma_{\sigma^{-i-1}\omega}) \\ \varepsilon\sum_{k=0}^{n-1}\beta_{L,\sigma^{-k-1}\omega}\prod_{i=0}^{k-1}(1-\varepsilon\gamma_{\sigma^{-i-1}\omega}) & 1- \varepsilon\sum_{k=0}^{n-1}\beta_{R,\sigma^{-k-1}\omega}\prod_{i=0}^{k-1}(1-\varepsilon\gamma_{\sigma^{-i-1}\omega})\end{pmatrix}
        \label{eqn:Pn}
    \end{equation}\normalsize
 where $\gamma_\omega:= \beta_{L,\omega} + \beta_{R,\omega}$. Here, we highlight the similarities between the entries of the matrix $\lim_{n\to \infty } P_{\sigma^{-n}\omega}^{\varepsilon\, (n)}$ and \eqn{eqn:pweps}. In turn, and as we will see in \rem{rem:markov-connection}, computing $\lim_{\varepsilon\to 0} \pi_\omega^\varepsilon$ yields the limiting invariant measure of the matrix cocycle $(P_\omega^\varepsilon)_{\omega\in\Omega}$ as $\varepsilon\to 0$ as well. \label{rem:lim-motiv}
\end{remark}}
\noindent
For the sake of presentation, we focus on the limiting behaviour of $\pi_{\sigma \omega}^\varepsilon$ as opposed to $\pi_{\omega}^\varepsilon$ (defined in \eqn{eqn:pweps}). Further, when the upper and lower indices of products and summations take non-integer values, we interpret them as the integer part.
\begin{lemma}
    Let $(\Omega,\mathcal{F},\mathbb{P})$ be a probability space. For $\star\in\{L,R\}$, let $\beta_\star\in L^\infty(\mathbb{P})$ satisfy $\int_{\Omega} \beta_{L,\omega}+\beta_{R,\omega}\, d\mathbb{P}(\omega)\neq 0$. If $\sigma:\Omega\to \Omega$ is as in \hyperref[list:P1]{\textbf{(P1)}}, then for $\mathbb{P}$-a.e. $\omega\in\Omega$
    \begin{align}
              \lim_{\varepsilon\to 0}\pi_{\omega}^\varepsilon&=\frac{\int_{\Omega} \beta_{R,\omega}\, d\mathbb{P}(\omega)}{\int_{\Omega}\beta_{L,\omega}+\beta_{R,\omega} \, d\mathbb{P}(\omega)}.
              \label{eqn:pwlim}
    \end{align}
    \label{lem:weights}
\end{lemma}
\begin{proof}
    The argument is divided into several steps. \jp{In \Step{step:tonep}, we illustrate the behaviour of the products appearing in \eqn{eqn:pweps}; this result is used repeatedly throughout the proof. In \Step{step:B} and \Step{stepp:B}, we establish a bound for the tail of \eqn{eqn:pweps} when $\varepsilon > 0$ is small. This is juxtaposed with \Step{step:C}, where we provide a bound for the initial terms of \eqn{eqn:pweps}. In \Step{step:D} and \Step{step:E}, we derive sharp estimates for the remaining terms of the sum. Taken together, these results lead to the proof of \eqn{eqn:pwlim} in \Step{step:LU} and \Step{step:ULlims}.} \\
    \\
    \noindent We compute $\lim_{\varepsilon\to 0}\pi_{\sigma\omega}^\varepsilon$ from which we can deduce \eqn{eqn:pwlim}. In what follows we recall that $\gamma_\omega= \beta_{L,\omega}+\beta_{R,\omega}$.   
\jp{\begin{step}
Fix $p>0$. For $\varepsilon>0$ sufficiently small, $n\in\mathbb{N}$, $t>0$ and $\mathbb{P}$-a.e. $\omega\in\Omega$
\begin{equation}
    \prod_{k=0}^{n+\left(\frac{t}{\varepsilon}\right)^p-1}(1-\varepsilon \gamma_{\sigma^{-k}\omega})=e^{(n\varepsilon+t^p\varepsilon^{1-p})\left(-\int_{\Omega}\gamma_s \, d\mathbb{P}(s)+o_{\omega,\varepsilon\to 0}(1) \right)}.
    \label{eqn:tonep}
\end{equation}
\label{step:tonep}    
\end{step}
\begin{proof}
Recall that $\gamma\in L^\infty(\mathbb{P})$. Take $\varepsilon>0$ sufficiently small such that $\varepsilon ||\gamma||_{L^\infty(\mathbb{P})}<1$. Then, for all $p>0$, $n\in\mathbb{N}$ and $t>0$
\begin{align*}
    \frac{1}{n\varepsilon+t^p\varepsilon^{1-p}}\log \left(\prod_{k=0}^{n+\left(\frac{t}{\varepsilon}\right)^p-1}(1-\varepsilon \gamma_{\sigma^{-k}\omega}) \right)&= \frac{1}{n\varepsilon+t^p\varepsilon^{1-p}}\sum_{k=0}^{n+\left(\frac{t}{\varepsilon}\right)^p-1}\log(1-\varepsilon \gamma_{\sigma^{-k}\omega})\\
    &=\frac{1}{\varepsilon}\left(n+\left(\frac{t}{\varepsilon}\right)^p\right)^{-1}\sum_{k=0}^{n+\left(\frac{t}{\varepsilon}\right)^p-1}\sum_{j=1}^\infty -\frac{(\varepsilon\gamma_{\sigma^{-k}\omega})^{j}}{j}\\
    &=\left(n+\left(\frac{t}{\varepsilon}\right)^p\right)^{-1}\sum_{k=0}^{n+\left(\frac{t}{\varepsilon}\right)^p-1}\sum_{j=1}^\infty -\frac{\varepsilon^{j-1}\gamma_{\sigma^{-k}\omega}^{j}}{j}\\
    &=\left(n+\left(\frac{t}{\varepsilon}\right)^p\right)^{-1}\sum_{k=0}^{n+\left(\frac{t}{\varepsilon}\right)^p-1} -\gamma_{\sigma^{-k}\omega}+O_{\varepsilon\to 0}(\varepsilon).
\end{align*}
Since $\gamma \in  L^\infty(\mathbb{P})$, by Birkhoff's ergodic theorem, for $\mathbb{P}$-a.e. $\omega\in\Omega$  
\begin{align*}
    \phantom{\iff}& \frac{1}{n\varepsilon+t^p\varepsilon^{1-p}}\log \left(\prod_{k=0}^{n+\left(\frac{t}{\varepsilon}\right)^p-1}(1-\varepsilon \gamma_{\sigma^{-k}\omega}) \right) = -\int_{\Omega} \gamma_s \, d\mathbb{P}(s) +o_{\omega,\varepsilon\to 0}(1). 
    \end{align*}
     Hence for $\mathbb{P}$-a.e. $\omega\in\Omega$  
    $$\prod_{k=0}^{n+\left(\frac{t}{\varepsilon}\right)^p-1}(1-\varepsilon \gamma_{\sigma^{-k}\omega})=e^{(n\varepsilon+t^p\varepsilon^{1-p})\left(-\int_{\Omega}\gamma_s \, d\mathbb{P}(s)+o_{\omega, \varepsilon\to 0}(1) \right)}.$$
\end{proof} }   
\noindent
\Step{step:tonep} describes how the products behave for various values of $p>0$, including the rate of convergence. This proves to be useful as illustrated in the following steps.
\begin{step}
 For all $\varepsilon>0$ sufficiently small, $t>0$ and $\mathbb{P}$-a.e. $\omega\in\Omega$, \jp{there exists $M>0$ such that} 
    \begin{align*}     B_{t,\omega}^\varepsilon&:=\varepsilon\sum_{n=\frac{t}{\varepsilon}}^\infty (\beta_{R,\sigma^{-n}\omega}+o_{\varepsilon\to 0}(1))\prod_{k=0}^{n-1}(1-\varepsilon\gamma_{\sigma^{-k}\omega})\\
        &\leq (M+o_{\varepsilon\to 0}(1))\frac{\varepsilon e^{t\left(-\int_\Omega\gamma_s \, d\mathbb{P}(s)+o_{\omega,\varepsilon\to 0}(1)\right)}}{1-e^{\varepsilon\left(-\int_{\Omega} \gamma_s \, d\mathbb{P}(s)+o_{\omega,\varepsilon\to 0}(1)\right)}}.
    \end{align*}
    
        \label{step:B}   
\end{step}
\begin{proof}
 Since $\beta_R\in L^\infty(\mathbb{P})$, there exists an $M>0$ such that $\beta_{R,\omega}\leq M$ for $\mathbb{P}$-a.e. $\omega\in\Omega$. Thus
        \begin{align*}
            B_{t,\omega}^\varepsilon&\leq (M+o_{\varepsilon\to 0}(1))\varepsilon\sum_{n=0}^\infty \prod_{k=0}^{n+\frac{t}{\varepsilon}-1}(1-\varepsilon\gamma_{\sigma^{-k}\omega})\\
            &=(M+o_{\varepsilon\to 0}(1))\varepsilon\sum_{n=0}^\infty e^{(n\varepsilon+t)\left(-\int_{\Omega}\gamma_s \, d\mathbb{P}(s)+o_{\omega,\varepsilon\to 0}(1) \right)},
        \end{align*}
        where we have utilised \Step{step:tonep} in the last equality \jp{with $p=1$}. \jp{Finally,} we compute the resulting geometric series. We emphasise that this series converges for all $\varepsilon>0$ chosen small such that the error in the exponent can be controlled. Therefore
        \begin{align*}
            B_{t,\omega}^\varepsilon&\leq(M+o_{\varepsilon\to 0}(1))\varepsilon e^{t\left(-\int_{\Omega}\gamma_s \, d\mathbb{P}(s)+o_{\omega,\varepsilon\to 0}(1)\right)} \sum_{n=0}^\infty e^{n\varepsilon\left(-\int_{\Omega}\gamma_s \, d\mathbb{P}(s)+o_{\omega,\varepsilon\to 0}(1)\right)} \\
            &=(M+o_{\varepsilon\to 0}(1))\frac{\varepsilon e^{t\left(-\int_\Omega \gamma_s \, d\mathbb{P}(s)+o_{\omega,\varepsilon\to 0}(1)\right)}}{1-e^{\varepsilon\left(-\int_{\Omega} \gamma_s \, d\mathbb{P}(s)+o_{\omega,\varepsilon\to 0}(1)\right)}}.
        \end{align*}
    
\end{proof}

\begin{step}
For any $t>0$ and $\mathbb{P}$-a.e. $\omega\in\Omega$ 
    $$\limsup_{\varepsilon\to 0}B_{t,\omega}^\varepsilon\leq \frac{2M}{\int_\Omega \gamma_s \, d\mathbb{P}(s)e^{\frac{t}{2}\int_\Omega \gamma_s \, d\mathbb{P}(s)}}.$$
    \label{stepp:B}    
\end{step}
\begin{proof}
Recall that from \Step{step:B} we found that for $\mathbb{P}$-a.e. $\omega\in\Omega$ 
    $$B_{t,\omega}^\varepsilon\leq (M+o_{\varepsilon\to 0}(1))\frac{\varepsilon e^{t\left(-\int_\Omega \gamma_s \, d\mathbb{P}(s)+o_{\omega,\varepsilon\to 0}(1)\right)}}{1-e^{\varepsilon\left(-\int_{\Omega} \gamma_s \, d\mathbb{P}(s)+o_{\omega,\varepsilon\to 0}(1)\right)}}.$$
    For each $\omega\in\Omega$ and $\varepsilon>0$ sufficiently small, $-\int_\Omega \gamma_s \, d\mathbb{P}(s)+o_{\omega,\varepsilon\to 0}(1)$ is bounded above by $-\frac{1}{2}\int_\Omega \gamma_s \, d\mathbb{P}(s)$. Therefore, we find that for \jp{$\mathbb{P}$-a.e. $\omega\in\Omega$}
        \begin{align*}
        \limsup_{\varepsilon\to 0}B_{t,\omega}^\varepsilon&\leq \lim_{\varepsilon\to 0}(M+o_{\varepsilon\to 0}(1))\frac{\varepsilon e^{-\frac{t}{2}\int_\Omega \gamma_s \, d\mathbb{P}(s)}}{1-e^{-\frac{\varepsilon}{2}\int_{\Omega} \gamma_s \, d\mathbb{P}(s)}}   \\
        &=\frac{2M}{\int_\Omega \gamma_s \, d\mathbb{P}(s)e^{\frac{t}{2}\int_\Omega \gamma_s \, d\mathbb{P}(s)}}.
        \end{align*}        
\end{proof}
 \begin{step}
     For all $\varepsilon>0$ sufficiently small, $t>0$ and $\mathbb{P}$-a.e. $\omega\in\Omega$ 
 \begin{align*}  C_{t,\omega}^\varepsilon&:=\varepsilon\sum_{n=0}^{\sqrt{\frac{t}{\varepsilon}}-1} (\beta_{R,\sigma^{-n}\omega}+o_{\varepsilon\to 0}(1))\prod_{k=0}^{n-1}(1-\varepsilon\gamma_{\sigma^{-k}\omega})\\
        &\leq {\varepsilon}\sum_{n=0}^{\sqrt{\frac{t}{\varepsilon}}-1} (\beta_{R,\sigma^{-n}\omega}+o_{\varepsilon\to 0}(1)).
 \end{align*}
 \label{step:C}
 \end{step}
\begin{proof}
This follows immediately from the fact that for $\varepsilon>0$ sufficiently small we have for each $k=0,\dots,n-1$ and $\mathbb{P}$-a.e. $\omega\in\Omega$, $0\leq(1-\varepsilon\gamma_{\sigma^{-k}\omega})\leq 1$. Thus for $\mathbb{P}$-a.e. $\omega\in\Omega$
\begin{align*}
C_{t,\omega}^\varepsilon&\leq\varepsilon\sum_{n=0}^{\sqrt{\frac{t}{\varepsilon}}-1} (\beta_{R,\sigma^{-n}\omega}+o_{\varepsilon\to 0}(1)).
\end{align*}
\end{proof}

\begin{step}
For all $\varepsilon>0$ sufficiently small, $t>0$ and $\mathbb{P}$-a.e. $\omega\in\Omega$
 \begin{align*}  D_{t,\omega}^\varepsilon&:=\varepsilon\sum_{n=\sqrt{\frac{t}{\varepsilon}}}^{{\frac{t}{\varepsilon}}-1} (\beta_{R,\sigma^{-n}\omega}+o_{\varepsilon\to 0}(1))\prod_{k=0}^{n-1}(1-\varepsilon\gamma_{\sigma^{-k}\omega})\\
        &\leq {{\varepsilon}{}}e^{\sqrt{t\varepsilon}\left(-\int_\Omega \gamma_s \, d\mathbb{P}(s) 
+o_{\omega,\varepsilon\to 0}(1) \right)}\sum_{n=0}^{{\frac{t}{\varepsilon}}-1} \Big((\beta_{R,\sigma^{-n-\sqrt{\frac{t}{\varepsilon}}}\omega}+o_{\varepsilon\to 0}(1))\\
&\quad \times e^{n\varepsilon\left(-\int_\Omega \gamma_s \, d\mathbb{P}(s)+o_{\omega,\varepsilon\to 0}(1)\right)}\Big)\\
&=:e^{\sqrt{t\varepsilon}\left(-\int_\Omega \gamma_s \, d\mathbb{P}(s) 
+o_{\omega,\varepsilon\to 0}(1) \right)} E_{t,\omega}^\varepsilon.
\end{align*}
\label{step:D}    
\end{step}
\begin{proof}
This follows by reindexing the partial sum and applying \Step{step:tonep} \jp{with $p=\frac{1}{2}$}. Indeed
\begin{align}
    D_{t,\omega}^\varepsilon&= \varepsilon\sum_{n=0}^{{\frac{t}{\varepsilon}}-\sqrt{\frac{t}{\varepsilon}}-1} (\beta_{R,\sigma^{-n-\sqrt{\frac{t}{\varepsilon}}}\omega}+o_{\varepsilon\to 0}(1))\prod_{k=0}^{n+\sqrt{\frac{t}{\varepsilon}}-1}(1-\varepsilon\gamma_{\sigma^{-k}\omega}) \nonumber \\
    &\leq \varepsilon\sum_{n=0}^{{\frac{t}{\varepsilon}}-1} (\beta_{R,\sigma^{-n-\sqrt{\frac{t}{\varepsilon}}}\omega}+o_{\varepsilon\to 0}(1))\prod_{k=0}^{n+\sqrt{\frac{t}{\varepsilon}}-1}(1-\varepsilon\gamma_{\sigma^{-k}\omega}) \label{eqn:Dup} \\ 
    &= \varepsilon\sum_{n=0}^{{\frac{t}{\varepsilon}}-1}(\beta_{R,\sigma^{-n-\sqrt{\frac{t}{\varepsilon}}}\omega}+o_{\varepsilon\to 0}(1))e^{(n\varepsilon+\sqrt{t\varepsilon})\left(-\int_\Omega \gamma_s \, d\mathbb{P}(s)+o_{\omega,\varepsilon\to 0}(1)\right)} \nonumber \\
    &={{\varepsilon}{}}e^{\sqrt{t\varepsilon}\left(-\int_\Omega \gamma_s \, d\mathbb{P}(s) 
+o_{\omega,\varepsilon\to 0}(1) \right)}\sum_{n=0}^{{\frac{t}{\varepsilon}}-1} \Big((\beta_{R,\sigma^{-n-\sqrt{\frac{t}{\varepsilon}}}\omega}+o_{\varepsilon\to 0}(1))\nonumber \\
&\quad \times e^{n\varepsilon\left(-\int_\Omega \gamma_s \, d\mathbb{P}(s)+o_{\omega,\varepsilon\to 0}(1)\right)}\Big).\nonumber
\end{align}
Letting 
\begin{equation}
    E_{t,\omega}^\varepsilon:=\varepsilon\sum_{n=0}^{{\frac{t}{\varepsilon}}-1} (\beta_{R,\sigma^{-n-\sqrt{\frac{t}{\varepsilon}}}\omega}+o_{\varepsilon\to 0}(1))e^{n\varepsilon\left(-\int_\Omega \gamma_s \, d\mathbb{P}(s)+o_{\omega,\varepsilon\to 0}(1)\right)}
    \label{eqn:E}
\end{equation}
we obtain our result.     
\end{proof}
\noindent
Finding sharp estimates for $E_{t,\omega}^\varepsilon$ is crucial in determining the limiting behaviour of $\pi_{\sigma\omega}^\varepsilon$ when $\varepsilon$ is small. Rather than studying the entire sum \eqn{eqn:E}, we fix some $\delta>0$ and split \eqn{eqn:E} into $t/\delta$ sums, each containing $\delta/\varepsilon$ terms, as illustrated in \fig{fig:sumsplit}. 
\begin{figure}[ht]
  \centering
  \begin{tikzpicture}[scale=0.85, transform shape]
    \foreach \x/\y in {0/0, 2/\frac{\delta}{\varepsilon}-1, 4/\cdots, 6/\frac{1}{\varepsilon}-\frac{\delta}{\varepsilon},8/\frac{1}{\varepsilon}-1, 10/\cdots, 12/\cdots, 14/\frac{t}{\varepsilon}-1}
      \draw[ultra thick] (\x,0.25) -- (\x,-0.25) node[below]{$\y$};
    \draw[ultra thick] (0,0)--(14,0);
    
    \draw[thick,decorate,decoration={brace,amplitude=10pt}] (0,0.35) -- (2,0.35)
      node[midway,above,yshift=9pt] {$\frac{\delta}{\varepsilon}$ terms};
      \draw[thick,decorate,decoration={brace,amplitude=10pt}] (6,0.35) -- (8,0.35)
      node[midway,above,yshift=9pt] {$\frac{\delta}{\varepsilon}$ terms};
  \end{tikzpicture}
    \caption{Splitting of $E_{t,\omega}^\varepsilon$}
    \label{fig:sumsplit}
\end{figure}
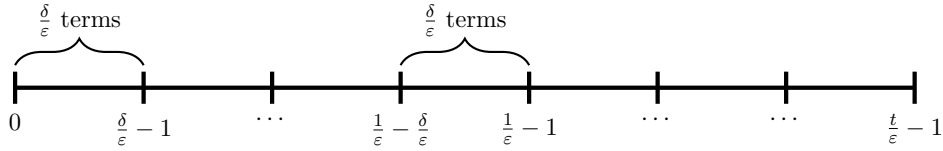
\noindent
\\
In other words, we consider the splitting of the sum in the following way: 
$$\sum_{n=0}^{\frac{t}{\varepsilon}-1}\bullet = \sum_{n=0}^{\frac{\delta}{\varepsilon}-1}\bullet+\sum_{n=\frac{\delta}{\varepsilon}}^{\frac{2\delta}{\varepsilon}-1}\bullet+\cdots +\sum_{n=\frac{t-\delta}{\varepsilon}}^{\frac{t}{\varepsilon}-1}\bullet.$$
With this, we obtain bounds on $E_{t,\omega}^{\varepsilon}$ (which may depend on $\delta$) as follows.

\begin{step}
Fix $t,\delta>0$ and let $\varepsilon>0$ be sufficiently small, then for $\mathbb{P}$-a.e. $\omega\in \Omega$
    \begin{align}
    E_{t,\omega}^{\varepsilon} &\leq \varepsilon \sum_{j=0}^{\frac{t}{\delta}-1}e^{j\delta\left(-\int_\Omega \gamma_s \, d\mathbb{P}(s) + o_{\omega,\varepsilon\to 0}(1)\right)}\sum_{n=\frac{j\delta}{\varepsilon}}^{\frac{(j+1)\delta}{\varepsilon}-1}(\beta_{R,\sigma^{-n-\sqrt{\frac{t}{\varepsilon}}}\omega}+o_{\omega,\varepsilon\to 0}(1))\\
    E_{t,\omega}^{\varepsilon} &\geq \varepsilon e^{(\delta-\varepsilon)\left(-\int_\Omega \gamma_s \, d\mathbb{P}(s) + o_{\omega,\varepsilon\to 0}(1)\right)} \sum_{j=0}^{\frac{t}{\delta}-1}\Big(e^{j\delta\left(-\int_\Omega \gamma_s \, d\mathbb{P}(s) + o_{\omega,\varepsilon\to 0}(1)\right)}\nonumber \\
    &\quad  \times \sum_{n=\frac{j\delta}{\varepsilon}}^{\frac{(j+1)\delta}{\varepsilon}-1}(\beta_{R,\sigma^{-n-\sqrt{\frac{t}{\varepsilon}}}\omega}+o_{\omega,\varepsilon\to 0}(1))\Big),
    \end{align}
     where $E_{t,\omega}^\varepsilon$ is as in \eqn{eqn:E}. 
    \label{step:E}    
\end{step}
\begin{proof}
 Observe that for all $\alpha,\beta\in\mathbb{R}^+$, if $\alpha\leq n \leq \beta$ then for sufficiently small $\varepsilon>0$, 
    \begin{equation}
        e^{\beta \varepsilon\left(-\int_\Omega \gamma_s \, d\mathbb{P}(s) + o_{\omega,\varepsilon\to 0}(1)\right)}\leq e^{n\varepsilon\left(-\int_\Omega \gamma_s \, d\mathbb{P}(s) + o_{\omega,\varepsilon\to 0}(1)\right)}\leq e^{\alpha \varepsilon\left(-\int_\Omega \gamma_s \, d\mathbb{P}(s) + o_{\omega,\varepsilon\to 0}(1)\right)}.
        \label{eqn:expest}
    \end{equation}
    The result follows by bounding each sum containing $\delta/\varepsilon$ terms through this estimate. Indeed, following the splitting of the series illustrated in \fig{fig:sumsplit}
    \begin{align*}
        E_{t,\omega}^\varepsilon&=\varepsilon \sum_{j=0}^{\frac{t}{\delta}-1}\sum_{n=\frac{j\delta}{\varepsilon}}^{\frac{(j+1)\delta}{\varepsilon}-1}(\beta_{R,\sigma^{-n-\sqrt{\frac{t}{\varepsilon}}}\omega}+o_{\varepsilon\to 0}(1))e^{n\varepsilon\left(-\int_\Omega \gamma_s \, d\mathbb{P}(s) + o_{\omega,\varepsilon\to 0}(1)\right)}.
    \end{align*}
    Thus, through \eqn{eqn:expest} we have 
    \begin{align*}
        E_{t,\omega}^\varepsilon&\leq \varepsilon \sum_{j=0}^{\frac{t}{\delta}-1}\sum_{n=\frac{j\delta}{\varepsilon}}^{\frac{(j+1)\delta}{\varepsilon}-1}(\beta_{R,\sigma^{-n-\sqrt{\frac{t}{\varepsilon}}}\omega}+o_{\varepsilon\to 0}(1))e^{\frac{j\delta}{\varepsilon}\varepsilon\left(-\int_\Omega \gamma_s \, d\mathbb{P}(s) + o_{\omega,\varepsilon\to 0}(1)\right)}\\
        &=\varepsilon \sum_{j=0}^{\frac{t}{\delta}-1}e^{j\delta\left(-\int_\Omega \gamma_s \, d\mathbb{P}(s) + o_{\omega,\varepsilon\to 0}(1)\right)}\sum_{n=\frac{j\delta}{\varepsilon}}^{\frac{(j+1)\delta}{\varepsilon}-1}(\beta_{R,\sigma^{-n-\sqrt{\frac{t}{\varepsilon}}}\omega}+o_{\varepsilon\to 0}(1)).
    \end{align*}
    For the lower estimate,
    \begin{align*}
        E_{t,\omega}^\varepsilon&\geq \varepsilon \sum_{j=0}^{\frac{t}{\delta}-1}\sum_{n=\frac{j\delta}{\varepsilon}}^{\frac{(j+1)\delta}{\varepsilon}-1}(\beta_{R,\sigma^{-n-\sqrt{\frac{t}{\varepsilon}}}\omega}+o_{\varepsilon\to 0}(1))e^{\left(\frac{(j+1)\delta}{\varepsilon}-1\right)\varepsilon\left(-\int_\Omega \gamma_s \, d\mathbb{P}(s) + o_{\omega,\varepsilon\to 0}(1)\right)}\\
        &=\varepsilon \sum_{j=0}^{\frac{t}{\delta}-1} e^{\left((j+1)\delta-{\varepsilon}\right)\left(-\int_\Omega \gamma_s \, d\mathbb{P}(s) + o_{\omega,\varepsilon\to 0}(1)\right)}
 \sum_{n=\frac{j\delta}{\varepsilon}}^{\frac{(j+1)\delta}{\varepsilon}-1}(\beta_{R,\sigma^{-n-\sqrt{\frac{t}{\varepsilon}}}\omega}+o_{\omega,\varepsilon\to 0}(1))\\
        &=\varepsilon e^{(\delta-\varepsilon)\left(-\int_\Omega \gamma_s \, d\mathbb{P}(s) + o_{\omega,\varepsilon\to 0}(1)\right)} \sum_{j=0}^{\frac{t}{\delta}-1}\Big( e^{j\delta\left(-\int_\Omega \gamma_s \, d\mathbb{P}(s) + o_{\omega,\varepsilon\to 0}(1)\right)}\\
        & \quad \times\sum_{n=\frac{j\delta}{\varepsilon}}^{\frac{(j+1)\delta}{\varepsilon}-1}(\beta_{R,\sigma^{-n-\sqrt{\frac{t}{\varepsilon}}}\omega}+o_{\varepsilon\to 0}(1))\Big).
    \end{align*}    
\end{proof}
\noindent
The previous steps now provide sharp estimates on $\pi_{\sigma\omega}^\varepsilon$ (defined on a different fibre in \eqn{eqn:pweps}) which allow us to compute $\lim_{\varepsilon\to 0} \pi_{\sigma\omega}^\varepsilon$ explicitly.
\begin{step}
    For all $t,\delta>0$ and $\mathbb{P}$-a.e. $\omega\in\Omega$ 
    \begin{align}
   L_{t,\delta}\leq \lim_{\varepsilon\to 0}\pi_{\sigma\omega}^\varepsilon \leq U_{t,\delta}
   \label{eqn:ineqp}
\end{align}
where
\begin{align}
     L_{t,\delta}&:=\delta e^{-\delta \int_\Omega \gamma_s \, d\mathbb{P}(s)}\int_\Omega \beta_{R,s} \, d\mathbb{P}(s)\frac{1-e^{-t\int_\Omega \gamma_s \, d\mathbb{P}(s)}}{1-e^{-\delta\int_\Omega \gamma_s \, d\mathbb{P}(s) }} \label{eqn:L}\\
     U_{t,\delta}&:= \delta \int_\Omega \beta_{R,s} \, d\mathbb{P}(s) \frac{1-e^{-t\int_\Omega \gamma_s \, d\mathbb{P}(s)}}{1-e^{-\delta\int_\Omega \gamma_s \, d\mathbb{P}(s) }}+\frac{2M}{\int_\Omega \gamma_s \, d\mathbb{P}(s)e^{\frac{t}{2}\int_\Omega \gamma_s \, d\mathbb{P}(s)}}. \label{eqn:U}
\end{align}
\label{step:LU}
\end{step}
\begin{proof}
    Firstly we focus on the upper estimate. Observe that we can express \eqn{eqn:pweps} as 
\begin{align*}
    \pi_{\sigma\omega}^\varepsilon&=\varepsilon\sum_{n=0}^{\infty} (\beta_{R,\sigma^{-n}\omega}+o_{\varepsilon\to 0}(1))\prod_{k=0}^{n-1}(1-\varepsilon\gamma_{\sigma^{-k}\omega})\\
    &=\varepsilon\sum_{n=0}^{\sqrt{\frac{t}{\varepsilon}}-1} (\beta_{R,\sigma^{-n}\omega}+o_{\varepsilon\to 0}(1))\prod_{k=0}^{n-1}(1-\varepsilon\gamma_{\sigma^{-k}\omega})+\varepsilon\sum_{n=\sqrt{\frac{t}{\varepsilon}}}^{{\frac{t}{\varepsilon}}-1}\Big( (\beta_{R,\sigma^{-n}\omega}+o_{\varepsilon\to 0}(1))\\
    &\quad \times \prod_{k=0}^{n-1}(1-\varepsilon\gamma_{\sigma^{-k}\omega})\Big)
    +\varepsilon\sum_{n=\frac{t}{\varepsilon
    }}^{\infty} (\beta_{R,\sigma^{-n}\omega}+o_{\varepsilon\to 0}(1))\prod_{k=0}^{n-1}(1-\varepsilon\gamma_{\sigma^{-k}\omega})\\
    &=C_{t,\omega}^\varepsilon+D_{t,\omega}^\varepsilon+B_{t,\omega}^\varepsilon.
\end{align*}
Thus, by Step \ref{step:B}, \ref{step:C}, \ref{step:D} and \ref{step:E}
\begin{align*}
 \pi_{\sigma\omega}^\varepsilon&\leq \varepsilon\sum_{n=0}^{\sqrt{\frac{t}{\varepsilon}}-1} (\beta_{R,\sigma^{-n}\omega}+o_{\varepsilon\to 0}(1)) + e^{\sqrt{t\varepsilon}\left(-\int_\Omega \gamma_s \, d\mathbb{P}(s) 
+o_{\omega,\varepsilon\to 0}(1) \right)} E_{t,\omega}^\varepsilon \\
&\quad +(M+o_{\varepsilon\to 0}(1))\frac{\varepsilon e^{t\left(-\int_\Omega \gamma_s \, d\mathbb{P}(s)+o_{\omega,\varepsilon\to 0}(1)\right)}}{1-e^{\varepsilon\left(-\int_{\Omega} \gamma_s \, d\mathbb{P}(s)+o_{\omega,\varepsilon\to 0}(1)\right)}} \\
&\leq \sqrt{t\varepsilon}\sqrt{\frac{\varepsilon}{t}}\sum_{n=0}^{\sqrt{\frac{t}{\varepsilon}}-1} (\beta_{R,\sigma^{-n}\omega}+o_{\varepsilon\to 0}(1)) +\varepsilon e^{\sqrt{t\varepsilon}\left(-\int_\Omega \gamma_s \, d\mathbb{P}(s) 
+o_{\omega,\varepsilon\to 0}(1) \right)}\\
& \quad \times \bigg(\sum_{j=0}^{\frac{t}{\delta}-1}e^{j\delta\left(-\int_\Omega \gamma_s \, d\mathbb{P}(s) + o_{\omega,\varepsilon\to 0}(1)\right)}\sum_{n=\frac{j\delta}{\varepsilon}}^{\frac{(j+1)\delta}{\varepsilon}-1}(\beta_{R,\sigma^{-n-\sqrt{\frac{t}{\varepsilon}}}\omega}+o_{\varepsilon\to 0}(1))\bigg)\\
&\quad+(M+o_{\varepsilon\to 0}(1))\frac{\varepsilon e^{t\left(-\int_\Omega \gamma_s \, d\mathbb{P}(s)+o_{\omega,\varepsilon\to 0}(1)\right)}}{1-e^{\varepsilon\left(-\int_{\Omega} \gamma_s \, d\mathbb{P}(s)+o_{\omega,\varepsilon\to 0}(1)\right)}}.
\end{align*}
We may now take $\varepsilon\to 0$. In doing so we utilise Birkhoff's ergodic theorem on the first term, \prop{prop:movingavg} on the second term and \Step{stepp:B} on the third term to find that for all $t,\delta>0$ and $\mathbb{P}$-a.e. $\omega\in\Omega$
\begin{align*}
    \lim_{\varepsilon\to 0}\pi_{\sigma\omega}^\varepsilon &\leq 0+ \lim_{\varepsilon\to 0}\delta e^{\sqrt{t\varepsilon}\left(-\int_\Omega \gamma_s \, d\mathbb{P}(s) 
+o_{\omega,\varepsilon\to 0}(1) \right)}\sum_{j=0}^{\frac{t}{\delta}-1}\bigg(e^{j\delta\left(-\int_\Omega \gamma_s \, d\mathbb{P}(s) + o_{\omega,\varepsilon\to 0}(1)\right)} \\
&\quad \times \frac{\varepsilon}{\delta} \sum_{n=\frac{j\delta}{\varepsilon}}^{\frac{(j+1)\delta}{\varepsilon}-1}(\beta_{R,\sigma^{-n-\sqrt{\frac{t}{\varepsilon}}}\omega}+o_{\varepsilon\to 0}(1))\bigg)
+\frac{2M}{\int_\Omega \gamma_s \, d\mathbb{P}(s)e^{\frac{t}{2}\int_\Omega \gamma_s \, d\mathbb{P}(s)}}\\
&=\delta \int_\Omega \beta_{R,s} \, d\mathbb{P}(s) \sum_{j=0}^{\frac{t}{\delta}-1}e^{-j\delta \int_{\Omega} \gamma_s \, d\mathbb{P}(s)}+\frac{2M}{\int_\Omega \gamma_s \, d\mathbb{P}(s)e^{\frac{t}{2}\int_\Omega \gamma_s \, d\mathbb{P}(s)}}\\
&=\delta \int_\Omega \beta_{R,s} \, d\mathbb{P}(s) \frac{1-e^{-t\int_\Omega \gamma_s \, d\mathbb{P}(s)}}{1-e^{-\delta\int_\Omega \gamma_s \, d\mathbb{P}(s) }}+\frac{2M}{\int_\Omega \gamma_s \, d\mathbb{P}(s)e^{\frac{t}{2}\int_\Omega \gamma_s \, d\mathbb{P}(s)}}\\
&= U_{t,\delta}.
\end{align*}
For the lower estimate, we note that 
\begin{align*}
    \pi_{\sigma\omega}^\varepsilon&\geq \varepsilon\sum_{n=\sqrt{\frac{t}{\varepsilon}}}^{{\frac{t}{\varepsilon}}+\sqrt{\frac{t}{\varepsilon}}-1} (\beta_{R,\sigma^{-n}\omega}+o_{\varepsilon\to 0}(1))\prod_{k=0}^{n-1}(1-\varepsilon\gamma_{\sigma^{-k}\omega})\\
    &=\varepsilon\sum_{n=0}^{{\frac{t}{\varepsilon}}-1} (\beta_{R,\sigma^{-n-\sqrt{\frac{t}{\varepsilon}}}\omega}+o_{\varepsilon\to 0}(1))\prod_{k=0}^{n+\sqrt{\frac{t}{\varepsilon}}-1}(1-\varepsilon\gamma_{\sigma^{-k}\omega}).
\end{align*}
The previous line is precisely the upper bound of $D_{t,\omega}^\varepsilon$ obtained in \Step{step:D} (see \eqn{eqn:Dup}). Therefore, using the lower bound from \Step{step:E}
\begin{align*}
    \pi_{\sigma\omega}^\varepsilon&\geq e^{\sqrt{t\varepsilon}\left(-\int_\Omega \gamma_s \, d\mathbb{P}(s) 
+o_{\omega,\varepsilon\to 0}(1) \right)} E_{t,\omega}^\varepsilon\\
&\geq \varepsilon e^{(\sqrt{t\varepsilon}+\delta-\varepsilon)\left(-\int_\Omega \gamma_s \, d\mathbb{P}(s) + o_{\omega,\varepsilon\to 0}(1)\right)} \sum_{j=0}^{\frac{t}{\delta}-1}\bigg(e^{j\delta\left(-\int_\Omega \gamma_s \, d\mathbb{P}(s) + o_{\omega,\varepsilon\to 0}(1)\right)}\\
&\quad \times \sum_{n=\frac{j\delta}{\varepsilon}}^{\frac{(j+1)\delta}{\varepsilon}-1}(\beta_{R,\sigma^{-n-\sqrt{\frac{t}{\varepsilon}}}\omega}+o_{\varepsilon\to 0}(1))\bigg)\\
&=\delta e^{(\sqrt{t\varepsilon}+\delta-\varepsilon)\left(-\int_\Omega \gamma_s \, d\mathbb{P}(s) + o_{\omega,\varepsilon\to 0}(1)\right)} \sum_{j=0}^{\frac{t}{\delta}-1}\bigg(e^{j\delta\left(-\int_\Omega \gamma_s \, d\mathbb{P}(s) + o_{\omega,\varepsilon\to 0}(1)\right)} \\
&\quad \times \frac{\varepsilon}{\delta}\sum_{n=\frac{j\delta}{\varepsilon}}^{\frac{(j+1)\delta}{\varepsilon}-1}(\beta_{R,\sigma^{-n-\sqrt{\frac{t}{\varepsilon}}}\omega}+o_{\varepsilon\to 0}(1))\bigg).
\end{align*}
We emphasise the similarity between the above line and the upper bound of $D_{t,\omega}^\varepsilon$ obtained from both \Step{step:D} and \Step{step:E}. In particular, these bounds are identical up to a factor of $\exp((\delta-\varepsilon)(-\int_\Omega \gamma_s \, d\mathbb{P}(s)+o_{\omega,\varepsilon\to 0}(1)))$. Therefore, by the same argument as in the derivation of $U_{t,\delta}$ (applying \prop{prop:movingavg}), for $\mathbb{P}$-a.e. $\omega\in\Omega$
\begin{align*}
    \lim_{\varepsilon\to 0} \pi_{\sigma\omega}^\varepsilon &\geq \delta e^{-\delta \int_\Omega \gamma_s \, d\mathbb{P}(s)}\int_\Omega \beta_{R,s} \, d\mathbb{P}(s)\frac{1-e^{-t\int_\Omega \gamma_s \, d\mathbb{P}(s)}}{1-e^{-\delta\int_\Omega \gamma_s \, d\mathbb{P}(s) }}\\
    &=L_{t,\delta}.
\end{align*}
\end{proof}
\noindent
We proceed by taking $\delta$ sufficiently small and $t$ sufficiently large. Since $L_{t,\delta}$ and $U_{t,\delta}$ (defined in \eqn{eqn:L} and \eqn{eqn:U}, respectively) are independent of $\omega\in\Omega$, we replace the integration variable $s$ with $\omega$ in the final step to align with the statement of \lem{lem:weights}.  
\begin{step}
    The functions $L_{t,\delta}$ and $U_{t,\delta}$ defined in \eqn{eqn:L} and \eqn{eqn:U} satisfy $$\lim_{(t,\delta)\to(\infty,0)}L_{t,\delta}=\lim_{(t,\delta)\to(\infty,0)}U_{t,\delta}=\frac{\int_{\Omega} \beta_{R,\omega}\, d\mathbb{P}(\omega)}{\int_{\Omega}\beta_{L,\omega}+\beta_{R,\omega} \, d\mathbb{P}(\omega)}$$
    \label{step:ULlims}
\end{step}
\begin{proof}
We focus first on taking the appropriate limits of \eqn{eqn:L}. Indeed
\begin{align*}
\lim_{(t,\delta)\to(\infty,0)}L_{t,\delta}&=  \int_\Omega \beta_{R,\omega} \, d\mathbb{P}(\omega) \lim_{t\to\infty}\frac{1-e^{-t\int_\Omega \gamma_\omega \, d\mathbb{P}(\omega)}}{\int_\Omega \gamma_\omega \, d\mathbb{P}(\omega)}  \\
&=\frac{\int_{\Omega} \beta_{R,\omega}\, d\mathbb{P}(\omega)}{\int_{\Omega}\beta_{L,\omega}+\beta_{R,\omega} \, d\mathbb{P}(\omega)}.   
\end{align*}
Then, by taking limits of \eqn{eqn:U}
\begin{align*}
    \lim_{(t,\delta)\to(\infty,0)} U_{t,\delta}&=  \lim_{t\to\infty}\left( \int_\Omega \beta_{R,\omega} \, d\mathbb{P}(\omega) \frac{1-e^{-t\int_\Omega \gamma_\omega \, d\mathbb{P}(\omega)}}{\int_\Omega \gamma_\omega \, d\mathbb{P}(\omega)}+\frac{2M}{\int_\Omega \gamma_\omega \, d\mathbb{P}(\omega)e^{\frac{t}{2}\int_\Omega \gamma_\omega \, d\mathbb{P}(\omega)}}\right)  \\
&=\frac{\int_{\Omega} \beta_{R,\omega}\, d\mathbb{P}(\omega)}{\int_{\Omega}\beta_{L,\omega}+\beta_{R,\omega} \, d\mathbb{P}(\omega)}.
\end{align*}
\end{proof}
\noindent
The results of \Step{step:tonep}-\Step{step:ULlims} give us \eqn{eqn:pwlim} for $\mathbb{P}$-a.e. $\omega\in\Omega$.
\end{proof}

\begin{remark}
Provided that $\int_\Omega \beta_{L,\omega} + \beta_{R,\omega}\, d\mathbb{P}(\omega)\neq 0$ and $\beta_{\star}\in  L^\infty(\mathbb{P})$ for $\star\in\{L,R\}$, \lem{lem:weights} provides us with the limiting invariant measures (as $\varepsilon\to 0$) of 2-state Markov chains in random environments driven by $\sigma:\Omega\to \Omega$ (as in \hyperref[list:P1]{\textbf{(P1)}}), with transition matrices $(P_\omega^\varepsilon)_{\omega\in\Omega}$ \jp{given in \rem{rem:lim-motiv}}. For a fixed $\varepsilon>0$, the columns of the matrix $\lim_{n\to \infty}P_{\sigma^{-n}\omega}^{\varepsilon\, (n)}$ give rise to the random invariant measure of $(P_\omega^\varepsilon)_{\omega\in\Omega}$. We are interested in the behaviour of this measure as $\varepsilon\to 0$. \jp{Due to \eqn{eqn:Pn}, and} recalling the similarities between \eqn{eqn:pweps} and the elements of the matrix $\lim_{n\to \infty}P_{\sigma^{-n}\omega}^{\varepsilon\, (n)}$, \lem{lem:weights} shows that for $\mathbb{P}$-a.e. $\omega\in\Omega$
\begin{equation}
        \lim_{\varepsilon\to 0}\lim_{n\to \infty}P_{\sigma^{-n}\omega}^{\varepsilon\, (n)}=\begin{pmatrix}
        \frac{\int_{\Omega} \beta_{R,\omega}\, d\mathbb{P}(\omega)}{\int_{\Omega}\beta_{L,\omega}+\beta_{R,\omega} \, d\mathbb{P}(\omega)} & \frac{\int_{\Omega} \beta_{R,\omega}\, d\mathbb{P}(\omega)}{\int_{\Omega}\beta_{L,\omega}+\beta_{R,\omega} \, d\mathbb{P}(\omega)}\\
    \frac{\int_{\Omega} \beta_{L,\omega}\, d\mathbb{P}(\omega)}{\int_{\Omega}\beta_{L,\omega}+\beta_{R,\omega} \, d\mathbb{P}(\omega)} & \frac{\int_{\Omega} \beta_{L,\omega}\, d\mathbb{P}(\omega)}{\int_{\Omega}\beta_{L,\omega}+\beta_{R,\omega} \, d\mathbb{P}(\omega)}
    \end{pmatrix}.
    \label{eqn:Plim}
    \end{equation}
    Hence, the columns of the matrix in \eqn{eqn:Plim} correspond to the limiting random invariant measure of the matrix cocycle $(P_\omega^\varepsilon)_{\omega\in\Omega}$ as $\varepsilon\to 0$. We refer the reader to, for example \cite[Section 4]{Yin_Zhang_CTMC}, for a similar treatment of this problem in the setting of continuous time Markov chains. As we will find in \Sec{sec:LID}, particularly in the proof of \thrm{thrm:phi_lims}, one can show that \eqn{eqn:Plim} does not only hold for $\mathbb{P}$-a.e. $\omega\in\Omega$, but in fact uniformly over $\omega\in\Omega$ away from a $\mathbb{P}$-null set.
    \label{rem:markov-connection}
\end{remark}

\subsection{The limiting invariant density}
\label{sec:LID}
Using \lem{lem:weights} we can now characterise the limiting random invariant density. 
\begin{theorem}
 Let $\{(\Omega,\mathcal{F},\mathbb{P},\sigma,\BV(I),\mathcal{L}^\varepsilon)\}_{\varepsilon\geq 0}$ be a \jp{family} of random dynamical systems of paired metastable maps $T_\omega^\varepsilon:I\to I$ satisfying \hyperref[list:I1]{\textbf{(I1)}}-\hyperref[list:I6]{\textbf{(I6)}} and \hyperref[list:P1]{\textbf{(P1)}}-\hyperref[list:P7]{\textbf{(P7)}}. If $\int_{\Omega}\beta_{L,\omega} +\beta_{R,\omega}\, d\mathbb{P}(\omega)\neq 0$ then as $\varepsilon\to 0$
    $$\phi_\omega^\varepsilon \stackrel{L^1}{\to} \phi_{}^0:= \frac{\int_{\Omega} \beta_{R,\omega}\, d\mathbb{P}(\omega)}{\int_{\Omega}\beta_{L,\omega}+\beta_{R,\omega} \, d\mathbb{P}(\omega)} \phi_L +\frac{\int_{\Omega} \beta_{L,\omega}\, d\mathbb{P}(\omega)}{\int_{\Omega}\beta_{L,\omega}+\beta_{R,\omega} \, d\mathbb{P}(\omega)}\phi_R$$
    uniformly over $\omega\in\Omega$ away from a $\mathbb{P}$-null set. 
\label{thrm:phi_lims}
\end{theorem}
\jp{
\begin{remark}
Since our approach relies on perturbative techniques involving a weak norm, $L^1$, and a strong norm, $BV$, as pioneered by Keller and Liverani in \cite{stab_spec} and generalised to the random setting in \cite{Crimmins}, 
our work establishes convergence in the weak norm, $L^1$.
One could aim to strengthen this convergence by working in other function spaces, for instance fractional Sobolev spaces as in \cite{GTQ_semiinvert}. 
\end{remark}
}\noindent
Due to \hyperref[list:I4]{\textbf{(I4)}}, any accumulation point of $\phi_\omega^{{\varepsilon}}$ can be expressed as a (possibly random) convex combination of $\phi_L$ and $\phi_R$. To prove \thrm{thrm:phi_lims} we must first understand the asymptotic behaviour of the measure of the holes created under perturbation. Unless otherwise mentioned, throughout the remainder of this section we assume that the hypotheses of \thrm{thrm:phi_lims} are satisfied. \\
\\
Thanks to \lem{lem:cont} and \rem{rem:subseq}, for $\mathbb{P}$-a.e. $\omega\in\Omega$ we can choose a sequence of values $\tilde{\varepsilon}$ converging to $0$ such that $\phi_\omega^{\tilde{\varepsilon}}$ converges in $L^1$ to some function, which we denote by $\phi_\omega^0$. We first prove that the subsequence $\phi_\omega^{\tilde{\varepsilon}}$ converges uniformly over the holes away from a $\mathbb{P}$-null set. This was done in \cite[Lemma 3.14]{BS_rand} for finite $\Omega$ and i.i.d. driving. Upon close inspection of the proofs, the condition that $\Omega$ is finite is used to guarantee that the mapping $\omega\mapsto T_\omega^{\tilde{\varepsilon}}$ is finite for all ${\tilde{\varepsilon}}\geq 0$. This is satisfied in our setting due to \hyperref[list:P1]{\textbf{(P1)}}. Adapting \cite[Lemma 3.14]{BS_rand} to our notation we have the following. 
\begin{lemma}
In the setting of \thrm{thrm:phi_lims}, let $\phi_\omega^0$ be an accumulation point of $\phi_\omega^{{\varepsilon}}$ along the subsequence $\tilde{\varepsilon}\to 0$. That is, assume that
\begin{equation}
 \phi_\omega^0 = \lim_{\tilde{\varepsilon}\to 0} \phi_\omega^{\tilde{\varepsilon}}.   
 \label{eqn:lim-subseq}
\end{equation}
Then there exists $p:\Omega\to [0,1]$ such that
\begin{itemize}
    \item[(a)] $\phi_\omega^0 = p_\omega \phi_L + (1-p_\omega)\phi_R$,
    \item[(b)] $\lim_{{\tilde{\varepsilon}}\to 0} \esssup_{\omega\in\Omega} \sup_{x\in H_{L,\omega}^{\tilde{\varepsilon}}}|\phi_\omega^{\tilde{\varepsilon}}(x) - p_\omega \phi_L(x)|=0$,
    \item[(c)] $\lim_{{\tilde{\varepsilon}}\to 0} \esssup_{\omega\in\Omega} \sup_{x\in H_{R,\omega}^{\tilde{\varepsilon}}}|\phi_\omega^{\tilde{\varepsilon}}(x) - (1-p_\omega) \phi_R(x)|=0$.
\end{itemize}
\label{lem:unif_conv}
\end{lemma}
\begin{proof}
    Identical to that of \cite[Lemma 3.14]{BS_rand} replacing the assumption of finite $\Omega$ with $\omega\mapsto T_\omega^{\tilde{\varepsilon}}$ being finite. \jp{Although \cite[Lemma 3.14]{BS_rand} assumes that $\sigma$ is the left shift on the full shift over a finite alphabet, the proof relies only on the finiteness of the range and does not exploit any specific properties of the driving dynamics, so, with \hyperref[list:P1]{\textbf{(P1)}}, it extends directly to our setting}. We emphasise that \cite{BS_rand} assumes the initial system has minimal expansion greater than 2, which is used to obtain a uniform Lasota-Yorke inequality for the perturbed system. Thanks to \rem{rem:LY0}, this is satisfied in our setting by enforcing \hyperref[list:I1]{\textbf{(I1)}}, \hyperref[list:I2]{\textbf{(I2)}} and \hyperref[list:P5]{\textbf{(P5)}}.   
\end{proof}\noindent
With this, we have a better understanding of how the sequence $\phi_\omega^{\tilde{\varepsilon}}$ behaves over the holes. Further, we know that any accumulation point of $\phi_\omega^{{\varepsilon}}$ is a convex combination of the initially invariant densities. \\
\\
\noindent
Through \lem{lem:unif_conv} we construct a convergent subsequence of $\phi_\omega^\varepsilon$ converging in $L^1$ on $I_\star$ (for $\star\in\{L,R\}$), and uniformly over $\omega\in\Omega$ away from a $\mathbb{P}$-null set.
\begin{lemma}
    In the setting of \thrm{thrm:phi_lims}, if $p_\omega$ is as in \lem{lem:unif_conv}, then for $\mathbb{P}$-a.e. $\omega\in\Omega$
    \begin{align}
         p_\omega^{\tilde{\varepsilon}}:=\mu_\omega^{\tilde{\varepsilon}}(I_L)&= p_\omega  +o_{{\tilde{\varepsilon}}\to 0}(1), \label{eqn:measIL}\\
    1-p_\omega^{\tilde{\varepsilon}}:=\mu_\omega^{\tilde{\varepsilon}}(I_R)&= 1-p_\omega +o_{{\tilde{\varepsilon}}\to 0}(1). \label{eqn:measIR}
    \end{align}
    \begin{proof}
     We prove that \eqn{eqn:measIL} holds from which one can obtain \eqn{eqn:measIR}. For all $\tilde{\varepsilon}>0$, let $\hat{\phi}_\omega^{\tilde{\varepsilon}}:=\hat{p}_\omega^{\tilde{\varepsilon}}\phi_L+(1-\hat{p}_\omega^{\tilde{\varepsilon}})\phi_R$ be the best approximation of $\phi_\omega^{\tilde\varepsilon}$ in $\mathrm{span}\{\phi_L,\phi_R\}$ along the subsequence $\tilde{\varepsilon}\to 0$, \jp{that is also a density}. That is, let $\hat{p}_\omega^{\tilde{\varepsilon}}$ be such that $$||\phi_\omega^{\tilde{\varepsilon}}-\hat{\phi}_\omega^{\tilde{\varepsilon}}||_{L^1(\leb)}:=\inf_{\theta \in [0,1]}||\phi_\omega^{\tilde{\varepsilon}}-\theta \phi_L-(1-\theta)\phi_R||_{L^1(\leb)}.$$
     By \lem{lem:cont}, $\|\phi_\omega^{\tilde{\varepsilon}}- \hat{\phi}_\omega^{\tilde{\varepsilon}}\|_{L^1(\leb)}=o_{\tilde{\varepsilon}\to 0}(1)$. Thus,
     \begin{align}
        o_{\tilde{\varepsilon}\to 0}(1) &= \int_{I_L} |\phi_\omega^{\tilde{\varepsilon}}-\hat{p}_\omega^{\tilde{\varepsilon}} \phi_L|\, \dleb(x)\nonumber\\
        &\geq \left| \int_{I_L}\phi_\omega^{\tilde{\varepsilon}}\, \dleb(x) - \hat{p}_\omega^{\tilde{\varepsilon}} \int_{I_L} \phi_L\, \dleb(x) \right| \nonumber\\
        &=|p_\omega^{\tilde{\varepsilon}}-\hat{p}_\omega^{\tilde{\varepsilon}}| \label{eqn:p-phat}
     \end{align}
     Further,
     \begin{align}
         \sup_{x\in H_{L,\omega}^{\tilde{\varepsilon}}}|\hat{\phi}_\omega^{\tilde{\varepsilon}}(x) - p_\omega \phi_L(x)|&= \sup_{x\in H_{L,\omega}^{\tilde{\varepsilon}}}|(\hat{p}_\omega^{\tilde{\varepsilon}} - p_\omega )\phi_L(x)|\nonumber\\
         &=|\hat{p}_\omega^{\tilde{\varepsilon}} - p_\omega |\sup_{x\in H_{L,\omega}^{\tilde{\varepsilon}}}|\phi_L(x)|. \label{eqn:wwlo}
     \end{align}
    Since $\hat{\phi}_\omega^{\tilde{\varepsilon}}$ is the best approximation of $\phi_\omega^{\tilde{\varepsilon}}$ in $\mathrm{span}\{\phi_L,\phi_R\}$ \jp{that is also a density}, due to \lem{lem:unif_conv}(b), $$\sup_{x\in H_{L,\omega}^{\tilde{\varepsilon}}}|\hat{\phi}_\omega^{\tilde{\varepsilon}}(x) - p_\omega \phi_L(x)|=o_{\tilde{\varepsilon}\to 0}(1).$$ Recall that by \hyperref[list:P3]{\textbf{(P3)}}, as $\tilde{\varepsilon}\to 0$, $H_{L,\omega}^{\tilde{\varepsilon}}$ converges (in the Hausdorff metric) to the set of infinitesimal holes $H^0\cap I_{L}$ uniformly over $\omega\in\Omega$ away from a $\mathbb{P}$-null set. This, together with conditions \hyperref[list:I5]{\textbf{(I5)}} and \hyperref[list:I6]{\textbf{(I6)}} imply that there exists a constant $K>0$ such that $\sup_{x\in H_{L,\omega}^{\tilde{\varepsilon}}}|\phi_L(x)|=K+o_{\tilde{\varepsilon}\to 0}(1)$. Therefore, from \eqn{eqn:wwlo},
    \begin{align*}
        o_{\tilde{\varepsilon}\to 0}(1) =\sup_{x\in H_{L,\omega}^{\tilde{\varepsilon}}}|\hat{\phi}_\omega^{\tilde{\varepsilon}}(x) - p_\omega \phi_L(x)| = |\hat{p}_\omega^{\tilde{\varepsilon}} - p_\omega |(K+o_{\tilde{\varepsilon}\to 0}(1)).
    \end{align*}
    Thus $|\hat{p}_\omega^{\tilde{\varepsilon}} - p_\omega |=o_{\tilde{\varepsilon}\to 0}(1)$, which due to \eqn{eqn:p-phat} implies that $p_\omega^{\tilde{\varepsilon}}=p_\omega + o_{\tilde{\varepsilon}\to 0}(1)$.\label{lem:p-1-p}
    \end{proof}
\end{lemma}
\begin{remark}
By \lem{lem:cont}, we know that any accumulation point of $\phi_\omega^\varepsilon$ along the subsequence $\tilde{\varepsilon}\to 0$ converges in $L^1$ to $\phi_\omega^0$ for $\mathbb{P}$-a.e. $\omega\in\Omega$. \lem{lem:p-1-p} reveals that this convergence is in fact uniform over $\omega\in\Omega$ away from a $\mathbb{P}$-null set. \label{rem:unif}  
\end{remark}
\begin{lemma}
In the setting of \thrm{thrm:phi_lims}, if $p_\omega$ is as in \lem{lem:unif_conv}, then for $\mathbb{P}$-a.e. $\omega\in\Omega$
\begin{align}
    \mu_\omega^{\tilde{\varepsilon}}(H_{L,\omega}^{\tilde{\varepsilon}})&= p_\omega \mu_L(H_{L,\omega}^{\tilde{\varepsilon}}) +o_{{\tilde{\varepsilon}}\to 0}(1)\mu_L(H_{L,\omega}^{\tilde{\varepsilon}}), \label{eqn:measWL}\\
    \mu_\omega^{\tilde{\varepsilon}}(H_{R,\omega}^{\tilde{\varepsilon}})&= (1-p_\omega) \mu_R(H_{R,\omega}^{\tilde{\varepsilon}}) +o_{{\tilde{\varepsilon}}\to 0}(1)\mu_R(H_{R,\omega}^{\tilde{\varepsilon}}). \label{eqn:measWR}
\end{align}
\label{lem:holes}
\end{lemma}
\begin{proof}
    Due to \hyperref[list:P6]{\textbf{(P6)}}, for each $\omega\in\Omega$, $d\mu_\omega^{\tilde{\varepsilon}}=\phi_\omega^{\tilde{\varepsilon}}\,\dleb$. We establish \eqn{eqn:measWL}, as \eqn{eqn:measWR} follows in an almost identical manner. Observe that for any set
$S_\omega^{\tilde{\varepsilon}}\subseteq I_L$ 
    \begin{align}
      \mu_\omega^{\tilde{\varepsilon}}(S_{\omega}^{\tilde{\varepsilon}})&= \int_{S_{\omega}^{\tilde{\varepsilon}}} \phi_\omega^{\tilde{\varepsilon}}\, \dleb(x) \nonumber \\
      &=p_\omega \int_{S_{\omega}^{\tilde{\varepsilon}}}  \phi_L\, \dleb(x)+\int_{S_{\omega}^{\tilde{\varepsilon}}} \phi_\omega^{\tilde{\varepsilon}} - p_\omega \phi_L \, \dleb(x) \nonumber \\
      &= p_\omega \mu_L(S_{\omega}^{\tilde{\varepsilon}})+O_{{\tilde{\varepsilon}}\to 0}\left(\esssup_{\omega\in\Omega}\sup_{x\in S_{\omega}^{\tilde{\varepsilon}}}|\phi_\omega^{\tilde{\varepsilon}} - p_\omega \phi_L|\right)\leb(S_{\omega}^{\tilde{\varepsilon}}). \label{eqn:meas_O}
    \end{align}
    For \eqn{eqn:measWL}, take $S_\omega^{\tilde{\varepsilon}} = H_{L,\omega}^{\tilde{\varepsilon}}$ in \eqn{eqn:meas_O} and consider the set of infinitesimal holes $H^0\cap I_L=\{h_L^1,\cdots, h_L^J\}$ where $J\in\mathbb{N}$.\footnote{Recall that the set of all points belonging to $H^0:=(T^0)^{-1}(\{b\})\setminus \{b\}$ are referred to as infinitesimal holes.} \jp{Thanks to \hyperref[list:P3]{\textbf{(P3)}}, one can find a finite collection of disjoint intervals $H_{L,\omega}^{1,\tilde{\varepsilon}},\dots,H_{L,\omega}^{J,\tilde{\varepsilon}}$ such that $H_{L,\omega}^{\tilde{\varepsilon}}=\cup_{i=1}^J H_{L,\omega}^{i,\tilde{\varepsilon}}$, where for each $i=1,\dots, J$, $H_{L,\omega}^{i,\tilde{\varepsilon}}\to h_L^i$ in the Hausdorff metric uniformly over $\omega\in\Omega$ away from a $\mathbb{P}$-null set.} Recall that by \hyperref[list:I5]{\textbf{(I5)}}, $\phi_L$ is continuous at all points in $H^0\cap I_L$. Thus by Lebesgue's differentiation theorem 
    \begin{equation}
        \frac{\mu_L(H_{L,\omega}^{\tilde{\varepsilon}})}{\leb(H_{L,\omega}^{\tilde{\varepsilon}})} = \sum_{i=1}^J \frac{\mu_L(H_{L,\omega}^{i,\tilde{\varepsilon}})}{\leb(H_{L,\omega}^{i,\tilde{\varepsilon}})} \frac{\leb(H_{L,\omega}^{i,\tilde{\varepsilon}})}{\leb(H_{L,\omega}^{\tilde{\varepsilon}})}=\sum_{i=1}^J (\phi_L(h_L^i)+o_{\tilde{\varepsilon} \to 0}(1)) \frac{\leb(H_{L,\omega}^{i,\tilde{\varepsilon}})}{\leb(H_{L,\omega}^{\tilde{\varepsilon}})}.
        \label{eqn:meas_ratio}
    \end{equation}
    Note that $\sum_{i=1}^J\frac{\leb(H_{L,\omega}^{i,\tilde{\varepsilon}})}{\leb(H_{L,\omega}^{\tilde{\varepsilon}})}=1$. Further, \hyperref[list:I6]{\textbf{(I6)}} asserts that for each $i=1,\dots, J$, $\phi_L(h_L^i)>0$, and thus one can bound \eqn{eqn:meas_ratio} uniformly above and below over $\omega\in\Omega$. Additionally, by \lem{lem:unif_conv}(b),
    $$\lim_{{\tilde{\varepsilon}}\to 0} \esssup_{\omega\in\Omega}\sup_{x\in H_{L,\omega}^{\tilde{\varepsilon}}}|\phi_\omega^{\tilde{\varepsilon}} - p_\omega \phi_L|=0.$$
   Therefore, utilising \eqn{eqn:meas_O}, for $\mathbb{P}$-a.e. $\omega\in\Omega$
    \begin{align*}   \mu_\omega^{\tilde{\varepsilon}}(H_{L,\omega}^{\tilde{\varepsilon}})
        &=p_\omega \mu_L(H_{L,\omega}^{\tilde{\varepsilon}})+o_{{\tilde{\varepsilon}}\to 0}(1){\mu_L(H_{L,\omega}^{\tilde{\varepsilon}})}.
    \end{align*}
\end{proof}
\begin{corollary}
In the setting of \thrm{thrm:phi_lims}, if $p_\omega$ is as in \lem{lem:unif_conv}, then for $\mathbb{P}$-a.e. $\omega\in\Omega$
\begin{align}
 \mu_\omega^{\tilde{\varepsilon}}(H_{L,\omega}^{\tilde{\varepsilon}})&= p_\omega {\tilde{\varepsilon}} \beta_{L,\omega} +o_{{\tilde{\varepsilon}}\to 0}({\tilde{\varepsilon}}), \label{eqn:measWLe}\\
    \mu_\omega^{\tilde{\varepsilon}}(H_{R,\omega}^{\tilde{\varepsilon}})&= (1-p_\omega) {\tilde{\varepsilon}} \beta_{R,\omega} +o_{{\tilde{\varepsilon}}\to 0}({\tilde{\varepsilon}}). \label{eqn:measWRe}   
\end{align}
\label{cor:holes}
\end{corollary}
\begin{proof}
We only show that \eqn{eqn:measWLe} holds since one can show \eqn{eqn:measWRe} holds in an almost identical manner. By \hyperref[list:P4]{\textbf{(P4)}}, for $\star \in \{L,R\}$, $\mu_\star(H_{\star,\omega}^{\tilde{\varepsilon}})={\tilde{\varepsilon}} \beta_{\star,\omega}+o_{{\tilde{\varepsilon}}\to 0}({\tilde{\varepsilon}})$ where $\beta_\star \in  L^\infty(\mathbb{P})$. Therefore, by \eqn{eqn:measWL} from \lem{lem:holes}
\begin{align*}
    \mu_\omega^{\tilde{\varepsilon}}(H_{L,\omega}^{\tilde{\varepsilon}})&= p_\omega \mu_L(H_{L,\omega}^{\tilde{\varepsilon}}) +o_{{\tilde{\varepsilon}}\to 0}(1)\mu_L(H_{L,\omega}^{\tilde{\varepsilon}})\\
    &=p_\omega ({\tilde{\varepsilon}} \beta_{L,\omega}+o_{{\tilde{\varepsilon}}\to 0}({\tilde{\varepsilon}})) +o_{{\tilde{\varepsilon}}\to 0}(1)({\tilde{\varepsilon}} \beta_{L,\omega}+o_{{\tilde{\varepsilon}}\to 0}({\tilde{\varepsilon}}))\\
    &=p_\omega {\tilde{\varepsilon}} \beta_{L,\omega} +o_{{\tilde{\varepsilon}}\to 0}({\tilde{\varepsilon}}).
\end{align*}
\end{proof}
\noindent
We are now ready to prove \thrm{thrm:phi_lims}.
\begin{proof}[Proof of \thrm{thrm:phi_lims}]
A combination of \lem{lem:cont}, \ref{lem:unif_conv}(a) and \lem{lem:p-1-p} asserts that any accumulation point of $\phi_\omega^{{\varepsilon}}$ along the subsequence $\tilde{\varepsilon}\to 0$ satisfies $\phi_\omega^0=\lim_{\tilde{\varepsilon}\to 0}\phi_\omega^{\tilde{\varepsilon}}$ (as in \eqn{eqn:lim-subseq}) and 
\begin{equation}
    \phi_\omega^{\tilde{\varepsilon}} \stackrel{L^1}{\to} \phi_\omega^0:= p_\omega \phi_L +(1-p_\omega)\phi_R
    \label{eqn:phiwepsL1}
\end{equation}
{uniformly} over $\omega\in\Omega$ away from a $\mathbb{P}$-null set. It remains to determine the weights $p_\omega$ and $1-p_\omega$ of any accumulation point of $\phi_\omega^{{\varepsilon}}$ along the subsequence $\tilde{\varepsilon}\to 0$. We illustrate that these weights are related to the readily computed limit from \lem{lem:weights}. We do this by utilising properties of the random absolutely continuous invariant measure $(\mu_\omega^{\tilde{\varepsilon}})_{\omega\in\Omega}$. Namely, since $(\mu_\omega^{\tilde{\varepsilon}})_{\omega\in\Omega}$ is a RACIM for $(T_\omega^{\tilde{\varepsilon}})_{\omega\in\Omega}$ (see \dfn{def:RIM-RID} and \rem{rem:RACIM})
\begin{align}
    p_\omega^{\tilde{\varepsilon}}=\mu_{ \omega}^{\tilde{\varepsilon}}(I_L)&=\mu_{\sigma^{-1}\omega}^{\tilde{\varepsilon}}((T_{\sigma^{-1}\omega}^{\tilde{\varepsilon}})^{-1}(I_L)) \nonumber\\
    &=\mu_{\sigma^{-1}\omega}^{\tilde{\varepsilon}}((I_L\setminus H_{L,\sigma^{-1}\omega}^{\tilde{\varepsilon}})\cup H_{R,\sigma^{-1}\omega}^{\tilde{\varepsilon}})) \nonumber\\
    &=\mu_{\sigma^{-1}\omega}^{\tilde{\varepsilon}}(I_L)-\mu_{\sigma^{-1}\omega}^{\tilde{\varepsilon}}(H_{L,\sigma^{-1}\omega}^{\tilde{\varepsilon}}) +\mu_{\sigma^{-1}\omega}^{\tilde{\varepsilon}}(H_{R,\sigma^{-1}\omega}^{\tilde{\varepsilon}}) \nonumber \\
    &\stackrel{(\star)}{=}p_{\sigma^{-1}\omega}^{\tilde{\varepsilon}} - p_{\sigma^{-1}\omega}  {\tilde{\varepsilon}} \beta_{L,{\sigma^{-1}\omega} }+ (1-p_{\sigma^{-1}\omega} ){\tilde{\varepsilon}} \beta_{R,{\sigma^{-1}\omega} } +o_{{\tilde{\varepsilon}}\to 0}({\tilde{\varepsilon}}) \nonumber \\
    &\stackrel{(\star\star)}{=}p_{\sigma^{-1}\omega}^{\tilde{\varepsilon}} (1-{\tilde{\varepsilon}}(\beta_{L,{\sigma^{-1}\omega} }+\beta_{R,{\sigma^{-1}\omega} }))+{\tilde{\varepsilon}}\beta_{R,{\sigma^{-1}\omega} }+o_{{\tilde{\varepsilon}}\to 0}(\tilde{\varepsilon}). \label{eqn:psw} 
\end{align}
At $(\star)$ we have used \eqn{eqn:measIL} and \cor{cor:holes}, and at $(\star\star)$ we have used \eqn{eqn:measIL} and \eqn{eqn:measIR}. From \eqn{eqn:psw} we know that for all $n\in\mathbb{Z}$ 
\begin{equation}
    p_{\sigma^{n}\omega}^{\tilde{\varepsilon}}= p_{\sigma^{n-1}\omega}^{\tilde{\varepsilon}}(1-{\tilde{\varepsilon}}(\beta_{L,{\sigma^{n-1}\omega}}+\beta_{R,{\sigma^{n-1}\omega}}))+{\tilde{\varepsilon}}\beta_{R,{\sigma^{n-1}\omega}}+o_{{\tilde{\varepsilon}}\to 0}({\tilde{\varepsilon}}).
    \label{eqn:psnw}
\end{equation}
Fix $t>0$. Inductively using \eqn{eqn:psnw}, we find that for any $K^{\tilde{\varepsilon}}_t\in\mathbb{N}$
\begin{equation}
    p_{\omega}^{\tilde{\varepsilon}}=p_{\sigma^{-K_t^{\tilde{\varepsilon}}}\omega}^{\tilde{\varepsilon}}\prod_{k=0}^{K_t^{\tilde{\varepsilon}}-1}(1-{\tilde{\varepsilon}}\gamma_{\sigma^{-k-1}\omega})+\sum_{n=0}^{K^\varepsilon_t-1}({\tilde{\varepsilon}}\beta_{R,\sigma^{-n-1}\omega}+o_{{\tilde{\varepsilon}}\to 0}({\tilde{\varepsilon}}))\prod_{k=0}^{n-1}(1-{\tilde{\varepsilon}}\gamma_{\sigma^{-k-1}\omega}).\label{eqn:Ksum}
\end{equation}
Choose $K^{\tilde{\varepsilon}}_t=\frac{t}{{\tilde{\varepsilon}}}$. We show that by taking ${\tilde{\varepsilon}}$ arbitrarily small and then $t$ sufficiently large, \eqn{eqn:Ksum} converges to the same limit obtained in \lem{lem:weights} for $\mathbb{P}$-a.e. $\omega\in\Omega$. We recall the sequences of functions $(\pi_\omega^{\tilde{\varepsilon}})_{{\tilde{\varepsilon}}>0}$ where $\pi_\omega^{\tilde{\varepsilon}}$ is given by \eqn{eqn:pweps} (replacing $\varepsilon$ with $\tilde{\varepsilon}$). Then 
\begin{align*}
    \left|\pi_\omega^{\tilde{\varepsilon}} - p_{\omega}^{\tilde{\varepsilon}} \right|&\leq\sum_{n=\frac{t}{{\tilde{\varepsilon}}}}^\infty({\tilde{\varepsilon}}\beta_{R,\sigma^{-n-1}\omega}+o_{{\tilde{\varepsilon}}\to 0}({\tilde{\varepsilon}}))\prod_{k=0}^{n-1}(1-{\tilde{\varepsilon}}(\beta_{L,\sigma^{-k-1}\omega}+\beta_{R,\sigma^{-k-1}\omega}))\\
    &\quad + p_{\sigma^{-\frac{t}{\tilde{\varepsilon}}}\omega}^{\tilde{\varepsilon}}\prod_{k=0}^{\frac{t}{\tilde{\varepsilon}}-1}(1-\tilde{\varepsilon}\gamma_{\sigma^{-k-1}})\\
    &\stackrel{(\star\star\star)}{\leq} (M+o_{\tilde{\varepsilon}\to 0}(1))\frac{\tilde{\varepsilon} e^{t\left(-\int_\Omega \gamma_s \, d\mathbb{P}(s)+o_{\omega,\tilde{\varepsilon}\to 0}(1)\right)}}{1-e^{\tilde{\varepsilon}\left(-\int_{\Omega} \gamma_s \, d\mathbb{P}(s)+o_{\omega,\tilde{\varepsilon}\to 0}(1)\right)}}+p_{\sigma^{-\frac{t}{\tilde{\varepsilon}}}\omega}^{\tilde{\varepsilon}}e^{t\left(-\int_{\Omega}\gamma_s \, d\mathbb{P}(s)+o_{\omega,\tilde{\varepsilon}\to 0}(1) \right)}
\end{align*}
where at $(\star\star\star)$ we have used \Step{step:B} and \Step{step:tonep} (\jp{with $p=1$}) from the proof of \lem{lem:weights} on the first and second term, respectively. Using \Step{stepp:B} from the proof of \lem{lem:weights} on the first term, and noting by \eqn{eqn:measIL} and \lem{lem:unif_conv} that $p_{\sigma^{-\frac{t}{\tilde{\varepsilon}}}\omega}^{\tilde{\varepsilon}} \leq 1+o_{\tilde{\varepsilon}\to 0}(1)$ for $\mathbb{P}$-a.e. $\omega\in\Omega$, by taking $\tilde{\varepsilon}\to 0$ and then $t\to\infty$, $\lim_{\tilde{\varepsilon}\to 0}\left|\pi_\omega^{\tilde{\varepsilon}} - p_{\omega}^{\tilde{\varepsilon}}\right|= 0$ for $\mathbb{P}$-a.e. $\omega\in\Omega$. By \lem{lem:weights}, since $\int_\Omega \beta_{L,\omega}+\beta_{R,\omega}\, d\mathbb{P}(\omega)\neq 0$, we get for $\mathbb{P}$-a.e. $\omega\in\Omega$
\begin{equation}
    \lim_{\tilde{\varepsilon}\to 0}\pi_\omega^{\tilde{\varepsilon}} = \frac{\int_{\Omega} \beta_{R,\omega}\, d\mathbb{P}(\omega)}{\int_{\Omega}\beta_{L,\omega}+\beta_{R,\omega} \, d\mathbb{P}(\omega)}.\label{eqn:limpiproof}
\end{equation}
Since $\lim_{\tilde{\varepsilon}\to 0}\left|\pi_\omega^{\tilde{\varepsilon}} - p_{\omega}^{\tilde{\varepsilon}}\right|= 0$ for $\mathbb{P}$-a.e. $\omega\in\Omega$, \eqn{eqn:limpiproof} coincides with $\lim_{\tilde{\varepsilon}\to 0}p_\omega^{\tilde{\varepsilon}}=p_\omega$. Hence, by \lem{lem:p-1-p} 
$$p_\omega:= \frac{\int_{\Omega} \beta_{R,\omega}\, d\mathbb{P}(\omega)}{\int_{\Omega}\beta_{L,\omega}+\beta_{R,\omega} \, d\mathbb{P}(\omega)}$$
and
$$1-p_\omega:= \frac{\int_{\Omega} \beta_{L,\omega}\, d\mathbb{P}(\omega)}{\int_{\Omega}\beta_{L,\omega}+\beta_{R,\omega} \, d\mathbb{P}(\omega)}$$
uniformly over $\omega\in\Omega$ away from a $\mathbb{P}$-null set, as claimed in the statement of \thrm{thrm:phi_lims}.
\end{proof}\noindent

\begin{remark}
    If we consider the case that $\int_\Omega \beta_{L,\omega}+\beta_{R,\omega}\, d\mathbb{P}(\omega)= 0$ in the statement of \thrm{thrm:phi_lims}, then $\beta_{L,\omega}=\beta_{R,\omega}=0$ for $\mathbb{P}$-a.e. $\omega\in\Omega$. Here, one would require finer information regarding the errors appearing in \lem{lem:holes} to investigate the accumulation points of the random invariant density.  
\end{remark}

\begin{remark}
An interesting feature of this class of paired metastable systems has been uncovered by \thrm{thrm:phi_lims}. That is, the limiting invariant density of such systems depends only on perturbations through averaged quantities (with respect to $\mathbb{P}$). Apart from this, it is independent of the driving system $\sigma$.    
\end{remark}
\begin{remark}
    Comparing \thrm{thrm:phi_lims} to \cite[Theorem 3.4]{BS_rand}, we explicitly compute the weights $p_\omega$ and $1-p_\omega$ associated with the limiting invariant density as opposed to studying the so-called \textit{limiting averaged holes ratio} ($l.a.h.r.$). In the case of paired metastable systems with driving that is not necessarily i.i.d., we find that the weights $p_\omega$ and $1-p_\omega$ may be computed in a similar way to \cite{BS_rand}. Namely,
    \begin{align*}
        l.a.h.r.&:=\lim_{\varepsilon\to 0} \frac{\int_{\Omega} \mu_{R}(H_{R,\omega}^\varepsilon)\, d\mathbb{P}(\omega)}{\int_{\Omega} \mu_{L}(H_{L,\omega}^\varepsilon)\, d\mathbb{P}(\omega)}\\
        &\stackrel{(\star)}{=}\lim_{\varepsilon\to 0} \frac{\int_{\Omega} \beta_{R,\omega} +o_{\varepsilon\to 0}(1)\, d\mathbb{P}(\omega)}{\int_{\Omega} \beta_{L,\omega} +o_{\varepsilon\to 0}(1)\, d\mathbb{P}(\omega)}\\
        &=\frac{\int_{\Omega} \beta_{R,\omega} \, d\mathbb{P}(\omega)}{\int_{\Omega} \beta_{L,\omega} \, d\mathbb{P}(\omega)}\\
        &= \frac{p_\omega}{1-p_\omega}.
    \end{align*}
    Note that at $(\star)$ we have used \hyperref[list:P4]{\textbf{(P4)}}. Taking $\frac{p_\omega}{1-p_\omega}=\frac{\int_{\Omega} \beta_{R,\omega} \, d\mathbb{P}(\omega)}{\int_{\Omega} \beta_{L,\omega} \, d\mathbb{P}(\omega)}$, solving for $p_\omega$ and renormalising gives the weights $p_\omega$ and $1-p_\omega$ determined through \eqn{eqn:pwlim}. In our case, this process is redundant since \eqn{eqn:pwlim} provides us with a readily computed expression for the weights associated with the limiting invariant density.
\end{remark}

\section{Characterisation of the second Oseledets space}
\label{sec:2nd_space}
In this section we generalise \cite[Theorem 2]{GTHW_metastable} to the random setting. We show that under the assumptions listed in \Sec{sec:map+pert} that the asymptotic behaviour of the second Oseledets space can be understood. Further, we show that the limiting function spanning this space can be identified explicitly. 
\begin{lemma}
    In the setting of \thrm{thrm:phi_lims}, for all $\varepsilon>0$ sufficiently small there exists a family of functions $(\psi_\omega^\varepsilon)_{\omega\in\Omega}$ where $\psi_\omega^\varepsilon\in \BV(I)$ satisfies $\mathcal{L}_\omega^\varepsilon \psi_\omega^\varepsilon = \rho_\omega^\varepsilon \psi_{\sigma\omega}^\varepsilon$, \jp{with} $\int_\Omega \log|\rho_\omega^\varepsilon|\, d\mathbb{P}(\omega)<0$, $||\psi_\omega^\varepsilon||_{L^1(\leb)}=1$, and $\int_I \psi_{\omega}^\varepsilon\, \dleb(x) = 0$ for $\mathbb{P}$-a.e. $\omega\in\Omega$.
    \label{lem:pert2nd}
\end{lemma}
\begin{proof}
    For $\varepsilon>0$, the existence of a family of functions $(\psi_\omega^\varepsilon)_{\omega\in\Omega}$, where $\psi_\omega^\varepsilon\in \BV(I)$ satisfies \begin{equation}
        \mathcal{L}_\omega^\varepsilon \psi_\omega^\varepsilon = \rho_\omega^\varepsilon \psi_{\sigma\omega}^\varepsilon
        \label{eqn:Lpsi}
    \end{equation} 
    where $||\psi_\omega^\varepsilon||_{L^1(\leb)}=1$ is guaranteed by \lem{lem:cont}. \jp{For $n\in\mathbb{N}$ and $\omega\in\Omega$ let $\rho_\omega^{\varepsilon\,(n)}:= \rho_{\sigma^{n-1}\omega}^\varepsilon\cdots\rho_{\sigma\omega}^\varepsilon\rho_\omega^\varepsilon$}. To show that $\int_\Omega \log|\rho_\omega^\varepsilon|\, d\mathbb{P}(\omega)<0$ we inductively apply \eqn{eqn:Lpsi} and find that for every $n\in\mathbb{N}$ 
    \begin{align}
        \mathcal{L}_{\sigma^{-n}\omega}^{\varepsilon\, (n)} \psi_{\sigma^{-n}\omega}^\varepsilon= \rho_{\sigma^{-n}\omega}^{\varepsilon\, (n)}\psi_\omega^\varepsilon. \label{eqn:Lnpsi}
    \end{align}
    Since $\psi_\omega^\varepsilon\in \BV(I)$, and by \lem{lem:cont}(b), spans the second Oseledets space of $\mathcal{L}_\omega^\varepsilon$, 
    \begin{align*}
        \int_\Omega \log|\rho_\omega^\varepsilon|\, d\mathbb{P}(\omega)&\stackrel{(\star)}={}\lim_{n\to \infty} \frac{1}{n}\log |\rho_{\sigma^{-n}\omega}^{\varepsilon\, (n)}|= \lim_{n\to \infty} \frac{1}{n}\log ||\rho_{\sigma^{-n}\omega}^{\varepsilon\, (n)}\psi_\omega^\varepsilon||_{\BV}  \\
        &\stackrel{(\star\star)}{=}\lim_{n\to\infty} \frac{1}{n}\log || \mathcal{L}_{\sigma^{-n}\omega}^{\varepsilon\, (n)} \psi_{\sigma^{-n}\omega}^\varepsilon ||_{\BV}=\lambda_2^\varepsilon<0
    \end{align*}
    for $\mathbb{P}$-a.e. $\omega\in\Omega$. At $(\star)$ we have used Birkhoff's ergodic theorem, and at $(\star\star)$ we have used \eqn{eqn:Lnpsi}. It remains to show that $\int_I \psi_{\omega}^\varepsilon\, \dleb(x) = 0$  for $\mathbb{P}$-a.e. $\omega\in\Omega$. Suppose that there exists a set of positive $\mathbb{P}$-measure such that $\int_I \psi_{\omega}^\varepsilon\, \dleb(x) \neq 0$. Using \eqn{eqn:Lnpsi}, the fact that $\mathcal{L}_\omega^\varepsilon$ preserves integrals, and since $\int_I \psi_{\omega}^\varepsilon\, \dleb(x) \neq 0$,
    \begin{equation}
        \rho_{\sigma^{-n}\omega}^{\varepsilon\, (n)} = \frac{\int_I \mathcal{L}_{\sigma^{-n}\omega}^{\varepsilon\, (n)} \psi_{\sigma^{-n}\omega}^\varepsilon\, \dleb(x)}{\int_I \psi_{\omega}^\varepsilon\, \dleb(x)} = \frac{\int_I \psi_{\sigma^{-n}\omega}^\varepsilon\, \dleb(x)}{\int_I \psi_{\omega}^\varepsilon\, \dleb(x)}.
        \label{eqn:rholims}
    \end{equation}
    Since $\lim _{n\to\infty} \frac{1}{n}\log|\rho_{\sigma^{-n}\omega}^{\varepsilon\, (n)}| = \lambda_2^\varepsilon <0$ for $\mathbb{P}$-a.e. $\omega\in\Omega$, we know that $\rho_{\sigma^{-n}\omega}^{\varepsilon\, (n)}= e^{n(\lambda_2^\varepsilon+o_{n\to\infty}(1))}$. Take $n$ sufficiently large such that $e^{2n\lambda_2^\varepsilon}\leq \rho_{\sigma^{-n}\omega}^{\varepsilon\, (n)}\leq e^{\frac{n}{2}\lambda_2^\varepsilon}$. Recalling that $\lambda_2^\varepsilon<0$, $\lim_{n\to\infty} \rho_{\sigma^{-n}\omega}^{\varepsilon\, (n)} = 0$ for $\mathbb{P}$-a.e. $\omega\in\Omega$, and \eqn{eqn:rholims} implies that for $\mathbb{P}$-a.e. $\omega\in\Omega$ 
    \begin{align}
       0 =\lim_{n\to \infty} \rho_{\sigma^{-n}\omega}^{\varepsilon\, (n)} = \frac{1}{\int_I \psi_\omega^\varepsilon\, \dleb(x)}\lim_{n\to \infty}\int_I \psi_{\sigma^{-n}\omega}^\varepsilon\, \dleb(x).
        \label{eqn:limint}
    \end{align}
     Let $\iota_\omega^\varepsilon:=\int_I \psi_{\omega}^\varepsilon\, \dleb(x)$. Clearly, $\iota^\varepsilon:\Omega \to \mathbb{R}$ is measurable. By the Poincar\'e recurrence theorem, since $\eqn{eqn:limint}$ implies that $\lim_{n\to \infty}\iota^\varepsilon_{\sigma^{-n}\omega} = 0$ for $\mathbb{P}$-a.e. $\omega\in\Omega$ then $\iota_\omega^\varepsilon = 0$ for $\mathbb{P}$-a.e. $\omega\in\Omega$ contradicting the fact that $\int_I \psi_{\omega}^\varepsilon\, \dleb(x) \neq 0$ on a set of positive measure.   
\end{proof}
\noindent
Using \lem{lem:pert2nd} we may now characterise the second Oseledets space.  
\begin{theorem}
In the setting of \thrm{thrm:phi_lims}, choose the sign of $\psi_\omega^\varepsilon$ such that $\int_{I_L}\psi_\omega^\varepsilon\, \dleb(x)>0$ for $\mathbb{P}$-a.e. $\omega\in\Omega$. Then as $\varepsilon\to 0$,
$$\psi_\omega^\varepsilon\stackrel{L^1}{\to}\psi_{}^0:= \frac{1}{2}\phi_L - \frac{1}{2}\phi_R$$
\label{thrm:cont2nd}
for $\mathbb{P}$-a.e. $\omega\in\Omega$.
\end{theorem}
\begin{remark}
In the statement of \thrm{thrm:cont2nd} we assume that one can choose the sign of $\psi_\omega^\varepsilon$ such that $\int_{I_L}\psi_\omega^\varepsilon\, \dleb(x)>0$ for $\mathbb{P}$-a.e. $\omega\in\Omega$. Observe that this is not possible if $\int_{I_L}\psi_\omega^\varepsilon\, \dleb(x)=0$ for $\mathbb{P}$-a.e. $\omega\in\Omega$. However, since $\psi_\omega^\varepsilon$ converges to a non-trivial linear combination of the initially invariant densities, this case never arises.  
\end{remark}
\begin{proof}[Proof of \thrm{thrm:cont2nd}]
     By \lem{lem:cont}, any accumulation point of $\psi_\omega^{{\varepsilon}}$ along the subsequence $\tilde{\varepsilon}\to 0$ satisfies $\psi_\omega^0=\lim_{\tilde{\varepsilon}\to 0}\psi_\omega^{\tilde{\varepsilon}}$ and may be expressed as a linear combination of $\phi_L$ and $\phi_R$. The existence of a convergent subsequence is discussed in \rem{rem:subseq}. It remains to identify the (possibly random) weights associated with $\psi_\omega^0$. By \lem{lem:pert2nd} we know that for all ${\tilde{\varepsilon}}>0$, $\int_I \psi_\omega^{\tilde{\varepsilon}}\, \dleb(x)=0$ for $\mathbb{P}$-a.e. $\omega\in \Omega$. Therefore, 
    \begin{align*}
    c_{1,\omega}+c_{2,\omega} &= \int_I c_{1,\omega}\phi_L + c_{2,\omega}\phi_R\, \dleb(x)=\int_I\psi_\omega^0\, \dleb(x)=0.
    \end{align*}
    To compute $c_{1,\omega}$ and therefore $c_{2,\omega}$, we recall that by \lem{lem:pert2nd}, for all ${\tilde{\varepsilon}}>0$, $||\psi_\omega^{\tilde{\varepsilon}}||_{L^1(\leb)}=1$. Thus,
    \begin{align*} 
    1=\int_I|\psi_\omega^0|\, \dleb(x)&= |c_{1,\omega}|\int_I |\phi_L-\phi_R|\, \dleb(x)\\
    &=|c_{1,\omega}|\left(\int_{I_L}|\phi_L|\, \dleb(x)+\int_{I_R}|\phi_R|\, \dleb(x)\right)\\
    &=2|c_{1,\omega}|\\
    \end{align*}
    Hence, we may conclude that $c_{1,\omega}=\pm \frac{1}{2}$. By assumption, since $\int_{I_L}\psi_\omega^{\tilde{\varepsilon}}\, \dleb(x)>0$ we may identify the unique limiting function corresponding to any accumulation point of $\psi_\omega^\varepsilon$.
\end{proof}

\begin{remark}
We note that the limiting function admitted by \thrm{thrm:cont2nd} is identical to that obtained in \cite[Theorem 2]{GTHW_metastable}. The only difference between our statement and that given in \cite{GTHW_metastable} is that we obtain a family of functions $(\psi_\omega^\varepsilon)_{\omega\in\Omega}$ for $\varepsilon>0$ sufficiently small as opposed to a family of eigenfunctions. Furthermore, by studying such systems in the random setting we obtain fibrewise convergence for $\mathbb{P}$-a.e. $\omega\in\Omega$. 
\end{remark}

\section{The limiting invariant density for multiple initially invariant states}
\label{sec:mstate}
In this section, we consider the setting of \Sec{sec:map+pert} but rather allow the initial system $T^0$ to have $m\geq2$ invariant sets $I_1,\dots,I_m$. Extending \hyperref[list:I4]{\textbf{(I4)}}, for $i\in \{1,\cdots,m\}$, $T^0|_{I_{i}}$ has only one ACIM $\mu_i$, whose density is denoted by $\phi_i:= d\mu_i/\dleb$. Upon perturbation, this invariance is destroyed, allowing points to randomly switch between the initially invariant subintervals. Fix $\varepsilon>0$ and consider $C^2$-small perturbations of $T^0$, denoted $T^\varepsilon:\Omega\times I \to I$, in the same sense as described in \hyperref[list:P2]{\textbf{(P2)}}. For $i,j\in \{1,\cdots,m\}$, let $H_{i,j,\omega}^\varepsilon$ be the set of all points mapping from $I_i$ to $I_j$ under one iteration of $T_\omega^\varepsilon$. That is, $H_{i,j,\omega}^\varepsilon=I_i \cap (T_\omega^\varepsilon)^{-1}(I_j)$. \jp{For $m\geq2$ invariant sets, the infinitesimal holes from \hyperref[list:I3]{\textbf{(I3)}} are defined in terms of the boundary points $\mathfrak{B}=\{b_0,\cdots,b_m\}$. Here, $(b_i)_{i=1}^{m-1}\subset(-1,1),\ b_0=-1$ and $b_m=1$ are such that for $i=1,\dots,m$ the sets $I_i := [b_{i-1},b_i]$ are invariant under $T^0$. We then let $H^0:=(T^0)^{-1}(\mathfrak{B})\setminus \mathfrak{B}$.} This allows us to impose a similar condition to \hyperref[list:P3]{\textbf{(P3)}}, namely $H_{i,j,\omega}^\varepsilon$ is a union of finitely many intervals, and as $\varepsilon\to 0$, $H_{i,j,\omega}^\varepsilon$ converges uniformly (in the Hausdorff metric) over $\omega\in\Omega$ away from a $\mathbb{P}$-null set to $H^0\cap I_i$. We naturally extend \hyperref[list:P4]{\textbf{(P4)}} by assuming that for $i,j\in \{1,\cdots,m\}$, $\mu_{i}(H_{i,j,\omega}^\varepsilon)=\varepsilon\beta_{i,j,\omega}+o_{\varepsilon\to 0}(\varepsilon)$ (where $\beta_{i,j}\in L^{\infty}(\mathbb{P})$), or $\mu_{i}(H_{i,j,\omega}^\varepsilon)=0$ for $\mathbb{P}$-a.e. $\omega\in\Omega$. In a similar manner, one can modify the remaining conditions of \Sec{sec:map+pert} such that they apply to this more general setting. We call these extended conditions {\textoverline{\textbf{(I1)}}}-{\textoverline{\textbf{(I6)}}} and {\textoverline{\textbf{(P1)}}}-{\textoverline{\textbf{(P7)}}}, \jp{however do not state them in length for presentation purposes}.
 \begin{remark}
    Without loss of generality we assume that under one iteration of $T_\omega^\varepsilon$, for some $i\in\{2,\cdots,m-1\}$, points in $I_i$ can only map to neighbouring sets $I_{i-1},I_{i+1}$ or remain in $I_i$. When $i=1$, points in $I_i$ can only map to $I_{i+1}$ or remain in $I_i$, whereas if $i=m$, points in $I_i$ can only map to $I_{i-1}$ or remain in $I_i$. One can guarantee this is satisfied by relabeling the intervals $I_1,\dots,I_m$. With this in mind, $H_{i,j,\omega}^\varepsilon\neq \emptyset$ if and only if 
    \begin{itemize}
        \item $i=1$ and $j=2$
        \item $i\in\{2,\cdots,m-1\}$ and $j=i\pm1$, or
        \item $i=m$ and $j=m-1$.
    \end{itemize}
    \label{rem:neighbour}
\end{remark}   \noindent
We begin by stating the main result of this section. This characterises the limiting invariant density of $(T_\omega^\varepsilon)_{\omega\in\Omega}$, given by $(\phi_\omega^\varepsilon)_{\omega\in\Omega}$, as $\varepsilon\to 0$  and relates it to the system's induced $m$-state Markov chain in a random environment driven by $\sigma:\Omega\to \Omega$. \\
\\
Consider the $m$-state Markov chains in random environments driven by $\sigma:\Omega\to \Omega$, with  transition matrices $(M_\omega^\varepsilon)_{\omega\in\Omega}$ where
\small
\begin{equation}
 M_{\omega}^{{\varepsilon}}:=\begin{pmatrix}
    1-{\varepsilon} \beta_{1,2,\omega} & {\varepsilon}\beta_{2,1,\omega} & 0 & 0 &  \cdots & \cdots & 0 & 0 \\
    {\varepsilon} \beta_{1,2,\omega} & 1-{\varepsilon}(\beta_{2,1,\omega}+\beta_{2,3,\omega}) & {\varepsilon} \beta_{3,2,\omega} & 0 & \cdots  & \cdots & \vdots & \vdots \\ 
    \vdots & \vdots & \vdots & \vdots & \vdots & \vdots & \vdots &\vdots \\
    0 & \cdots & \cdots & \cdots & \cdots  & 0 & {\varepsilon}\beta_{m-1,m,\omega} & 1-{\varepsilon}\beta_{m,m-1,\omega} 
\end{pmatrix}. \label{eqn:M-matrix}
\end{equation}\normalsize
For $i,j\in\{1,\cdots ,m\}$, $(M_{\omega}^{{\varepsilon}})_{ij}$ describes the one-step transition probabilities for the induced $m$-state Markov chain in a random environment, driven by $\sigma:\Omega\to \Omega$, for the map $T_\omega^\varepsilon$ from $I_j$ to $I_i$. For all $\varepsilon>0$, let $\Delta,N:\Omega \to  M_{m\times m}(\mathbb{R})$ be such that  
 \begin{align}
     \mathrm{diag}(M_\omega^{{\varepsilon}})=: I - {\varepsilon} \Delta_\omega \label{eqn:Delta}
 \end{align}
  and 
\begin{equation}
    M_{\omega}^{{\varepsilon}}=:I-{\varepsilon} \Delta_\omega +{\varepsilon} N_\omega.\label{eqn:Mdecomp}
\end{equation} \begin{theorem}
     Let $\{(\Omega,\mathcal{F},\mathbb{P},\sigma,\BV(I),\mathcal{L}^\varepsilon)\}_{\varepsilon\geq 0}$ be a \jp{family} of random dynamical systems of  metastable maps $T_\omega^\varepsilon:I\to I$ satisfying {\textoverline{\textbf{(I1)}}}-{\textoverline{\textbf{(I6)}}} and {\textoverline{\textbf{(P1)}}}-{\textoverline{\textbf{(P7)}}}. Suppose that $\Delta,N : \Omega \to M_{m\times m}(\mathbb{R})$ are as in \eqn{eqn:Delta} and \eqn{eqn:Mdecomp}. If $\int_\Omega \Delta_\omega \, d\mathbb{P}(\omega)$ is invertible, then there exists a unique vector $v^0 = \begin{pmatrix}
      c_1 & c_2 & \cdots & c_m   
     \end{pmatrix}^T\in\mathbb{R}^m$ with $c_i\geq 0$ for each $i\in\{1,\cdots ,m\}$ and $\sum_{i=1}^m c_i=1$, such that as $\varepsilon\to 0$,
     $$\phi_\omega^{{\varepsilon}} \stackrel{L^1}{\to} \phi^0= \sum_{i=1}^m c_{i} \phi_i$$
     {uniformly} over $\omega\in\Omega$ away from a $\mathbb{P}$-null set. Further, $v^0$ is the solution to
     \begin{equation}
         \left(I- \left(\int_{\Omega} \Delta_\omega\, d\mathbb{P}(\omega)\right)^{-1}\int_{\Omega}N_\omega\, d\mathbb{P}(\omega)\right)v^0 = 0\label{eqn:matrix equation}
     \end{equation}\label{thrm:m-invdens}
     that satisfies $\sum_{i=1}^m (v^0)_i=1$.
\end{theorem}\noindent
\begin{remark}
    In some cases, the assumption that $\int_\Omega \Delta_\omega\, d\mathbb{P}(\omega)$ is invertible in the statement of \thrm{thrm:m-invdens} may be relaxed. For example, if $\left(\int_\Omega \Delta_\omega\, d\mathbb{P}(\omega)\right)_{ii}=0$ for exactly one $i\in \{1,\cdots, m\}$, then the $i^{\mathrm{th}}$ state of the Markov chain in a random environment driven by $\sigma:\Omega\to \Omega$ with transition matrices $(M_\omega^\varepsilon)_{\omega\in \Omega}$ (as in \eqn{eqn:M-matrix}) is absorbing. In this case, as $\varepsilon\to 0$,
     $$\phi_\omega^{{\varepsilon}} \stackrel{L^1}{\to} \phi^0= \phi_i$$
     {uniformly} over $\omega\in\Omega$ away from a $\mathbb{P}$-null set. With this in mind, together with \thrm{thrm:m-invdens}, for systems with $m=2$ initially invariant sets, the statement of \thrm{thrm:m-invdens} reduces to the statement of \thrm{thrm:phi_lims}.  \label{rem:not-invert}   
\end{remark}
\noindent
The proof of \thrm{thrm:m-invdens} is divided into several lemmata. We first show that the Oseledets decomposition of $\mathcal{L}_\omega^\varepsilon$ depends continuously on perturbations.
\begin{lemma}
In the setting of \thrm{thrm:m-invdens}, the Lyapunov exponents and Oseledets projections of $\mathcal{L}_\omega^\varepsilon$ converge to those of $\mathcal{L}^0$ as $\varepsilon\to 0$. In particular, for all $\varepsilon>0$ sufficiently small, the top Oseledets space of $\mathcal{L}_\omega^\varepsilon$ is one-dimensional, and spanned by $\phi_\omega^\varepsilon\in \BV(I)$ with associated Lyapunov exponent $\lambda_1^\varepsilon=0$ of multiplicity one. Further, as $\varepsilon\to 0$, the space spanned by $\phi_\omega^{{\varepsilon}}$ converges uniformly over $\omega\in\Omega$ away from a $\mathbb{P}$-null set to the plane spanned by $\{\phi_1, \cdots, \phi_m\}$. Any accumulation point of the sequence $\phi_\omega^{{\varepsilon}}$, denoted $\phi_\omega^0$, lies in $\BV(I)$ and satisfies $\phi_\omega^0 \in \mathrm{span} \{\phi_1, \cdots, \phi_m\}$ where this convergence holds in $L^1$, and for $\mathbb{P}$-a.e. $\omega\in\Omega$.   
 \label{lem:cont_mstate}
\end{lemma}
\begin{proof}
    Identical to that of \lem{lem:cont} with the inclusion of $m$ initially invariant states. This \jp{\jps{implies} that, due to \rem{rmk:free-splitting}, $(\Omega,\mathcal{F},\mathbb{P},\sigma,\BV(I),\mathcal{L}^0)$ has a hyperbolic Oseledets splitting of index $d=m$, and further, that $\phi_\omega^0 \in \mathrm{span} \{\phi_1, \cdots, \phi_m\}$.} Recall that the existence of a convergent subsequence is guaranteed by \rem{rem:subseq}. 
\end{proof}\noindent
\lem{lem:cont_mstate} asserts that any accumulation point of $\phi_\omega^\varepsilon$, say along the subsequence $\tilde{\varepsilon}\to 0$, satisfies 
\begin{equation}
    \lim_{\tilde{\varepsilon}\to 0}\phi_\omega^{\tilde{\varepsilon}} = \phi_\omega^0:= \sum_{i=1}^m c_{i,\omega} \phi_i    \label{eqn:mstate-dens}
\end{equation}
in $L^1$ and for $\mathbb{P}$-a.e $\omega\in\Omega$. We aim to show that the accumulation points of $\phi_\omega^\varepsilon$ are always the same, so that $\phi_\omega^\varepsilon$ admits a limit as $\varepsilon\to 0$.\\
\\
\noindent
We replicate the procedure from \Sec{sec:LID} beginning with \lem{lem:p-1-p}, \ref{lem:holes} and \cor{cor:holes}. Recall that by \rem{rem:neighbour}, for $i\in\{2,\cdots,m-1\}$, non-empty holes must be of the form $H_{1,2,\omega}^{\tilde{\varepsilon}},H_{i,i\pm1,\omega}^{\tilde{\varepsilon}},H_{i\pm1,i,\omega}^{\tilde{\varepsilon}}$ and $H_{m,m-1,\omega}^{\tilde{\varepsilon}}$. 
\begin{lemma}
In the setting of \thrm{thrm:m-invdens}, if for $i=1,\dots, m$,  $c_{i,\omega}$ is as in \eqn{eqn:mstate-dens}, then for $\mathbb{P}$-a.e. $\omega\in\Omega$
\begin{align}
    c_{i,\omega}^{\tilde{\varepsilon}}:=\mu_\omega^{\tilde{\varepsilon}}(I_i)&=c_{i,\omega}+o_{\tilde{\varepsilon}\to 0}(1), \label{eqn:measIi}\\
    \mu_\omega^{\tilde{\varepsilon}}(H_{1,2,\omega}^{\tilde{\varepsilon}})&=c_{1,\omega}\tilde{\varepsilon}\beta_{1,2,\omega}+o_{\tilde{\varepsilon}\to 0}(\tilde{\varepsilon}), \label{eqn:H12}\\
    \mu_\omega^{\tilde{\varepsilon}}(H_{m,m-1,\omega}^{\tilde{\varepsilon}})&=c_{m,\omega}\tilde{\varepsilon}\beta_{m,m-1,\omega}+o_{\tilde{\varepsilon}\to 0}(\tilde{\varepsilon}). \label{eqn:Hmm-1}    
    \end{align}
 Further, for $i\in\{2,\cdots,m-1\}$
 \begin{align}
     \mu_\omega^{\tilde{\varepsilon}}(H_{i,i\pm 1,\omega}^{\tilde{\varepsilon}})&=c_{i,\omega}\tilde{\varepsilon}\beta_{i,i\pm 1,\omega}+o_{\tilde{\varepsilon}\to 0}(\tilde{\varepsilon}), \label{eqn:Hii+1}\\
    \mu_\omega^{\tilde{\varepsilon}}(H_{i\pm 1,i,\omega}^{\tilde{\varepsilon}})&=c_{i\pm 1,\omega}\tilde{\varepsilon}\beta_{i\pm 1,i,\omega}+o_{\tilde{\varepsilon}\to 0}(\tilde{\varepsilon}). \label{eqn:Hi+1i}
 \end{align}
 \label{lem:measure of holes}
\end{lemma}
\begin{proof}
Note that by {\textoverline{\textbf{(P6)}}}, for $\tilde{\varepsilon}>0$, $(T_\omega^{\tilde{\varepsilon}})_{\omega\in\Omega}$ admits a unique RACIM $(\mu_\omega^{\tilde{\varepsilon}})_{\omega\in\Omega}$. We first check whether \lem{lem:unif_conv} holds. Note that the proof of \lem{lem:unif_conv} never directly uses the fact that the system has two initially invariant states. This guarantees that any accumulation point of $\phi_\omega^\varepsilon$ along the subsequence $\tilde{\varepsilon}\to 0$ converges over each hole and over $\omega\in\Omega$ away from a $\mathbb{P}$-null set. Hence, the remainder of the proof is identical to that of \lem{lem:p-1-p}, \lem{lem:holes} and \cor{cor:holes}. Observe that in the case of two initially invariant states, we could obtain \eqn{eqn:measIL} and \eqn{eqn:measWLe}, from which we could derive \eqn{eqn:measIR} and \eqn{eqn:measWRe}. In the case of $m$ initially invariant sets, one could proceed in a similar manner. 
\end{proof}\noindent
We are now ready to prove \thrm{thrm:m-invdens}. In what follows, let
\begin{equation}
    v_{\omega}^{0}:=\begin{pmatrix}
 c_{1,\omega} \\ \vdots \\ c_{m,\omega}   
\end{pmatrix}\label{eqn:vw0}
\end{equation}
be an accumulation point of
\begin{equation} 
v_{\omega}^{\varepsilon}:=\begin{pmatrix}
 c_{1,\omega}^{\varepsilon} \\ \vdots \\ c_{m,\omega}^{\varepsilon}\end{pmatrix} 
    \label{eqn:vweps}
\end{equation}   
along the subsequence $\tilde{\varepsilon}\to 0$. We emphasise that the weights associated with any accumulation point of $\phi_\omega^\varepsilon$ along the subsequence $\tilde{\varepsilon}\to 0$ in \eqn{eqn:mstate-dens} are precisely the weights defined in \eqn{eqn:vw0}. For $\tilde{\varepsilon}>0$ and $i\in\{1,\cdots,m\}$, the $i^{\mathrm{th}}$ element of $v_\omega^{\tilde{\varepsilon}}$ is given by \eqn{eqn:measIi}. Furthermore, \eqn{eqn:measIi} asserts that as $\tilde{\varepsilon}\to 0$, $c_{i,\omega}^{\tilde{\varepsilon}}$ converges {uniformly} over $\omega\in\Omega$ away from a $\mathbb{P}$-null set to $c_{i,\omega}$.
\begin{proof}[Proof of \thrm{thrm:m-invdens}]
By \lem{lem:cont_mstate} and \lem{lem:measure of holes}, any accumulation point of $\phi_\omega^\varepsilon$ along the subsequence $\tilde{\varepsilon}^\prime\to 0$ satisfies $\phi_\omega^0 = \lim_{\tilde{\varepsilon}^\prime\to0}\phi_\omega^{\tilde{\varepsilon}^\prime}$ and
\begin{align}
    \phi_\omega^{\tilde{\varepsilon}^\prime} \stackrel{L^1}{\to} \phi_\omega^0&= \sum_{i=1}^m c_{i,\omega} \phi_i \nonumber \\
    &:= v_{\omega}^0\cdot \Phi \nonumber
\end{align}
{uniformly} over $\omega\in\Omega$ away from a $\mathbb{P}$-null set.\footnote{As discussed in \rem{rem:unif} and featured in the proof of \thrm{thrm:phi_lims}, \lem{lem:measure of holes} (particularly \eqn{eqn:measIi}) guarantees this convergence is uniform over $\omega\in\Omega$ away from a $\mathbb{P}$-null set.} Here, $v_{\omega}^0$ is as in \eqn{eqn:vw0} and $\Phi:=\begin{pmatrix} \phi_1 & \cdots & \phi_m\end{pmatrix}^T$. For $k\in\mathbb{Z}$, consider the restricted subsequence $\tilde{\varepsilon}\to 0$, along which any accumulation point of $\phi_{\sigma^k\omega}^\varepsilon$ satisfies $\phi_{\sigma^k\omega}^0=\lim_{\tilde{\varepsilon}\to 0}\phi_{\sigma^{k}\omega}^{\tilde{\varepsilon}}$ and
\begin{equation}
    \phi_{\sigma^k\omega}^{\tilde{\varepsilon}} \stackrel{L^1}{\to} \phi_{\sigma^k\omega}^0 = v_{\sigma^k\omega}^0\cdot \Phi\label{eqn:tilde-subseq}
\end{equation}
{uniformly} over $\omega\in\Omega$ away from a $\mathbb{P}$-null set.\footnote{The existence of such a convergent subsequence is guaranteed by using the \textit{diagonal sequence trick} (see for example \cite[Theorem 1.24]{Reed}).} We utilise properties of the RACIM for $(T_\omega^{\tilde{\varepsilon}})_{\omega\in\Omega}$ whose existence is guaranteed by {\textoverline{\textbf{(P6)}}}. That is, since $(\mu_{\omega}^{\tilde{\varepsilon}})_{\omega\in\Omega}$ is a RACIM for $(T_\omega^{\tilde{\varepsilon}})_{\omega\in\Omega}$, by following a similar calculation completed in the proof of \thrm{thrm:phi_lims}, for $i\in\{2,\cdots,m-1\}$
\begin{align}
    c_{i,\omega}^{\tilde{\varepsilon}} = \mu_\omega^{\tilde{\varepsilon}}(I_i) &= \mu_{\sigma^{-1}\omega}^{\tilde{\varepsilon}}((T_{\sigma^{-1}\omega}^{\tilde{\varepsilon}})^{-1}(I_i)) \nonumber \\
    &= \mu_{\sigma^{-1}\omega}^{\tilde{\varepsilon}}(I_i\setminus(H_{i,i+1,\sigma^{-1}\omega}^{\tilde{\varepsilon}} \cup H_{i,i-1,\sigma^{-1}\omega}^{\tilde{\varepsilon}}) \cup(H_{i+1,i,\sigma^{-1}\omega}^{\tilde{\varepsilon}} \cup H_{i-1,i,\sigma^{-1}\omega}^{\tilde{\varepsilon}})) \label{eqn:meas_union}\\
    &= c_{i,\sigma^{-1}\omega}^{\tilde{\varepsilon}}(1-\varepsilon(\beta_{i,i+1,\sigma^{-1}\omega}+\beta_{i,i-1,\sigma^{-1}\omega})) + c_{i+1,\sigma^{-1}\omega}^{\tilde{\varepsilon}}\tilde{\varepsilon}\beta_{i+1,i,\sigma^{-1}\omega} \nonumber\\
    &\quad+ c_{i-1,\sigma^{-1}\omega}^{\tilde{\varepsilon}}\tilde{\varepsilon}\beta_{i-1,i,\sigma^{-1}\omega}+o_{\tilde{\varepsilon}\to 0}(\tilde{\varepsilon}). \nonumber
\end{align}
Here we have used \lem{lem:measure of holes} and \eqn{eqn:measIi}. Similarly, one can obtain an expression for $c_{1,\omega}^{\tilde{\varepsilon}}$ and $c_{m,\omega}^{\tilde{\varepsilon}}$ with the exception in \eqn{eqn:meas_union} that $H_{1,0,{\sigma^{-1}\omega}}^{\tilde{\varepsilon}},H_{0,1,{\sigma^{-1}\omega}}^{\tilde{\varepsilon}}=\emptyset$ and $H_{m,m+1,{\sigma^{-1}\omega}}^{\tilde{\varepsilon}},H_{m+1,m,{\sigma^{-1}\omega}}^{\tilde{\varepsilon}}=\emptyset$, respectively. Thus, we may cast $v_\omega^{\tilde{\varepsilon}}$ (from \eqn{eqn:vweps}) in the form of a matrix equation where for each $n\in\mathbb{Z}$
\begin{equation}
    v_{\sigma^{n}\omega}^{\tilde{\varepsilon}}= M_{\sigma^{n-1}\omega}^{\tilde{\varepsilon}}v_{\sigma^{n-1}\omega}^{\tilde{\varepsilon}} + o_{\tilde{\varepsilon}\to 0}(\tilde{\varepsilon})
    \label{eqn:vnw}
\end{equation}
with the error term applied element-wise and $ M_{\omega}^{\tilde{\varepsilon}}$ is as in \eqn{eqn:M-matrix} with $\varepsilon$ replaced with $\tilde{\varepsilon}$. Taking  $\Delta,N : \Omega \to M_{m\times m}(\mathbb{R})$ as in \eqn{eqn:Delta} and \eqn{eqn:Mdecomp}, from \eqn{eqn:vnw}, 
\begin{align}
v_{\omega}^{\tilde{\varepsilon}}&= (I-\tilde{\varepsilon} \Delta_{\sigma^{-1}\omega})v_{\sigma^{-1}\omega}^{\tilde{\varepsilon}} +\tilde{\varepsilon} N_{\sigma^{-1}\omega} v_{\sigma^{-1}\omega}^{\tilde{\varepsilon}} +o_{\tilde{\varepsilon}\to 0}(\tilde{\varepsilon})\nonumber\\
&\stackrel{(\star)}{=} (I-\tilde{\varepsilon} \Delta_{\sigma^{-1}\omega})v_{\sigma^{-1}\omega}^{\tilde{\varepsilon}} +\tilde{\varepsilon} N_{\sigma^{-1}\omega} v_{\sigma^{-1}\omega}^{0} +o_{\tilde{\varepsilon}\to 0}(\tilde{\varepsilon}).\label{eqn:decompv}
\end{align}
At $(\star)$ we have used the fact that due to \eqn{eqn:measIi} and \eqn{eqn:tilde-subseq}, for any $k\in\mathbb{Z}$, $v_{\sigma^k\omega}^{\tilde{\varepsilon}} = v_{\sigma^k\omega}^0 + o_{\tilde{\varepsilon}\to 0}(1)$. Furthermore, since the entries of $N:\Omega \to M_{m\times m}(\mathbb{R})$ are $L^\infty(\mathbb{P})$ functions, the uniform error (over $\omega\in\Omega$) can be controlled. Fix $t>0$. Inductively using \eqn{eqn:decompv}, we find that for any $K_t^{\tilde{\varepsilon}}\in \mathbb{N}$ 
\begin{align}
    v_{\omega}^{\tilde{\varepsilon}}&= \left( \prod_{k=0}^{K_t^{\tilde{\varepsilon}}-1} (I-\tilde{\varepsilon} \Delta_{\sigma^{-k-1}\omega}) \right) v_{\sigma^{-K_t^{\tilde{\varepsilon}}}\omega}^{\tilde{\varepsilon}} \nonumber\\
    &\quad+\sum_{n=0}^{K_t^{\tilde{\varepsilon}}-1} \left( \prod_{k=0}^{n-1} (I-\tilde{\varepsilon} \Delta_{\sigma^{-k-1}\omega}) \right)(\tilde{\varepsilon} N_{\sigma^{-n-1}\omega}v_{\sigma^{-n-1}\omega}^{0} +o_{\tilde{\varepsilon}\to 0}(\tilde{\varepsilon})). \label{eqn:vKsum}
\end{align}
Set $K_t^{\tilde{\varepsilon}}= \frac{t}{\tilde{\varepsilon}}$. For $i\in\{1,\cdots, m\}$, observe that the $i^{\mathrm{th}}$ entry of \eqn{eqn:vKsum} is in the form of \eqn{eqn:Ksum} (with $\beta_R,\gamma\in L^\infty(\mathbb{P})$ replaced with functions, say $u_i,s_i\in L^\infty(\mathbb{P})$). Thus, by a similar argument to that made in the proof of \thrm{thrm:phi_lims}, $v_{\omega}^{\tilde{\varepsilon}}$ converges uniformly over $\omega\in\Omega$ away from a $\mathbb{P}$-null set to the same limit as 
\begin{equation}
    q_\omega^{\tilde{\varepsilon}}:=\sum_{n=0}^{\infty} \left( \prod_{k=0}^{n-1} (I-\tilde{\varepsilon} \Delta_{\sigma^{-k-1}\omega}) \right)(\tilde{\varepsilon} N_{\sigma^{-n-1}\omega}v_{\sigma^{-n-1}\omega}^{0} +o_{\tilde{\varepsilon}\to 0}(\tilde{\varepsilon}))\label{eqn:qomeps}
\end{equation}
as $\tilde{\varepsilon}\to 0$. Since $\Delta_\omega$ is a diagonal matrix and $N_\omega$ is as in \eqn{eqn:Mdecomp} for $\omega\in\Omega$, expanding \eqn{eqn:qomeps} we find that \footnotesize
\begin{align}
    q_\omega^{\tilde{\varepsilon}}&= \begin{pmatrix}
     \sum_{n=0}^\infty (\tilde{\varepsilon} \beta_{2,1,\sigma^{-n-1}\omega}c_{2,\sigma^{-n-1}\omega}+o_{\tilde{\varepsilon}\to 0}(\tilde{\varepsilon}))\prod_{k=0}^{n-1}(1-\tilde{\varepsilon} (\Delta_{\sigma^{-k-1}\omega})_{11})\\   
     \sum_{n=0}^\infty (\tilde{\varepsilon} (\beta_{1,2,\sigma^{-n-1}\omega}c_{1,\sigma^{-n-1}\omega}+\beta_{3,2,\sigma^{-n-1}\omega}c_{3,\sigma^{-n-1}\omega})+o_{\tilde{\varepsilon}\to 0}(\tilde{\varepsilon}))\prod_{k=0}^{n-1}(1-\tilde{\varepsilon} (\Delta_{\sigma^{-k-1}\omega})_{22})\\
     \vdots \\
     \sum_{n=0}^\infty (\tilde{\varepsilon} \beta_{m-1,m,\sigma^{-n-1}\omega}c_{m-1,\sigma^{-n-1}\omega}+o_{\tilde{\varepsilon}\to 0}(\tilde{\varepsilon}))\prod_{k=0}^{n-1}(1-\tilde{\varepsilon}(\Delta_{\sigma^{-k-1}\omega})_{mm})
    \end{pmatrix}\label{eqn:qomepssums}
\end{align}
\normalsize
where for $i\in\{1,\cdots ,m\}$, $c_{i,\omega}$ are as in \eqn{eqn:vw0}. Thus, by \lem{lem:weights}, since $\left(\int_\Omega \Delta_\omega \, d\mathbb{P}(\omega)\right)_{ii}\neq 0$ for $i\in\{1,\cdots ,m\}$ (since $\int_\Omega \Delta_\omega \, d\mathbb{P}(\omega)$ is invertible), for $\mathbb{P}$-a.e. $\omega\in\Omega$
$$\lim_{\tilde{\varepsilon}\to 0}q_{\omega}^{\tilde{\varepsilon}} =  \begin{pmatrix}
     \frac{\int_{\Omega} \beta_{2,1,\omega}c_{2,\omega}\, d\mathbb{P}(\omega)}{\int_\Omega \beta_{1,2,\omega}\, d\mathbb{P}(\omega)} \\
     \frac{\int_{\Omega} \beta_{1,2,\omega}c_{1,\omega} + \beta_{3,2,\omega}c_{3,\omega} \, d\mathbb{P}(\omega)}{\int_\Omega \beta_{2,1,\omega}+ \beta_{2,3,\omega} \, d\mathbb{P}(\omega)} \\
     \vdots \\
     \frac{\int_{\Omega} \beta_{m-1,m,\omega}c_{m-1,\omega}\, d\mathbb{P}(\omega)}{\int_\Omega \beta_{m,m-1,\omega}\, d\mathbb{P}(\omega)}
    \end{pmatrix}.$$
Therefore, $v_\omega^0 = \lim_{\tilde{\varepsilon}\to 0}q_\omega^{\tilde{\varepsilon}}:=v^0$ is independent of $\omega\in\Omega$ for $\mathbb{P}$-a.e. $\omega\in\Omega$. Using \eqn{eqn:qomeps} and noting that for $i\in\{1,\cdots ,m\}$, $\left(\int_\Omega \Delta_\omega \, d\mathbb{P}(\omega)\right)_{ii}\neq 0$, for $\mathbb{P}$-a.e. $\omega\in\Omega$,
\begin{align}
   v^0 &= \lim_{\tilde{\varepsilon}\to 0} \sum_{n=0}^{\infty} \left( \prod_{k=0}^{n-1} (I-\tilde{\varepsilon} \Delta_{\sigma^{-k-1}\omega}) \right)(\tilde{\varepsilon} N_{\sigma^{-n-1}\omega}v^{0} +o_{\tilde{\varepsilon}\to 0}(\tilde{\varepsilon}))\nonumber \\
   &\stackrel{(\star\star)}{=} \left(\int_{\Omega} \Delta_\omega\, d\mathbb{P}(\omega)\right)^{-1}\int_{\Omega}N_\omega\, d\mathbb{P}(\omega) v^0. \label{eqn:matequation}
\end{align}
At $(\star\star)$ we have used \lem{lem:weights}, since the entries of the matrix 
$$\sum_{n=0}^{\infty} \left( \prod_{k=0}^{n-1} (I-\tilde{\varepsilon} \Delta_{\sigma^{-k-1}\omega}) \right)(\tilde{\varepsilon} N_{\sigma^{-n-1}\omega}+o_{\tilde{\varepsilon}\to 0}(\tilde{\varepsilon}))$$
are in the form of \eqn{eqn:pweps} with $\left(\int_\Omega \Delta_\omega \, d\mathbb{P}(\omega)\right)_{ii}\neq 0$ for $i\in\{1,\cdots ,m\}$. Due to \eqn{eqn:measIi}, $v_\omega^{\tilde{\varepsilon}}$ converges uniformly over $\omega\in\Omega$ away from a $\mathbb{P}$-null set to the same limit as $q_\omega^{\tilde{\varepsilon}}$, we may conclude that $v_\omega^{\tilde{\varepsilon}}$ converges uniformly over $\omega\in\Omega$ away from a $\mathbb{P}$-null set to the non-trivial solution of \eqn{eqn:matequation}.
\end{proof}
\begin{remark}
    Similar to the discussion in \rem{rem:markov-connection}, if $\int_\Omega \Delta_\omega\, d\mathbb{P}(\omega)$ is invertible, and $\beta_{i,j}\in L^\infty(\mathbb{P})$, then \thrm{thrm:m-invdens} provides us with the limiting invariant measures of the $m$-state Markov chains in random environments driven by $\sigma:\Omega\to \Omega$ (as $\varepsilon\to0$), with transition matrices $(M_{\omega}^\varepsilon)_{\omega\in\Omega}$ given by \eqn{eqn:M-matrix} for each $\omega\in\Omega$.\footnote{Refer to \rem{rem:not-invert} as to how this invertibility assumption may be relaxed.} Namely, if $\int_\Omega \Delta_\omega\, d\mathbb{P}(\omega)$ is invertible, as $\varepsilon\to 0$, the random invariant measure of $(M_{\omega}^\varepsilon)_{\omega\in\Omega}$ converges uniformly over $\omega\in\Omega$ away from a $\mathbb{P}$-null set to the non-trivial solution, $v^0$, of \eqn{eqn:matrix equation}. This follows from the proof of \thrm{thrm:m-invdens} by setting the errors appearing in \eqn{eqn:H12}, \eqn{eqn:Hmm-1}, \eqn{eqn:Hii+1} and \eqn{eqn:Hi+1i} to zero. 
\end{remark}
\begin{remark}
    In the setting that $T^0$ admits $m\geq2$ initially invariant sets, the results of \Sec{sec:2nd_space} assert that the functions spanning the subsequent Oseledets spaces of $\mathcal{L}_\omega^\varepsilon$ with Lyapunov exponents approaching zero (as $\varepsilon\to 0$) converge in $L^1$ to a linear combination of the initially invariant densities for $\mathbb{P}$-a.e. $\omega\in\Omega$. Furthermore, the limiting function has integral zero.   
\end{remark}

\section{Example: Random paired tent maps}
\label{sec:example}
In this section, we show that our results may be used to approximate the limiting random invariant density, and the second Oseledets space for Perron-Frobenius operator cocycles associated with random paired tent maps depending on a parameter $\varepsilon$. This class of random dynamical systems was introduced by Horan in \cite{horan_lin,Horan} and has been used as a primary candidate of a non-autonomous system that admits metastability \cite{GTQ_quarantine}. For this reason, we benchmark our results against such a model. 
\subsection{The system of concern}
Let $I=[-1,1]$ and for $0\leq a,b\leq 1$ consider the paired tent map $T_{a,b}:I\to I$ given by
\begin{equation}
    T_{a,b}(x) := \begin{cases} 2(1+b)(x+1)-1, \ \ \ \ \ &x\in [-1,-1/2]\\
    -2(1+b)x-1, &x\in[-1/2,0)\\
    0, & x=0\\
    -2(1+a)x+1, &x\in(0,1/2]\\
    2(1+a)(x-1)+1, &x\in[1/2,1]
    \end{cases}
    \label{eqn:ptm}
\end{equation}
This map satisfies $T_{a,b}(-1)=-1$, $T_{a,b}(-1/2)=b$ and $\lim_{x\to 0^-}T(x)=-1$; $\lim_{x\to 0^+}T_{a,b}(x)=1$, $T_{a,b}(1/2)=-a$ and $T_{a,b}(1)=1$. The map $T_{0,0}$ comprises of tent maps on disjoint subintervals $I_L=[-1,0]$ and $I_R=[0,1]$. For small positive $a$ and $b$, there is a small amount of leakage between these subintervals: points near $-1/2$ are mapped to $I_R$ whilst points near $1/2$ are mapped to $I_L$. This behaviour may be seen in \fig{fig:ptm}.
\begin{figure}[H]

\begin{subfigure}{0.45\linewidth}
\begin{tikzpicture}[
  declare function={
    func(\x)= (\x <-1/2) * (2*\x+1)   +
              and(\x >= -1/2, \x < 0) * (-2*\x-1)     +
              and(\x > 0, \x <= 1/2) * (-2*\x+1)   +
              (\x > 1/2) * (2*\x-1);
  }
,scale=0.78]
\begin{axis}[
  samples = 7, axis x line=middle, axis y line=middle,
  ymin=-1, ymax=1, ytick={-1,...,1}, ylabel=$T_{0,0}$,
  xmin=-1, xmax=1, xtick={-1,...,1}, xlabel=$x$
]
\addplot [blue, domain=-1:0, thick] {func(x)}; 
\addplot [blue, domain=0.001:1, thick] {func(x)}; 
\addplot[blue, only marks, mark=*, mark size=0.8] coordinates {(0,0)};
\end{axis}
\end{tikzpicture}
    \label{fig:ptma}
    \end{subfigure}\hfill
\begin{subfigure}{0.45\linewidth}
\begin{tikzpicture}[
  declare function={
    func(\x)= (\x <-1/2) * (2.2*\x+1.2)   +
              and(\x >= -1/2, \x < 0) * (-2.2*\x-1)     +
              and(\x > 0, \x <= 1/2) * (-2.2*\x+1)   +
              (\x > 1/2) * (2.2*\x-1.2);
  }
,scale=0.78]
\begin{axis}[
  samples = 7, axis x line=middle, axis y line=middle,
  ymin=-1, ymax=1, ytick={-1,...,1}, ylabel=$T_{a,b}$,
  xmin=-1, xmax=1, xtick={-1,...,1}, xlabel=$x$
]
\addplot [blue, domain=-1:0, thick] {func(x)}; 
\addplot [blue, domain=0.001:1, thick] {func(x)}; 
\addplot[blue, only marks, mark=*, mark size=0.8] coordinates {(0,0)};
\addplot [mark=none, red, samples=2,dashed] {0.1};
\node[above right] at (axis cs: -1, 0.1) {$b$};

\addplot [mark=none, red, samples=2,dashed] {-0.1};
\node[below left] at (axis cs: 1, -0.1) {$-a$};
\end{axis}
\end{tikzpicture}
    \label{fig:ptmb}
    \end{subfigure}
    \caption{Paired tent map $T_{a,b}$ on $I=[-1,1]$ with $a=b=0$ (left) and $a=b=\frac{1}{10}$ (right)}
    \label{fig:ptm}
\end{figure}
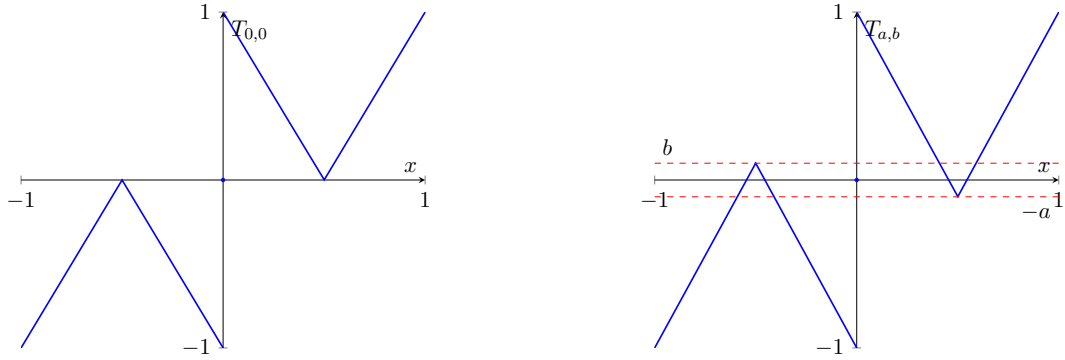
\noindent
Now consider a cocycle of paired tent maps driven by an ergodic, measure-preserving and invertible transformation $\sigma:\Omega \to \Omega$. The $\omega$-dependence is introduced in this system by making the leakage between intervals $I_L$ and $I_R$ random. This is guaranteed by considering the evolution of a point $x\in I$ under the dynamics $T_\omega := T_{a_\omega,b_\omega}$ where $a,b:\Omega \to [0,1]$ are measurable functions, and thus $a,b\in L^\infty(\mathbb{P})$. Since we are interested in the behaviour of such a system when small amounts of leakage can occur, we consider for sufficiently small $\varepsilon>0$ the map $T_\omega^\varepsilon := T_{\varepsilon a_\omega,\varepsilon b_\omega}$.\footnote{One can interpret the parameter $\varepsilon$ as a leakage controller. The larger you take $\varepsilon$, the more likely it is for points to map to and from the initially invariant subintervals.} Here, there is a low probability that a point will be mapped from $I_L$ to $I_R$ in one step and vice-versa. When $\varepsilon = 0$, $T^0 := T_{0,0}$ and the map consists of tent maps on disjoint subintervals. Furthermore, the dynamics is entirely deterministic with the $\varepsilon=0$ perturbation, removing the randomness of the system. 
\\
\\
The random maps $(T_\omega^\varepsilon)_{\omega\in\Omega} := (T_{\varepsilon a_\omega, \varepsilon b_\omega})_{\omega\in\Omega}$, driven by $\sigma$, are the primary focus of this section. We verify the conditions outlined in \Sec{sec:map+pert} in order to apply \thrm{thrm:phi_lims} and \thrm{thrm:cont2nd} to this random paired metastable system. The main results of this section are the following.
\begin{theorem}
 Let $\{(\Omega,\mathcal{F},\mathbb{P},\sigma,\BV(I),\mathcal{L}^\varepsilon)\}_{\varepsilon\geq 0}$ be a \jp{family} of random dynamical systems of paired tent maps $(T^\varepsilon_\omega)_{\omega\in\Omega}:= (T_{\varepsilon a_\omega, \varepsilon b_\omega})_{\omega\in\Omega}$ satisfying \hyperref[list:P1]{\textbf{(P1)}} where $a,b :\Omega \to [0,1]$ are measurable, and $\int_\Omega a_\omega +b_\omega\, d\mathbb{P}(\omega)\neq 0$. If $(\phi_\omega^\varepsilon)_{\omega\in\Omega}$ denotes the random invariant density of $(T_{\omega}^\varepsilon)_{\omega\in\Omega}$, then as $\varepsilon\to 0$
    $$\phi_\omega^\varepsilon \stackrel{L^1}{\to} \phi^0:= \frac{\int_{\Omega} a_\omega \, d\mathbb{P}(\omega)}{\int_{\Omega} a_\omega +b_\omega \, d\mathbb{P}(\omega)}\mathds{1}_{I_L}+\frac{\int_{\Omega} b_\omega \, d\mathbb{P}(\omega)}{\int_{\Omega} a_\omega +b_\omega \, d\mathbb{P}(\omega)}\mathds{1}_{I_R}$$
    uniformly over $\omega\in\Omega$ away from a $\mathbb{P}$-null set.
    \label{thrm:ptm-inv}
\end{theorem}
\begin{theorem}
In the setting of \thrm{thrm:ptm-inv}, let $(\psi_\omega^\varepsilon)_{\omega\in\Omega}$ denote the functions spanning the second Oseledets spaces of $(\mathcal{L}_{\omega}^\varepsilon)_{\omega\in\Omega}$. Choose the sign of $\psi_\omega^\varepsilon$ such that $\int_{I_L}\psi_\omega^\varepsilon\, \dleb(x) >0$ for  $\mathbb{P}$-a.e. $\omega\in\Omega$. Then as $\varepsilon\to 0$,
$$\psi_\omega^\varepsilon\stackrel{L^1}{\to}\psi_{}^0:= \frac{1}{2}\mathds{1}_{I_L} - \frac{1}{2}\mathds{1}_{I_R}$$
for $\mathbb{P}$-a.e. $\omega\in\Omega$.
\label{thrm:ptm_cont2nd}
\end{theorem}
\begin{remark}
    For $\star \in \{ L,R \}$, we know that $\phi_\star = \mathds{1}_{I_\star}$ in the statements of \thrm{thrm:phi_lims} and \thrm{thrm:cont2nd} since the unique ACIMs of the unperturbed system supported on the initially invariant intervals are given by $\mu_\star = \leb|_{I_\star}$.
\end{remark}\noindent
The proof of \thrm{thrm:ptm-inv} and \thrm{thrm:ptm_cont2nd} are direct consequences of \thrm{thrm:phi_lims} and \thrm{thrm:cont2nd}, respectively. The remainder of this section is dedicated to verifying the conditions outlined in \Sec{sec:map+pert}, allowing us to apply \thrm{thrm:phi_lims} and \thrm{thrm:cont2nd} to random paired tent maps. 
\subsection{Conditions for the unperturbed map}
We begin by showing conditions \hyperref[list:I1]{\textbf{(I1)}}-\hyperref[list:I6]{\textbf{(I6)}} hold. Suppose that $\mathcal{L}^0$ is the Perron-Frobenius operator associated with the map $T^0:=T_{0,0}$ defined in \eqn{eqn:ptm} acting on $\BV(I)$. One can show that \hyperref[list:I1]{\textbf{(I1)}} is satisfied with critical set $\mathcal{C}^0 = \{-1,-\frac{1}{2},0,\frac{1}{2},1\}$. Furthermore, since $\inf_{x\in I\setminus \mathcal{C}^0}|(T^0)^\prime(x)|=2>1$, \hyperref[list:I2]{\textbf{(I2)}} holds. For \hyperref[list:I3]{\textbf{(I3)}}, we note that there exists a boundary point $b=0$ that splits $I=[-1,1]$ into subintervals $I_L=[-1,0]$ and $I_R=[0,1]$ which are invariant under $T^0$.\footnote{We remind the reader of why this is the case. Without loss of generality suppose $x\in I_L$, then $T_{0,0}(x)\in [-1,0]=I_L$ meaning $T_{0,0}^{-1}(I_L)\subset I_L$. An almost identical argument can be made for $I_R$.} On each of these invariant subintervals, the dynamics is given by the tent map. By \cite{LY_acim} and \cite{LY_unique}, for $\star\in\{L,R\}$, $T^0|_{I_\star}$ has a unique ACIM given by $\leb|_{I_\star}$ with density $\mathds{1}|_{I_\star}$ and thus \hyperref[list:I4]{\textbf{(I4)}} is satisfied. 
\\
\\
To verify \hyperref[list:I5]{\textbf{(I5)}}, we note that the infinitesimal holes for $T^0$ are $H^0=(T^0)^{-1}(\{b\})\setminus \{b\} = \{-\frac{1}{2},\frac{1}{2}\}$. Since $-1,0,1$ are \jp{all} fixed points of $T^0$ and elements of $\mathcal{C}^0$, and $x=\pm \frac{1}{2}$ are mapped to $0$ under $T^0$, we have that for each $k>0$ that $T^{0\, (k)}(\mathcal{C}^0)= \{-1,0,1\}$. Thus for each $k>0$, $T^{0\, (k)}(\mathcal{C}^0)\cap H^0=\emptyset$.
\begin{remark}
    We remind the reader that \hyperref[list:I5]{\textbf{(I5)}} is used to ensure that for $\star\in\{L,R\}$, the density $\phi_\star$ is continuous at all points in $H^0$. Verification of this condition may not be necessary if we know precisely what $\phi_\star$ is. Namely, in our case since $\phi_\star=\mathds{1}|_{I_\star}$, we know that $\phi_\star$ is continuous on $I_\star$ and hence on $H^0$.  
\end{remark}
\noindent
For \hyperref[list:I6]{\textbf{(I6)}}, since $H^0\cap I_L=\{-\frac{1}{2}\}$ and $H^0\cap I_R = \{\frac{1}{2}\}$, we know that since $\phi_\star = \mathds{1}|_{I_\star}$ for $\star\in\{L,R\}$, $\phi_\star$ is positive at each of the points in $H^0\cap I_L$ and $H^0\cap I_R$. 

\subsection{Conditions for the perturbed map}
Upon perturbing the initial system, we now consider the paired tent map cocycle $(T_\omega^\varepsilon)_{\omega\in\Omega}$. It remains to verify conditions \hyperref[list:P2]{\textbf{(P2)}}-\hyperref[list:P7]{\textbf{(P7)}}, from which, by enforcing \hyperref[list:P1]{\textbf{(P1)}}, \thrm{thrm:ptm-inv} and \thrm{thrm:ptm_cont2nd} hold.
\\
\\
One can show that \hyperref[list:P2]{\textbf{(P2)}} is satisfied. This follows primarily from the fact that the critical set for $T^\varepsilon:\Omega\times I\to I$ is identical to that of $T^0:I\to I$. In particular, $\mathcal{C}_\omega^\varepsilon=\{-1,-\frac{1}{2},0,\frac{1}{2},1\}$. Conditions \hyperref[list:P3]{\textbf{(P3)}} and \hyperref[list:P7]{\textbf{(P7)}} hold due to the following.
\noindent
\begin{lemma}
 In the setting of \thrm{thrm:ptm-inv}, let $H_{L,\omega}^\varepsilon$ and $H_{R,\omega}^\varepsilon$ denote the set of all points mapped by $T_\omega^\varepsilon$ from $I_L$ to $I_R$ and from $I_R$ to $I_L$, respectively. Then
\begin{align*}
    H_{L,\omega}^\varepsilon &= \left[-1+\frac{1}{2(1+\varepsilon b_\omega)},-\frac{1}{2(1+\varepsilon b_\omega)}\right]\\
    H_{R,\omega}^\varepsilon &= \left[\frac{1}{2(1+\varepsilon a_\omega)}, 1-\frac{1}{2(1+\varepsilon a_\omega)}\right].    
\end{align*}
Furthermore, conditions \hyperref[list:P3]{\textbf{(P3)}} and \hyperref[list:P7]{\textbf{(P7)}} are satisfied.
\label{lem:WLR}
\end{lemma}
\begin{proof}
By \eqn{eqn:ptm}, the left and right endpoints of the interval $H_{L,\omega}^\varepsilon$ are determined by the solutions $x$ to the equations
\begin{align*}
 2(1+\varepsilon b_\omega)(x+1)-1&=0\\
    -2(1+\varepsilon b_\omega)x-1&=0,    
\end{align*}
respectively. Similarly the left and right endpoints of the interval $H_{R,\omega}^\varepsilon$ are determined by the solutions $x$ to the equations
\begin{align*}
 -2(1+\varepsilon a_\omega)x+1&=0\\
    2(1+\varepsilon a_\omega)(x-1)+1&=0,   
\end{align*}
respectively. Upon solving these systems of linear equations we obtain $H_{L,\omega}^\varepsilon$ and $H_{R,\omega}^\varepsilon$ as claimed. Since $a,b:\Omega\to [0,1]$ are measurable, one can show that \hyperref[list:P3]{\textbf{(P3)}} holds. Finally, \jp{we show that \hyperref[list:P7]{\textbf{(P7)}} holds. We note that the boundary point $b=0\in \mathcal{C}_0$, thus we verify condition \hyperref[list:P7]{\textbf{(P7)}}(b). Indeed, observe that $T^0(0^{-})=-1<b<1=T(0^+)$.\footnote{We recall that $f(0^s):= \lim_{x\to 0^s}f(x)$ for $s\in\{+,-\}$.} Furthermore, for all $\varepsilon>0$ and $\mathbb{P}$-a.e. $\omega\in\Omega$, $b\in \mathcal{C}_\omega^\varepsilon$. Therefore, \hyperref[list:P7]{\textbf{(P7)}}(b) holds}. 
\end{proof}
\noindent
\lem{lem:WLR} allows us to verify \hyperref[list:P4]{\textbf{(P4)}}.
\begin{lemma}
    In the setting of \thrm{thrm:ptm-inv} 
    \begin{align}
        \mu_L(H_{L,\omega}^\varepsilon)&=\varepsilon b_\omega +o_{\varepsilon\to 0}(\varepsilon),\label{eqn:muL}\\
        \mu_R(H_{R,\omega}^\varepsilon)&=\varepsilon a_\omega +o_{\varepsilon\to 0}(\varepsilon), \nonumber
    \end{align}
    where $a,b\in L^\infty(\mathbb{P})$. Hence \hyperref[list:P4]{\textbf{(P4)}} is satisfied.
    \label{lem:meas-ptm}
\end{lemma}
\begin{proof}
We establish \eqn{eqn:muL} as the computation for $\mu_R(H_{R,\omega}^\varepsilon)$ follows in an almost identical manner. Since $\mu_L=\leb|_{I_L}$, by \lem{lem:WLR}, for $\varepsilon>0$ sufficiently small,
\begin{align*}
\mu_L(H_{L,\omega}^\varepsilon)&= \leb|_{I_L}\left(\left[-1+\frac{1}{2(1+\varepsilon b_\omega)},-\frac{1}{2(1+\varepsilon b_\omega)}\right]\right)\\
&= 1- \frac{1}{1+\varepsilon b_\omega} \\
&= \varepsilon b_\omega - \varepsilon \frac{\varepsilon b_\omega^2}{1+\varepsilon b_\omega}\\
&= \varepsilon b_\omega+\varepsilon \cdot o_{\varepsilon\to 0}(1).
\end{align*}
Observe that the error term in the last line is independent of $\omega\in\Omega$ since the function $\chi_\omega^\varepsilon:= \frac{\varepsilon b_\omega^2}{1+\varepsilon b_\omega}$ converges uniformly over $\omega\in\Omega$ away from a $\mathbb{P}$-null set to $0$ as $\varepsilon\to 0$. Since $b:\Omega \to [0,1]$ is measurable and uniformly bounded, then $b\in L^\infty(\mathbb{P})$ and \hyperref[list:P4]{\textbf{(P4)}} is satisfied.
\end{proof}
\noindent
\noindent
We now determine whether \hyperref[list:P5]{\textbf{(P5)}} is satisfied. Establishing a uniform Lasota-Yorke inequality for $\mathcal{L}_\omega^\varepsilon$ is rather technical due to the creation of small monotonic branches of the second iterate map when $\varepsilon>0$ is sufficiently small. This issue can be avoided by utilising the so-called balanced Lasota-Yorke inequality developed by Horan in \cite{Horan}. To apply the results presented in \cite{Horan}, we record for the reader's convenience relevant definitions.
\begin{definition}
    Let $a<b \in \mathbb{R}$, let $I=[a,b]$, and let $\overline{\leb}$ denote the normalised Lebesgue measure on $I$. A map $T:I\to I$ satisfies the collection of assumptions ($\mathcal{M}$) when:
    \begin{itemize}
        \item There exists a countable cover of $I$ by closed intervals $\{I_n\}_{n\in{N}}$ with $I_n = [a_n, b_n]$, where $a_n < b_n$. Here $N$ may be finite or countably infinite.
        \item $(a_n,b_n)\cap (a_m.b_m)=\emptyset$ for all $n\neq m$.
        \item $\cup_{n\in \mathbb{N}} (a_n,b_n)$ is dense in $I$ and has measure $1$.
        \item $T|_{(a_n,b_n)}$ is continuous and extends to a homeomorphism $T_n:I_n\to T_n(I_n)$.
        \item There exists a bounded measurable function $g : I \to [0, \infty)$ such that $\mathcal{L}(f )(x) :=\sum_{y\in T^{-1}(x)}g(y)f (y)$ defines an operator that preserves $\overline{\leb}$ (that is, $\overline{\leb}(\mathcal{L}(f)) = \overline{\leb}(f)$ for all integrable functions $f$), $g$ has finite variation on each $I_n$, and $g$ is $0$ at the endpoints of each $I_n$.
    \end{itemize}
    \label{def:Mclass}
\end{definition}
\noindent
Horan shows in \cite{Horan} that for each $\omega\in\Omega$, random paired tent maps satisfy the conditions of \dfn{def:Mclass} when $g_\omega^\varepsilon(x):= |(T_\omega^\varepsilon)^\prime(x)|^{-1}$. For the sake of notation, let $\accentset{\circ}{I}_n=(a_n,b_n)$ and $f(x^s):= \lim_{t\to x^s}f(t)$ for $s\in\{+,-\}$. To formally state the balanced Lasota-Yorke inequality we also require the definition of \textit{hanging points} from Eslami and G\'ora \cite{EG_LY}.  
\begin{definition}
    Let $I=[a,b]$ and $T:I\to I$ satisfy the assumptions ($\mathcal{M}$). For pairs $(x,s)$ of the form $(a_n,+)$ or $(b_n,-)$, we say that the pair $(x,s)$ is a \textit{hanging point} for $T$ when $T(x^s)\notin \{a,b\}$, or alternatively $\mathds{1}_{\{a,b\}}(T(x^s))=0$. We say that a hanging point $(x,s)$ is contained in an interval $[c,d]\subset I$ when $x\in [c,d)$ and $s=+$ or when $x\in (c,d]$ and $s=-$, so that limits of the form $f(x^s)$ make sense in $[c,d]$.     
\end{definition}
\begin{example}
    The random paired tent map $T_\omega^\varepsilon:I\to I$ has 4 hanging points 
    $$H=\{(-0.5,-), (-0.5,+), (0.5,-), (0.5,+)\}.$$
    The remaining endpoints of the branches of monotonicity of $T_\omega^\varepsilon$ are not hanging points since $T_\omega^\varepsilon(-1^+)=-1$, $T_\omega^\varepsilon(0^+)= 1$, $T_\omega^\varepsilon(0^-)= -1$, and $T_\omega^\varepsilon(1^-)=1$.
    \label{ex:ptmhanging}
\end{example}
\noindent
With the above definitions and examples, we may now state Horan's balanced Lasota-Yorke inequality.
\begin{proposition}[{\cite[Proposition 3.6]{Horan}}]
Let $I=[a,b]\subset \mathbb{R}$ and $T:I\to I$ satisfy the assumptions ($\mathcal{M}$). Let $H$ be the collection of hanging points for $T$. Suppose further that 
\[ \sup_{n\in \mathbb{N}}\left\{\frac{\Var_{\accentset{\circ}{I}_n}(g)}{\overline{\leb}(I_n)} \right\} < \infty, \quad \text{and} \quad \sum_{(z,s)\in H} g(z^s) < \infty. \]
Then for any $f\in \BV(I)$ and any finite collection of closed intervals $\mathcal{J} = \{ J_m \}_{m=1}^M$ with disjoint non-empty interiors such that each hanging point of $T$ is contained in some $J_m$ (hence, in only one $J_m$), we have:
\begin{align*}
\Var(\mathcal{L}(f)) & \leq \left( \sup_{n}\left\{ \norm{\restr{g}{\accentset{\circ}{I}_n}}_\infty + \Var_{\accentset{\circ}{I}_n}(g) \right\} + \max_m\{ h_{\mathcal{J}}(m) \} \right) \Var(f) \\
& \hspace{20pt} + \left( \sup_n\left\{ \frac{\Var_{\accentset{\circ}{I}_n}(g)}{\overline{\leb}(I_n)} \right\} + \max_m\left\{ \frac{h_{\mathcal{J}}(m)}{\overline{\leb}(J_m)} \right\} \right) \norm{f}_{L^1(\overline{\leb})},
\end{align*}
where $\displaystyle h_{\mathcal{J}}(m) := \sum_{(z,s)\in H\cap J_m} g(z^{s})$.
\label{prop:unif-LY}
\end{proposition}
\noindent
For the paired tent map $T_{a_\omega,b_\omega}$ we have a uniform Lasota-Yorke inequality for the second iterate map. 
\begin{proposition}[{\cite[Proposition 4.4]{Horan}}]
    \label{prop:tent-map-ly-ineq}
For any paired tent map cocycle $(T_{a_\omega,b_\omega})_{\omega\in\Omega}$ over $\sigma$ and associated second-iterate Perron-Frobenius operator $\mathcal{L}_\omega^{(2)}$, we have that for any $f\in \BV(I)$: 
\begin{equation}
     \Var(\mathcal{L}_\omega^{(2)}f) \leq \frac{3}{4}\Var(f) + 6\norm{f}_{L^1(\overline{\leb})}. 
     \label{eqn:lyineq}
\end{equation}
\end{proposition}
\noindent
A combination of \prop{prop:unif-LY} and \prop{prop:tent-map-ly-ineq} reveals a uniform Lasota-Yorke inequality for the set $\{(\Omega,\mathcal{F},\mathbb{P},\sigma, \BV(I) ,\mathcal{L}^\varepsilon)\}_{\varepsilon\geq 0}$. By iterating the inequality obtained in \prop{prop:tent-map-ly-ineq} we obtain a uniform Lasota-Yorke inequality for the even iterates of the Perron-Frobenius operator $\mathcal{L}_\omega^\varepsilon$.
\begin{lemma}
In the setting of \thrm{thrm:ptm-inv}, for any $f\in \BV (I)$ and $n\in\mathbb{N}$
$$\Var({\mathcal{L}_\omega^{\varepsilon\, (n)}f})\leq \left(\frac{3}{4}\right)^n\Var (f) + 26||f||_{L^1(\overline{\leb})}.$$
Hence, \hyperref[list:P5]{\textbf{(P5)}} is satisfied.
\label{lem:ULY}
\end{lemma}
\begin{proof}
The proof is divided into several steps.
\begin{stp}
For any $f\in \BV (I)$ and $n\in\mathbb{N}$
$$\Var({\mathcal{L}_\omega^{\varepsilon\, (2n)}(f)})\leq \left(\frac{3}{4}\right)^n\Var (f) + 24||f||_{L^1(\overline{\leb})}.$$ 
\label{stp:even-iterates}
\end{stp}
\begin{proof}
    This follows immediately by iterating \eqn{eqn:lyineq} from \prop{prop:tent-map-ly-ineq} and recalling that $||\mathcal{L}_\omega^{\varepsilon\, (n)}f||_{L^1(\overline{\leb})}\leq ||f||_{L^1(\overline{\leb})}$. Indeed,
    \begin{align*}
        \Var({\mathcal{L}_\omega^{\varepsilon\, (4)}(f)})&\leq \frac{3}{4}\Var ({\mathcal{L}_\omega^{\varepsilon\, (2)}(f)}) + 6||f||_{L^1(\overline{\leb})} \\
        &\leq \frac{3}{4}\left(\frac{3}{4}\Var (f) + 6||f||_{L^1(\overline{\leb})} \right) + 6||f||_{L^1(\overline{\leb})}\\
        &=\left( \frac{3}{4}\right)^2 \Var(f) + 6\left(1+\frac{3}{4} \right)||f||_{L^1(\overline{\leb})}.
    \end{align*}
    By inductively continuing this procedure we find that
\begin{align*}
    \Var({\mathcal{L}_\omega^{\varepsilon\, (2n)}(f)})&\leq \left(\frac{3}{4}\right)^n\Var (f) + 6\sum_{k=0}^{n-1} \left(\frac{3}{4}\right)^k||f||_{L^1(\overline{\leb})}\\
    &=\left(\frac{3}{4}\right)^n\Var (f) + 24\left( 1- \left(\frac{3}{4}\right)^n\right)||f||_{L^1(\overline{\leb})}\\
    &\leq \left(\frac{3}{4}\right)^n\Var (f) + 24||f||_{L^1(\overline{\leb})}.
\end{align*} 
\end{proof}
\noindent
We now obtain a uniform Lasota-Yorke inequality for the odd iterates of $\mathcal{L}_\omega^\varepsilon$. This requires us to estimate $\Var({\mathcal{L}_\omega^\varepsilon(f)})$.
\begin{stp}
For any $f\in \BV (I)$
$$\Var({\mathcal{L}_\omega^\varepsilon(f)})\leq \Var (f) + 2||f||_{L^1(\overline{\leb})}$$    
\label{stp:first-iterate}    
\end{stp}
   \begin{proof}
       As verified in \cite{Horan}, $I$ and $T_\omega^\varepsilon$ satisfy ($\mathcal{M}$) (see \dfn{def:Mclass}) for each $\omega\in\Omega$. It suffices to compute the relevant quantities in the statement of \prop{prop:unif-LY}. Take $I_1 = [-1,-0.5], I_2 = [-0.5,0], I_3=[0,0.5], I_4=[0.5,1]$ and $J_m=I_m$. As demonstrated in \ex{ex:ptmhanging}, the collection of hanging points for the paired-tent map is 
    $$H=\{(-0.5,-), (-0.5,+) ,(0.5,-),(0.5,+)\}.$$
    Furthermore, $g_\omega^\varepsilon(x) = |(T_\omega^\varepsilon)^\prime(x)|^{-1}$. Each $J_m$ contains at most one element of $H$ and thus 
    \begin{align*}
        h_{\mathcal{J}}(1)&= g_\omega^\varepsilon((-0.5)^-)=\frac{1}{2(1+\varepsilon b_\omega)}=g_\omega^\varepsilon((-0.5)^+)=h_{\mathcal{J}}(2)\\
        h_{\mathcal{J}}(3)&= g_\omega^\varepsilon((0.5)^-)=  \frac{1}{2(1+\varepsilon a_\omega)}=g_\omega^\varepsilon((0.5)^+)= h_{\mathcal{J}}(4)
    \end{align*}
    suggesting that for each $m=1,2,3,4$ we have that $h_{\mathcal{J}}(m)\leq \frac{1}{2}$. Furthermore, recalling that $\overline{\leb}$ denotes the normalised Lebesgue measure on $I=[-1,1]$ and $I_m=J_m$, we have that for each $m$, $\overline{\leb}(I_m)=\overline{\leb}(J_m)=\frac{1}{2}\cdot\frac{1}{2}=\frac{1}{4}$. With this, we can compute the relevant quantities appearing in the statement of \prop{prop:unif-LY}. Indeed, since $g_\omega^\varepsilon$ is constant on each branch 
    $$\Var_{\accentset{\circ}{I}_n}(g_\omega^\varepsilon)=0.$$
    Therefore,
    \begin{align*}
        \sup_n\left\{ \norm{\restr{g_\omega^\varepsilon}{\accentset{\circ}{I}_n}}_\infty + \Var_{\accentset{\circ}{I}_n}(g_\omega^\varepsilon) \right\} + \max_m\{ h_{\mathcal{J}}(m) \} &\leq  \sup_n\left\{ \norm{\restr{g_\omega^\varepsilon}{\accentset{\circ}{I}_n}}_\infty \right\}+\frac{1}{2}\\
        &\leq \frac{1}{2}+\frac{1}{2}\\
        &=1.
    \end{align*}
    And,
    \begin{align*}
        \sup_n\left\{ \frac{\Var_{\accentset{\circ}{I}_n}(g_\omega^\varepsilon)}{\overline{\leb}(I_n)} \right\} + \max_m\left\{ \frac{h_{\mathcal{J}}(m)}{\overline{\leb}(J_m)} \right\} &\leq 0 + 2\\
        &=2.
    \end{align*}
    Putting these two together we obtain our result.
   \end{proof} 
   \noindent
  Combining \Step{stp:even-iterates} and \Step{stp:first-iterate} we can now control the variation for the odd iterates of $\mathcal{L}_\omega^\varepsilon$. This gives a uniform Lasota-Yorke inequality for every iterate of the Perron-Frobenius operator $\mathcal{L}_\omega^\varepsilon$ over both $\omega\in \Omega$ and $\varepsilon>0$.
  \begin{stp}
      For any $f\in \BV (I)$ and $n\in\mathbb{N}$
      $$\Var({\mathcal{L}_\omega^{\varepsilon\, (2n+1)}f})\leq \left(\frac{3}{4}\right)^n\Var (f) + 26||f||_{L^1(\overline{\leb})}.$$
      \label{stp:odd-iterates}
  \end{stp}
  \begin{proof}
    Indeed, for $n\in\mathbb{N}$, by \Step{stp:even-iterates} we have that 
    $$\Var({\mathcal{L}_\omega^{\varepsilon\, (2n+1)}(f)})\leq  \left(\frac{3}{4}\right)^n\Var (\mathcal{L}_\omega^\varepsilon f) + 24||\mathcal{L}_\omega^\varepsilon f||_{L^1(\overline{\leb})}.$$
    Thus, using \Step{stp:first-iterate} to control the variation term,
    \begin{align*}
    \Var({\mathcal{L}_\omega^{\varepsilon \, (2n+1)}(f)})&\leq  \left(\frac{3}{4}\right)^n\left( \Var(f) + 2||f||_{L^1(\overline{\leb})} \right) + 24||f||_{L^1(\overline{\leb})}  \\
    &\leq \left(\frac{3}{4}\right)^n\Var (f) +2 ||f||_{L^1(\overline{\leb})}+24||f||_{L^1(\overline{\leb})}\\
    &= \left(\frac{3}{4}\right)^n\Var (f) +26||f||_{L^1(\overline{\leb})}.
    \end{align*}
  \end{proof}
  \noindent
  Comparing the estimates from \Step{stp:even-iterates} and \Step{stp:odd-iterates} the result follows.
\end{proof}
\noindent
It remains to show that \hyperref[list:P6]{\textbf{(P6)}} holds for paired tent map cocycles.
\begin{lemma}
In the setting of \thrm{thrm:ptm-inv}, for all $\varepsilon>0$, $(T_\omega^\varepsilon)_{\omega\in\Omega}$ has a unique RACIM $(\mu_\omega^\varepsilon)_{\omega\in\Omega}$ with density $(\phi_\omega^\varepsilon)_{\omega\in\Omega}=(\frac{d\mu_\omega^\varepsilon}{\dleb})_{\omega\in\Omega}$ and hence \hyperref[list:P6]{\textbf{(P6)}} is satisfied. 
\label{lem:ptm_racim}
\begin{proof}
    This follows from \cite[Theorem 4.13]{Horan}. In particular, if $\mathcal{L}_\omega^\varepsilon$ denotes the Perron-Frobenius operator associated with the paired tent map $T_\omega^\varepsilon$, then for all $\varepsilon>0$, the top Oseledets space of $\mathcal{L}_\omega^\varepsilon$ is one-dimensional, and is spanned by $\phi_\omega^\varepsilon=d\mu_\omega^\varepsilon/\dleb$.  
\end{proof}
\end{lemma}
\begin{remark}
\lem{lem:ptm_racim} also follows from the main theorem in \cite{Unique-RACIM}, noting that for small $\varepsilon>0$, the map $T_\omega^\varepsilon$ satisfies the random covering property. That is, for each non-trivial subinterval $J\subset[-1,1]$, for $\mathbb{P}$-a.e. $\omega\in\Omega$, there exists $n_c:=n_c(\omega,J)<\infty$ such that 
$$\essinf_{x\in[-1,1]}(\mathcal{L}_\omega^{\varepsilon\, (n)}\mathds{1}_J)>0$$
for all $n\geq n_c$. Covering properties for cocycles of paired tent maps were also addressed in \cite{Horan}.
\end{remark}
\noindent
Since the initial system satisfies \hyperref[list:I1]{\textbf{(I1)}}-\hyperref[list:I6]{\textbf{(I6)}}, and \lem{lem:WLR}, \ref{lem:meas-ptm}, \ref{lem:ULY} and \lem{lem:ptm_racim} ensure that \hyperref[list:P3]{\textbf{(P3)}}-\hyperref[list:P7]{\textbf{(P7)}} hold, by noting that \hyperref[list:P2]{\textbf{(P2)}} is satisfied for $T^\varepsilon:\Omega\times I \to I$ and enforcing \hyperref[list:P1]{\textbf{(P1)}}, we have verified that all conditions outlined in \Sec{sec:map+pert} are satisfied for paired tent map cocycles. By \thrm{thrm:phi_lims} and \thrm{thrm:cont2nd}, \thrm{thrm:ptm-inv} and \thrm{thrm:ptm_cont2nd} follow.

\section*{Acknowledgments}
The authors acknowledge support from the Australian Research Council (DP220102216). JP acknowledges the Australian Government Research Training Program for financial support. The authors thank Jason Atnip for helpful conversations, and the anonymous referees for their valuable comments and suggestions. 

\footnotesize
\bibliographystyle{plain}
\bibliography{metastable}

\end{document}